\documentclass[12pt,reqno,a4paper]{amsart}

\usepackage[breaklinks,colorlinks,plainpages,hypertexnames=false,plainpages=false]{hyperref}
\hypersetup{urlcolor=blue, citecolor=blue, linkcolor=blue}

\usepackage[utf8]{inputenc}
\usepackage{amsmath}
\usepackage{amsthm}
\usepackage{amssymb}
\usepackage{amsfonts}
\usepackage{enumerate}
\usepackage{subcaption}
\usepackage{yfonts}
\usepackage{mathtools}
\usepackage{tikz-cd}
\usepackage{pgfplots}
\pgfplotsset{compat=newest}
\usetikzlibrary{decorations.markings}
\usepackage[colorinlistoftodos,prependcaption]{todonotes}
\usepackage{regexpatch}
\usepackage{hhline}
\usepackage{stmaryrd}
\usepackage{enumitem}
\usepackage{centernot}

\usepackage{mathrsfs}
\usepackage{comment}
\usepackage{listings}
\usepackage{multicol}
\usepackage[capitalise, noabbrev]{cleveref} %
\usepackage{algpseudocode}
\usepackage{algorithm}

\usepackage{mathdots}
\usepackage{cleveref} %
\DeclareMathAlphabet{\dutchcal}{U}{dutchcal}{m}{n}

\usepackage{array}
\newcolumntype{L}[1]{>{\raggedright\let\newline\\\arraybackslash\hspace{0pt}}m{#1}}
\newcolumntype{C}[1]{>{\centering\let\newline\\\arraybackslash\hspace{0pt}}m{#1}}
\newcolumntype{R}[1]{>{\raggedleft\let\newline\\\arraybackslash\hspace{0pt}}m{#1}}

\newcommand*\angles[1]{\langle #1 \rangle}

\usepackage{tikz}
\usetikzlibrary{arrows}
\usetikzlibrary{decorations.pathreplacing}
\usetikzlibrary{matrix}
\usetikzlibrary{calc}
\usetikzlibrary{shapes}
\usetikzlibrary{patterns}
\usetikzlibrary{fit,backgrounds,scopes}
\newtheoremstyle{theoremstyle}
{10pt}      %
{5pt}       %
{\itshape}  %
{}          %
{\bfseries} %
{}         %
{ }      %
{}          %

\newtheoremstyle{algorithmstyle}
{10pt}      %
{5pt}       %
{}  %
{}          %
{\bfseries} %
{}         %
{ }      %
{}          %

\newtheoremstyle{examplestyle}
{10pt}      %
{5pt}       %
{}          %
{}          %
{\bfseries} %
{}         %
{ }      %
{}          %

\newsavebox{\tempbox}

\newcommand{\textbox}[1]%
{\savebox{\tempbox}{#1}%
 \ifdim\wd\tempbox<7cm\relax
   \makebox[7cm]{\usebox{\tempbox}}%
 \else
   \parbox{7cm}{\raggedright #1}%
 \fi}

\makeatletter
\newtheorem*{rep@theorem}{\rep@title}
\newcommand{\newreptheorem}[2]{%
\newenvironment{rep#1}[1]{%
 \def\rep@title{#2 \ref{##1}}%
 \begin{rep@theorem}}%
 {\end{rep@theorem}}}
\makeatother

\makeatletter %
\newcommand{\subalign}[1]{%
  \vcenter{%
    \Let@ \restore@math@cr \default@tag
    \baselineskip\fontdimen10 \scriptfont\tw@
    \advance\baselineskip\fontdimen12 \scriptfont\tw@
    \lineskip\thr@@\fontdimen8 \scriptfont\thr@@
    \lineskiplimit\lineskip
    \ialign{\hfil$\m@th\scriptstyle##$&$\m@th\scriptstyle{}##$\hfil\crcr
      #1\crcr
    }%
  }%
}
\makeatother

\newreptheorem{theorem}{Theorem}
\newreptheorem{corollary}{Corollary}
\newreptheorem{proposition}{Proposition}

\makeatletter
\@namedef{subjclassname@2020}{%
  \textup{2020} Mathematics Subject Classification}
\makeatother

\makeatletter
\xpatchcmd{\@todo}{\setkeys{todonotes}{#1}}{\setkeys{todonotes}{inline,#1}}{}{}
\makeatother

\theoremstyle{theoremstyle}
\newtheorem{theorem}{Theorem}[section]
\newtheorem{lemma}[theorem]{Lemma}
\newtheorem{proposition}[theorem]{Proposition}

\theoremstyle{examplestyle}
\newtheorem{example}[theorem]{Example}
\newtheorem{definition}[theorem]{Definition}

\newtheorem*{notation*}{Notation}

\newtheorem{remark}[theorem]{Remark}

\newcommand{\CC}{\mathbb{C}}

\newcommand{\Q}{\mathbb{Q}}

\newcommand{\Z}{\mathbb{Z}}

\newcommand{\suchthat}{\;\ifnum\currentgrouptype=16 \middle\fi|\;}

\newcommand{\an}{{\mathrm{an}}}

\newcommand{\sep}{\text{sep}}

\newcommand{\bs}{\backslash}

\DeclareMathOperator{\Spec}{Spec}

\definecolor{ududff}{rgb}{0.30196078431372547,0.30196078431372547,1}
\definecolor{qqqqff}{rgb}{0,0,1}
\definecolor{qqccqq}{rgb}{0,0.8,0}
\definecolor{ffqqtt}{rgb}{1,0,0.2}
\definecolor{wwzzff}{rgb}{0.4,0.6,1}
\definecolor{ffxfqq}{rgb}{1,0.4980392156862745,0}
\definecolor{ffqqqq}{rgb}{1,0,0}
\definecolor{ududff}{rgb}{0.30196078431372547,0.30196078431372547,1}
\definecolor{zzttqq}{rgb}{0.6,0.2,0}
\definecolor{aureolin}{rgb}{0.99, 0.93, 0.0}

\definecolor{ffffff}{rgb}{1,1,1}%
\newcommand\restr[2]{{\left.\kern-\nulldelimiterspace #1 \right|_{#2}}}

\newcommand{\sh}{\mathrm{sh}}
\newcommand{\h}{\mathrm{h}}
\newcommand{\val}{\mathrm{val}}

\mathchardef\mhyphen="2D

\begin{document}

\title[An algorithm for power series and %
dual intersection graphs]{A symbolic algorithm for calculating \\
power series expansions and  \\
dual intersection graphs of semistable models}%

\author{Paul Alexander Helminck}
\address{Department of Mathematics, Tohoku University, Japan.}


\date{\today}


\begin{abstract}
In this paper we develop a symbolic algorithm to calculate multivariate power series expansions of univariate polynomials over general base rings. We use this to give a complete power series algorithm to calculate the dual intersection graph of a semistable model of a curve over a non-archimedean field. 

We first study the problem of recovering the relative poset structure of a finite covering $X'\to X$ of 
normal, relatively unibranch, Noetherian connected schemes. We show that we can reconstruct the poset structure of $X'$ in terms of group-theoretic data over the base $X$. This group-theoretic data consists of glued double cosets, and we show how these can be interpreted in terms of glued power series approximations. We then show how our algorithms calculate these glued power series approximations, so that we can work with the branches of normalizations $X'\to X$ without calculating integral closures. These algorithms have been implemented in OSCAR. %
We give a detailed study of the key steps in these algorithms for coverings of semistable models, with various examples to illustrate the non-trivial gluing phenomena. We conclude by interpreting these techniques in the context of analytic spaces, with an eye towards future applications in $p$-adic integration theory.

\end{abstract}

\maketitle

\vspace{-0.9cm}
\section{Introduction}

Algebraic varieties are often given as a finite branched covering of a simpler one.  
In this paper, we study the problem of symbolically calculating the branches of such a finite covering.
For instance, in the classical case of an algebraic curve $X$ over $\CC$, the branches of a covering $X'\to X$ can be represented by %
one-dimensional power series expansions 
\begin{equation*}
z(t)=\sum_{i=r}^{\infty}c_{i}t^{i/n},
\end{equation*}        
where $t$ is some local uniformizer on $X$, $n$ is the local ramification index, $r\in\mathbb{Z}$, and $c_{i}\in\mathbb{C}$. The corresponding algorithm that calculates these power series expansions is commonly known as the \emph{Newton-Puiseux algorithm}.   %

The main goal in this paper is to find a suitable analogue of the Newton-Puiseux algorithm for higher-dimensional schemes defined over general base rings. For instance, consider the polynomial 
\begin{equation}\label{eq:MainHypersurface}
f=z^3 + 2z^2 - zu^4v^2 - zv + z - v.
\end{equation}
We can view this as a univariate polynomial in $z$ over the base ring $A=\mathbb{Q}[u,v]$. Its zero set $V(f)$ defines a surface in $\mathbb{A}^{3}=\Spec(\mathbb{Q}[u,v,z])$, and the projection map onto $(u,v)$ %
defines a $3:1$ covering map $V(f)\to X=\mathbb{A}^{2}$. The corresponding hypersurface can be found in \cref{fig:Surface}. 

\begin{figure}[h]
\scalebox{0.15}{
\includegraphics{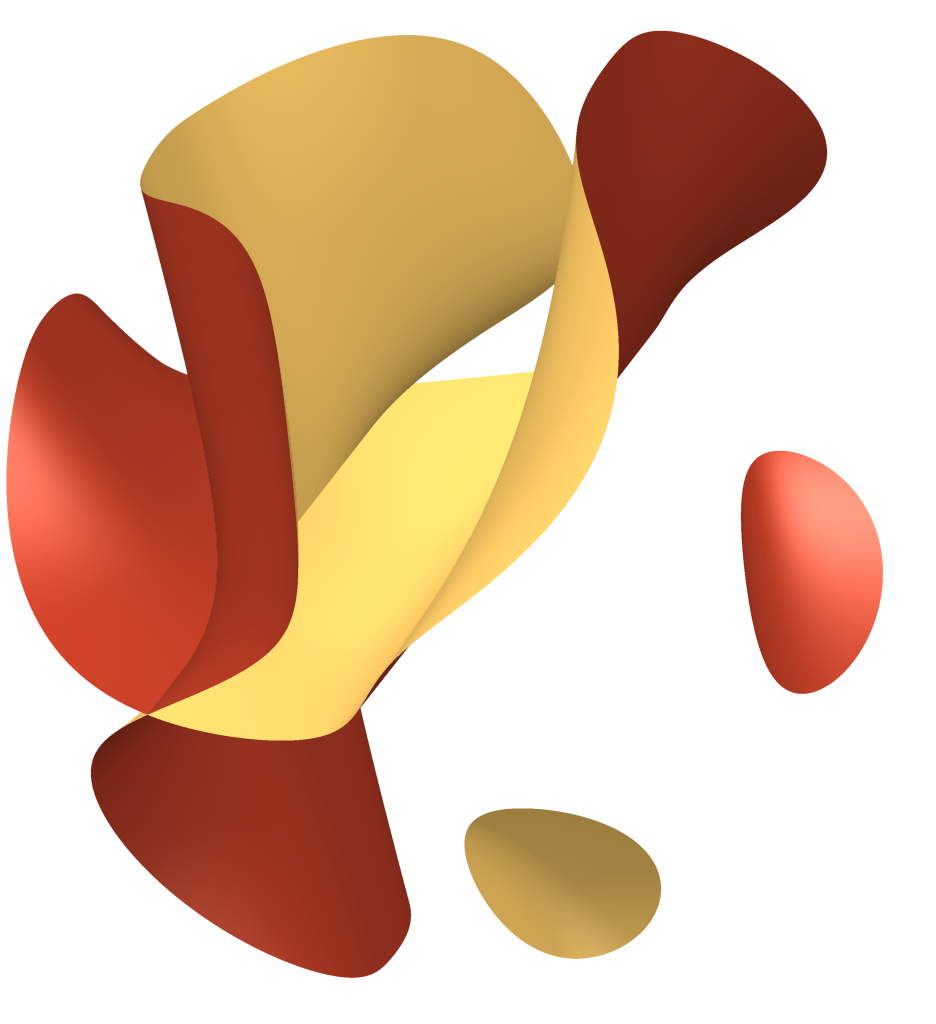}
}
\caption{\label{fig:Surface}The hypersurface defined by the polynomial $f(u,v,z)$ in \cref{eq:MainHypersurface}. The projection map $(u,v,z)\mapsto(u,v)$ induces a $3:1$ covering of the affine plane. }
\end{figure}
In this case, we can again find an infinite power series representation near $(0,0)$ for the three branches of this covering: %
\begin{equation*}
z=\sum_{(i,j)\in\mathbb{N}^{2}}c_{i,j}u^{i}v^{j}
\end{equation*}
Here the $c_{i,j}$ are defined over an algebraic closure of $\mathbb{Q}$. The generalized Newton-Puiseux algorithm in this paper allows us to compute this representation symbolically. See \cref{exa:PlaneQuartic} and the accompanying \texttt{Example4.12.jl} for the computations.  %

 \vspace{0.25cm}

\subsection{Motivation} 
Our motivation for finding these power series representations comes from number theory, as we explain now.    
Let $X$ be a smooth, geometrically connected projective curve over a local field $K$ of characteristic zero with valuation ring $\mathcal{O}_{K}$. By the semistable reduction theorem \cite{DM69}, $X$ admits a semistable model $\mathcal{X}/\mathcal{O}_{K}$ %
after a finite extension of $K$. The dual intersection graph $\Sigma(\mathcal{X})$ of the special fiber of this model gives a combinatorial representation of the way the irreducible components of $\mathcal{X}_{s}$ intersect. For a set of dual intersection graphs coming from curves of genus $3$, see \cref{fig:Covering}. The dual intersection graph of a semistable model contains various important arithmetic invariants (see \cite{DDMM2023} for the case of hyperelliptic curves), and for applications to Diophantine equations it is important to be able to explicitly calculate with these models $\mathcal{X}$ and their associated graphs $\Sigma(\mathcal{X})$.  Recent examples of these computations can be found in %
\cite{KK2022} and \cite{BRHS2024}, where they played important roles in carrying out the Chabauty-Coleman method at primes of bad reduction.  %

One of the main motivations for constructing generalized Newton-Puiseux algorithms for higher-dimensional schemes in this paper is to obtain a complete algorithm  
for calculating dual intersection graphs. 
To that end, we first choose a finite branched cover 
\begin{equation*}
\phi: X\to \mathbb{P}^{1}.
\end{equation*}
Algebraically, this map corresponds to a choice of a non-constant function $f\in K(X)$ by the equivalence between the category of finite function field extensions and the category of finite morphisms of normal curves. 
If $X$ is for instance given as the smooth completion of a plane curve $C=V(F(x,y))$, then one often takes $f$ to be $x$ or $y$. To deduce a semistable model from this covering, we can use the 
simultaneous semistable reduction theorem \cite[Theorem 4.5]{Liu2006}\cite[Theorem A]{ABBR2015}. This says that  
we can find a semistable model $\mathcal{Y}$ of $\mathbb{P}^{1}$ and a finite extension $K'$ of $K$ such that the normalized base change of $\mathcal{Y}$ in $K'(X)$ is semistable.
If the covering $\phi$ for instance is sufficiently tame, then it suffices to take a semistable model $\mathcal{Y}$ of the marked curve $(\mathbb{P}^{1},D)$, where $D$ is the branch locus, see \cref{pro:SimultaneousSemistable} for a precise statement. In general, finding this model $\mathcal{Y}$ and the corresponding extension $K\subset K'$ is quite difficult due to the presence of wild ramification, see \cite{Ossen24}, \cite{Ossen25}, \cite{AW12}, \cite{OW16}, \cite{CTT16} and \cite{BT20}. 

Let $\phi_{\mathcal{O}_{K}}:\mathcal{X}\to \mathcal{Y}$ be a morphism of semistable models arising from one of the procedures above. Write $S_{\mathcal{Y}}\subset \mathcal{Y}$ for the union of the set of generic points of the special fiber $\mathcal{Y}_{s}$ and the set of  intersection points of the corresponding components. We view these as points of $\mathcal{Y}$ through the closed immersion $\mathcal{Y}_{s}\to \mathcal{Y}$. To find the dual intersection graph $\Sigma(\mathcal{X})$ of $\mathcal{X}$, it then suffices to know the poset structure of the inverse image $\phi^{-1}_{\mathcal{O}_{K}}(S_{\mathcal{Y}})=:S_{\mathcal{X}}$. As a first approximation, we calculate the fibers over the generic points $\eta$ of $\mathcal{Y}_{s}$.
Since the corresponding local rings are discrete valuation rings, we can compute these fibers with an adaptation of the classical Newton-Puiseux algorithm. Next, we consider a pair $(\eta,z)$, where $\eta$ is as above and $z$ is an intersection point of $\mathcal{Y}_{s}$, corresponding to an edge of the dual intersection graph $\Sigma(\mathcal{Y})$. This pair gives rise to a rank-$2$ valuation on $K(Y)$. Using an adaptation of the MacLane algorithm \cite{MacLane36}\cite{Ruth2015}, we can then reconstruct extensions of this pair $(\eta,z)$\footnote{This has been implemented in SageMath in the form of the package MCLF.}.     

\vspace{0.1cm}

We now consider the following natural problem: %
\vspace{0.15cm}
\begin{center}
Is the extension data for rank-$2$ valuations sufficient to recover the graph $\Sigma(\mathcal{X})$?
\end{center}
\vspace{0.05cm}
In many cases, the answer is yes. For instance, for superelliptic curves this is true by \cite[Lemma 3.9]{H2022}. In general however, the answer is no. Namely, consider the three coverings of metric graphs in \cref{fig:Covering}. By the lifting results in \cite{ABBR2015}, all of these coverings are realizable. Moreover, their local structures are the same, so that they give the same rank-$2$ data. %
Note that one cannot circumvent this problem by subdividing the edge using a blow-up, as this gives multiple instances of the same problem. What is missing here %
is a way to glue the branches of different neighboring rank-$2$ valuations. One of the main goals of this paper is to solve these types of gluing problems, both from a theoretical and algorithmic point of view. 

Our main theoretical result is a group-theoretic reconstruction algorithm for the relative poset structure of a finite covering $\phi: X'\to X$ of normal connected Noetherian schemes. That is, given a subset $S\subset X$, we show how the poset $\phi^{-1}(S)$ can be expressed purely in terms of group theory on $X$. We further discuss the details of this result in \cref{sec:TheoreticalOverview}. 

To calculate with this group-theoretic data in practice, we develop a general symbolic algorithm to calculate multivariate power series expansions of univariate polynomials $f(z)$. This can be seen as a multivariate generalization of the classical Newton-Puiseux algorithm (see the proof of \cite[Theorem 2.1.5]{MS15}) that calculates the branches of a plane curve. We discuss the precise form of this algorithm in \cref{sec:AlgorithmicAspects}. We have implemented this algorithm in OSCAR.

  \begin{figure}[ht]
 \begin{minipage}{0.3\textwidth}
\scalebox{0.33}{
\begin{tikzpicture}[line cap=round,line join=round,>=triangle 45,x=1cm,y=1cm]
\clip(2,0) rectangle (13,11);
\draw [line width=1.5pt] (6,2)-- (4,4);
\draw [line width=1.5pt] (6,2)-- (3,1);
\draw [line width=1.5pt] (6,2)-- (9,2);
\draw [line width=1.5pt] (9,2)-- (12,2);
\draw [line width=1.5pt] (6,10)-- (9,8);
\draw [line width=1.5pt] (6,8)-- (9,8);
\draw [line width=1.5pt] (9,8)-- (12,8);
\draw [line width=1.5pt] (9,8)-- (12,6);
\draw [line width=1.5pt] (12,6)-- (9,6);
\draw [line width=1.5pt] (9,6)-- (6,6);
\draw [line width=1.5pt] (6,10)-- (3,9);
\draw [line width=1.5pt] (2.98,6.84)-- (6,8);
\draw [line width=1.5pt] (6,6)-- (2.98,6.84);
\draw [line width=1.5pt] (6,10)-- (4,10);
\draw [line width=1.5pt,loosely dashed] (6,8)-- (4,10);
\draw [line width=1.5pt,loosely dashed] (6,6)-- (4,8);
\draw (8.75,9.25) node[anchor=north west] {\Large{$1$}};
\draw [->,line width=1.5pt] (7.48,4.98) -- (7.48,3);
\begin{scriptsize}
\draw [fill=qqqqff] (6,2) circle (5.0pt); %
\draw [fill=qqccqq] (4,4) circle (5.0pt);
\draw [fill=ffqqtt] (3,1) circle (5.0pt);  %
\draw [fill=qqqqff] (9,2) circle (5.0pt);  %
\draw [fill=qqqqff] (12,2) circle (5.0pt); %
\draw [fill=qqqqff] (6,10) circle (5.0pt); %
\draw [fill=qqqqff] (9,8) circle (5.0pt); %
\draw [fill=qqqqff] (6,8) circle (5.0pt); %
\draw [fill=qqqqff] (12,8) circle (5.0pt); %
\draw [fill=qqqqff] (12,6) circle (5.0pt); %
\draw [fill=qqqqff] (9,6) circle (5.0pt);
\draw [fill=qqqqff] (6,6) circle (5.0pt); %
\draw [fill=ffqqtt] (3,9) circle (5.0pt);  %
\draw [fill=ffqqtt] (3,6.84) circle (5.0pt); %
\draw [fill=qqccqq] (4,10) circle (5.0pt); %
\draw [fill=qqccqq] (4,8) circle (5.0pt);
\end{scriptsize}
\end{tikzpicture}
}

\end{minipage}
\begin{minipage}{0.3\textwidth}
\scalebox{0.33}{
\begin{tikzpicture}[line cap=round,line join=round,>=triangle 45,x=1cm,y=1cm]
\clip(2,0) rectangle (13,11);
\draw [line width=1.5pt] (6,2)-- (4,4);
\draw [line width=1.5pt] (6,2)-- (3,1);
\draw [line width=1.5pt] (6,2)-- (9,2);
\draw [line width=1.5pt] (9,2)-- (12,2);
\draw [line width=1.5pt] (6,10)-- (9,8);
\draw [line width=1.5pt] (6,8)-- (9,8);
\draw [line width=1.5pt] (12,6)-- (9,6);
\draw [line width=1.5pt] (9,6)-- (6,6);
\draw [line width=1.5pt] (6,10)-- (3,9);
\draw [line width=1.5pt] (2.98,6.84)-- (6,8);
\draw [line width=1.5pt] (6,6)-- (2.98,6.84);
\draw [line width=1.5pt] (6,10)-- (4,10);
\draw [line width=1.5pt,loosely dashed] (6,8)-- (4,10);
\draw [line width=1.5pt,loosely dashed] (6,6)-- (4,8);
\draw (8.83077103751822,9.611308443048706) node[anchor=north west] {\Large{$1$}};
\draw [->,line width=1.5pt] (7.48,4.98) -- (7.48,3);
\draw [shift={(10.5,6.5)},line width=1.5pt]  plot[domain=0.7853981633974483:2.356194490192345,variable=\t]({1*2.121320343559643*cos(\t r)+0*2.121320343559643*sin(\t r)},{0*2.121320343559643*cos(\t r)+1*2.121320343559643*sin(\t r)});
\draw [shift={(10.5,9.5)},line width=1.5pt]  plot[domain=3.9269908169872414:5.497787143782138,variable=\t]({1*2.121320343559643*cos(\t r)+0*2.121320343559643*sin(\t r)},{0*2.121320343559643*cos(\t r)+1*2.121320343559643*sin(\t r)});
\begin{scriptsize}
\draw [fill=qqqqff] (6,2) circle (5.0pt); %
\draw [fill=qqccqq] (4,4) circle (5.0pt);
\draw [fill=ffqqtt] (3,1) circle (5.0pt);  %
\draw [fill=qqqqff] (9,2) circle (5.0pt);  %
\draw [fill=qqqqff] (12,2) circle (5.0pt); %
\draw [fill=qqqqff] (6,10) circle (5.0pt); %
\draw [fill=qqqqff] (9,8) circle (5.0pt); %
\draw [fill=qqqqff] (6,8) circle (5.0pt); %
\draw [fill=qqqqff] (12,8) circle (5.0pt); %
\draw [fill=qqqqff] (12,6) circle (5.0pt); %
\draw [fill=qqqqff] (9,6) circle (5.0pt);
\draw [fill=qqqqff] (6,6) circle (5.0pt); %
\draw [fill=ffqqtt] (3,9) circle (5.0pt);  %
\draw [fill=ffqqtt] (3,6.84) circle (5.0pt); %
\draw [fill=qqccqq] (4,10) circle (5.0pt); %
\draw [fill=qqccqq] (4,8) circle (5.0pt);
\end{scriptsize}
\end{tikzpicture}
}

\end{minipage}
\begin{minipage}{0.3\textwidth}
\scalebox{0.33}{
\begin{tikzpicture}[line cap=round,line join=round,>=triangle 45,x=1cm,y=1cm]
\clip(2,0) rectangle (13,11);
\draw [line width=1.5pt] (6,2)-- (4,4);
\draw [line width=1.5pt] (6,2)-- (3,1);
\draw [line width=1.5pt] (6,2)-- (9,2);
\draw [line width=1.5pt] (9,2)-- (12,2);
\draw [line width=1.5pt] (6,10)-- (9,8);
\draw [line width=1.5pt] (6,8)-- (9,8);
\draw [line width=1.5pt] (12,6)-- (9,6);
\draw [line width=1.5pt] (9,6)-- (6,6);
\draw [line width=1.5pt] (6,10)-- (3,9);
\draw [line width=1.5pt] (6,6)-- (2.98,6.84);
\draw [line width=1.5pt] (6,10)-- (4,10);
\draw [line width=1.5pt,loosely dashed] (6,6)-- (4,8);
\draw (8.83077103751822,9.611308443048706) node[anchor=north west] {\Large{$1$}};
\draw [->,line width=1.5pt] (7.48,4.98) -- (7.48,3);
\draw [shift={(10.5,6.5)},line width=1.5pt]  plot[domain=0.7853981633974483:2.356194490192345,variable=\t]({1*2.121320343559643*cos(\t r)+0*2.121320343559643*sin(\t r)},{0*2.121320343559643*cos(\t r)+1*2.121320343559643*sin(\t r)});
\draw [shift={(10.5,9.5)},line width=1.5pt]  plot[domain=3.9269908169872414:5.497787143782138,variable=\t]({1*2.121320343559643*cos(\t r)+0*2.121320343559643*sin(\t r)},{0*2.121320343559643*cos(\t r)+1*2.121320343559643*sin(\t r)});
\draw [line width=1.5pt] (2.98,6.84)-- (6,8);
\draw [line width=1.5pt] (4,8)-- (6,8);
\begin{scriptsize}
\draw [fill=qqqqff] (6,2) circle (5.0pt); %
\draw [fill=qqccqq] (4,4) circle (5.0pt);
\draw [fill=ffqqtt] (3,1) circle (5.0pt);  %
\draw [fill=qqqqff] (9,2) circle (5.0pt);  %
\draw [fill=qqqqff] (12,2) circle (5.0pt); %
\draw [fill=qqqqff] (6,10) circle (5.0pt); %
\draw [fill=qqqqff] (9,8) circle (5.0pt); %
\draw [fill=qqqqff] (6,8) circle (5.0pt); %
\draw [fill=qqqqff] (12,8) circle (5.0pt); %
\draw [fill=qqqqff] (12,6) circle (5.0pt); %
\draw [fill=qqqqff] (9,6) circle (5.0pt);
\draw [fill=qqqqff] (6,6) circle (5.0pt); %
\draw [fill=ffqqtt] (3,9) circle (5.0pt);  %
\draw [fill=ffqqtt] (3,6.84) circle (5.0pt); %
\draw [fill=qqccqq] (4,10) circle (5.0pt); %
\draw [fill=qqccqq] (4,8) circle (5.0pt);
\end{scriptsize}
\end{tikzpicture}
}
\end{minipage}
\caption{\label{fig:Covering}
Three coverings of dual intersection graphs coming from plane quartics.
Locally, these coverings are the same, but globally, they are not. See 
 \cref{exa:PlaneQuartic} for an explicit plane quartic over $\Q_{p}$ giving rise to the covering on the left. The colors on the left of each covering illustrate the different branches near the trivalent vertex.  }%
\end{figure}

 \subsection{Theoretical overview of the paper}\label{sec:TheoreticalOverview}

We start the paper in \cref{sec:ReconstructionSection} by laying the necessary theoretical foundations for our algorithms. We give a short overview of the ideas here. %
Fix a normal connected Noetherian base scheme $X$ and a finite subset $S\subset X$. We will assume that $X$ is sufficiently non-singular over $S$: for every pair of elements $x,y\in{S}$ with $x\in V(y)=\overline{\{y\}}$, we require $V(y)$ to be unibranch at $x$ \cite[\href{https://stacks.math.columbia.edu/tag/0BPZ}{Definition 0BPZ}]{stacks-project}. In this case, we say that $X$ is \emph{relatively unibranch} with respect to $S$. For semistable models, this requirement boils down to the model being \emph{strongly semistable}, which can always be achieved after a blow-up and a field extension. %

Let $\phi: X'\to X$ be a finite separable dominant morphism of connected Noetherian schemes of degree $n=[K(X'):K(X)]$, where $K(X)$ and $K(X')$ are the function fields of $X$ and $X'$ respectively. We view $X'$ as generically consisting of $n$ branches over $X$ that possibly come together at various points of $X$. We would now like to glue these branches and reconstruct $X'$ on $X$. We will first explain the local picture in \cref{sec:LocalPart}, and then move on to the global picture in \cref{sec:GlobalPart}. To describe our reconstruction algorithm, we will make use of the following notation: %
\begin{itemize}
\item $L^{\sep}=K(X)^{\sep}$ is a fixed separable closure of $L=K(X)$. 
\item $G=\text{Gal}(L^{\sep}/L)$ is the associated absolute Galois group. 
\item $X^{\sep}\to X$ is the normalization of $X$ in $K(X)^{\sep}$. 
\item $x^{\sep}\in X^{\sep}$ is a point lying over $x$. 
\item $D_{x}:=D_{x^{\sep}/x}=\{\sigma\in{G}:\sigma(x^{\sep})=x^{\sep}\}$. 
\item $H\subset G$ is the subgroup corresponding to an embedding $L'=K(X')\to L^{\sep}$. 
\item $f(z)\in K(X)[z]$ is an irreducible polynomial such that $L[z]/(f(z))\simeq K(X')=L'$. 
\item The $\alpha_{1},...,\alpha_{n}$ are the roots of $f(z)$ in $K(X)^{\sep}$. 
\end{itemize}
\vspace{0.1cm}
Throughout this section we fix an isomorphism $L[z]/(f(z))\to L'\subset L^{\sep}$ sending $z$ to $\alpha_{1}=:\alpha$. We will also write $X'=X_{H}$ and $\phi_{H}:X_{H}\to X$ to emphasize the role of the subgroup. %
We will refer to an element of the fiber $\phi^{-1}_{H}(x)$ as a \emph{branch} of $\phi_{H}: X_{H}\to X$ over $x$. \subsubsection{The local part}\label{sec:LocalPart}
 The main idea at a single point $x\in{S}$ is summarized in the following diagram:
\begin{figure}[h]
\begin{displaymath}
\begin{tikzcd}
\boxed{1}:& \boxed{\text{Local branches of }X'\to X \text{ at }x}\arrow[d,Leftrightarrow]
\\
\boxed{2}:& \boxed{\text{Elements of the double coset space } D_{x}\backslash G/H} \arrow[d,Leftrightarrow]\\
\boxed{3}:& \boxed{D_{x}\mhyphen\text{orbits of the roots }\alpha_{i} \text{ of }f(z)} \arrow[d,Leftrightarrow]\\
\boxed{4}:&\boxed{D_{x}\mhyphen\text{orbits of sufficiently precise approximations }\gamma_{i} \text{ of the roots }\alpha_{i}.}
\end{tikzcd}
\end{displaymath}

\caption{\label{fig:Correspondences} The correspondences between branches, double cosets and power series that we will use throughout this paper. }
\end{figure}

\vspace{0.18cm}

The correspondences in \cref{fig:Correspondences} can be described as follows: %

\begin{itemize}
\item[] $\boxed{2}\Rightarrow \boxed{1}$ There is a bijection $D_{x}\backslash G/H\to H\backslash G/D_{x}$ sending $D_{x}\sigma{H}\mapsto H\sigma^{-1}D_{x}$. We then send $H\sigma^{-1}D_{x}$ to %
image of $\sigma^{-1}(x^{\sep})$ in $X'=X_{H}$ under the map $X^{\sep}\to X_{H}$ arising from the fixed embedding $K(X')=K(X_{H})\to K(X)^{\sep}$. 
\vspace{0.2cm}
\item[] $\boxed{2}\Rightarrow \boxed{3}$ A double coset $D_{x}\sigma H$ is sent to the $D_{x}$-orbit of $\sigma(\alpha)$. Note that $\alpha$ is the image of $z\in K(X)[z]/(f(z))$ under the fixed embedding $K(X)[z]/(f(z))\to K(X)^{\sep}$.   %
\vspace{0.2cm}
\item[] $\boxed{3}\Rightarrow \boxed{4}$ We can assume that the roots $\alpha_{i}$ are integral. After truncating the roots up to a separating height (see Sections \ref{sec:FromTopoiToPowerSeries} and  \ref{sec:PowerSeries}), we then obtain the desired approximations $\gamma_{i}$. 
\end{itemize}
The fact that these give bijections essentially follows from commutative algebra, the orbit-stabilizer theorem and Galois theory. We emphasize here that the \emph{objects of interest} to us are the branches in $\boxed{1}$, and the \emph{objects we can effectively compute with} are the power series expansions in $\boxed{4}$. %
In the second part of the paper, we will show how to calculate these power series expansions, so that we can work with the branches of normalizations $X'\to X$ without computing integral closures.   

\subsubsection{The global part}\label{sec:GlobalPart}

To recover the global poset structure of $\phi_{H}:X_{H}\to X$, we would like to glue instances of the data in \cref{fig:Correspondences} for various points $x\in{S}$. To that end, let $x,y\in{S}$ as before, and suppose that $x$ is a specialization of $y$, so that $x\in V(y)$.  %
If we view $S$ as a graph through its poset structure, then $e=xy$ is an edge in this graph. 
\begin{example}\label{exa:Poset1}
If $X=\Spec(A)$ for $A=K[x_{1},x_{2}]$ and $K$ a field, then we can consider the set $S=\{(x_{1},x_{2}), (x_{1}),(x_{2})\}\subset X$. The graph induced by $S$ consists of two line segments, corresponding to the inclusions $(x_{1},x_{2})\supset (x_{1})$ and $(x_{1},x_{2})\supset (x_{2})$.  

\end{example}  %
If we fix a lifting $e^{\sep}=x^{\sep}y^{\sep}$ of $e$ to $X^{\sep}$ so that $x^{\sep}\in V(y^{\sep})$, then we can introduce the modified decomposition group 
\begin{equation*}
D_{e}=D_{x}\cap D_{y}.
\end{equation*}    
We can then consider the $D_{e}$-orbits of the objects in \cref{fig:Correspondences}. That is, we consider $\boxed{2}$ the $D_{e}$-orbits of $G/H$, $\boxed{3}$ the $D_{e}$-orbits of the roots of $f(z)$, $\boxed{4}$ the $D_{e}$-orbits of compatible $x$-adic and $y$-adic approximations of the roots of $f(z)$\footnote{The exact meaning of compatible approximations will be given in \cref{def:CompatibleSeparatingSets}}.   %
The $D_{e}$-orbits can automatically be compared to their $D_{x}$ and $D_{y}$-orbits, since we have inclusions $D_{e}\subset D_{x},D_{y}$. We show that the diagram of orbits that follows from this allows us to reconstruct the different edges over $e$ for any covering $X_{H}\to X$, see \cref{pro:RecoveringTrees} for the exact statement. More generally, fix a \emph{spanning tree} of the Hasse diagram of $S$, together with a lift of this tree to $X^{\sep}$. We can then reconstruct the poset structure over this spanning tree using the idea above. For every edge outside the spanning tree, we fix a different lifting together with a suitable comparison map. By combining these, we obtain a poset $|S_{H}|$ and a bijection $|S_{H}|\to \phi^{-1}_{H}(S)$. In other words, we can globalize the construction from \cref{sec:LocalPart} and reconstruct the poset structure of $X_{H}$ group-theoretically on $X$. We summarize this in the following theorem.   %
\begin{reptheorem}{thm:MainThmv2}
Let $X$ be a normal connected Noetherian scheme and let $S\subset X$ be a finite subset such that $X$ is relatively unibranch over $S$. 
For every closed subgroup $H$ of $G$ with map of normalizations  %
$\phi_{H}:X_{H}\to X$, there is an isomorphism $|S_{H}|\to \phi_{H}^{-1}(S)$ of partially ordered sets. This association is functorial in $H$. 
\end{reptheorem}

          \begin{remark}
         The reason we need the condition of being relatively unibranch is that one cannot easily classify extensions of poset structures in terms of Galois theory in this case, see \cref{exa:NotUnibranch}. That is, the action of the Galois group on extensions of edges (let alone larger substructures, see \cref{exa:CounterexampleExtensions}) is in general not transitive. It was pointed out to the author by Akio Tamagawa that this situation can be remedied by going to the Zariski-Riemann space of $X$. There, one does have a transitivity result for extensions by \cite[Chapter VI, Paragraph 7, Corollary 3]{ZS76}, which one can then in principle use to generalize the results of this paper. An incarnation of the fact that this problem goes away in this case can also be found in \cite{Helminck2024}, where the Berkovich spaces associated to modular curves are studied. There one also does not need the additional assumption of being relatively unibranch.        
          \end{remark}
          
          \begin{remark}\label{rem:Monodromy}
          We note here that there are two important factors that can cause monodromy in reconstructing the poset structure of $X_{H}$. First, if we change the conjugacy class of an individual subgroup $D_{x}$, then the fiber $\phi^{-1}_{H}(x)$ does not change. %
          However, the global structure of the poset can change, and this is exactly what causes the different possibilities in \cref{fig:Covering}, see \cref{exa:PlaneQuartic} for an explicit example with equations. We call this \emph{algebraic monodromy}. Secondly, if a poset has non-trivial topology, then this can give rise to monodromy as in classical algebraic topology. An explicit example of this can be found in \cref{exa:ExplicitTransferMap}. We call this \emph{topological monodromy}. For our reconstruction algorithm of dual intersection graphs of semistable models using coverings to $\mathbb{P}^{1}$, the latter type of monodromy will not play a role, since the base posets are trees.     %
          \end{remark}

\subsection{The algorithmic aspects}\label{sec:AlgorithmicAspects}

With the correspondences from \cref{fig:Correspondences} in place, we now see that glued power series expansions, together with the induced Galois actions, theoretically allow us to reconstruct the poset structure of a covering of schemes. In other words, glued power series expansions form effective proxies for the branches of the covering $X'\to X$. %
To calculate these power series in practice, we introduce iterated symbolic multivariate Newton-Puiseux algorithms, which generalize the original Newton-Puiseux algorithms found in \cite{Puiseux1850} and \cite{Duval1989}. %

To explain the algorithm, let $A=K[x_{1},...,x_{n}]$ and let $f(z)\in A[z]$. Our main assumption on $f(z)$ is now as follows:
\begin{center}
\underline{Main assumption}\\
\vspace{0.1cm}
\emph{We assume that $f(z)$ splits completely over the \emph{strict Henselization} of $A$ with respect to $\mathfrak{m}=(x_{1},...,x_{n})$.}
\end{center}
We explain the philosphy behind working with the strict Henselization here. In ordinary algebraic geometry, one often finds the need to invert certain functions in $A$ at a point $\mathfrak{m}\in\Spec(A)$. That is, one wishes to work with solutions of equations of the form  
\begin{equation*}
w\cdot g=1,
\end{equation*}
where $g\in{A}$ satisfies $g(\mathfrak{m})\neq{0}$, or equivalently $g\notin\mathfrak{m}$. This can be done by moving to the \emph{localization} $A_{\mathfrak{m}}$ of $A$ with respect to $\mathfrak{m}$. Note that computationally, one can simply work in the finite type algebra %
$A[1/g]\simeq A[w]/(w\cdot g-1)$ whenever one needs to invert $g$. The strictly Henselian case is completely analogous, but one works with more general equations. Namely, one works with solutions of equations of the form 
\begin{equation*}
h=\sum_{i=0}^{r}c_{i}w^{i}=0,
\end{equation*}
where $h$ now satisfies the \emph{implicit function theorem} from calculus. That is, the Jacobian matrix of $h$ has full rank, or in other words: %
the discriminant $\Delta(h)$ satisfies $\Delta(h)(\mathfrak{m})\neq{0}\Leftrightarrow \Delta(h)\notin\mathfrak{m}$. The ring that we obtain in this way is the \'{e}tale localization or \emph{strict Henselization} $A^{sh}_{\mathfrak{m}}$ of $A$ at $\mathfrak{m}$. As with the localization, one can work with finite type algebras of the form $A[w]/(h)$ here instead of the full strict Henselization.     
\begin{remark}
In the literature, one often works with the completion of $A$ at $\mathfrak{m}$ rather than the strict Henselization. This completion however does not allow one to move between different primes $\mathfrak{m}$ and $\mathfrak{p}$ for $\mathfrak{m}\supset \mathfrak{p}$, see \cref{rem:PowerSeriesVSCompletions}.  %
For strict Henselizations, we however do obtain natural morphisms $A^{\sh}_{\mathfrak{m}}\to A^{\sh}_{\mathfrak{p}}$. Note that this is similar to the classical case where we have localization homomorphisms $A_{\mathfrak{m}}\to A_{\mathfrak{p}}$.   
\end{remark}

We now return to our main assumption. Even if a polynomial $f(z)$ does not split over the strict Henselization of our given algebra, we can often modify it so that it does split. %
We will see how to do this %
for coverings of strictly semistable models in \cref{lem:KummerExtensions}. Here we simply allow a suitable Kummer extension of $A$ with respect to the coordinate axes, which again puts us in the scenario mentioned at the beginning of this section. We explain the connection between this modification and the classical Newton-Puiseux theorem in the following remark.      

\begin{remark}
Let $K=\mathbb{C}((t))$ be the field of Laurent power series over $\CC$, and let $f(z)\in K[z]$ be a non-constant polynomial. The Newton-Puiseux theorem (see \cite[Theorem 2.1.5]{MS15}) then tells us that $f(z)$ splits over the finite extension %
$K(t^{1/n})$ for some $n$. %
We can calculate this $n$ in practice %
explicitly using Newton polygons. 

In terms of strict Henselizations, we can interpret this term as follows. Suppose that we started with a polynomial $f(z)$ defined over the ordinary polynomial ring $A=\mathbb{Q}[t]$. A modified version of the Newton-Puiseux theorem then says that $f(z)$ splits completely over the field of fractions of the strict Henselization of $A_{n}=A[t^{1/n}]$ with respect to the ideal $(t^{1/n})$, for $n$ large enough. This essentially follows from the discrete version of Abhyankar's lemma, see \cite[\href{https://stacks.math.columbia.edu/tag/0BRM}{Lemma 0BRM}]{stacks-project}. %
All in all, working with the strict Henselization in this case can be seen as a \emph{symbolic} approach to the classical Newton-Puiseux theorem. %

We now discuss the analogue of these rings and their extensions for semistable models. %
There, one has coordinate rings of the form $A=R[u,v]/(uv-\pi^{m})$, where $\pi$ is a uniformizer of a valued field $K$. In other words, the local equations are of the form $uv=\pi^{m}$. %
One can then similarly construct a \emph{Newton-Puiseux}-type algebra by taking $n$-th roots of the different standard coordinate functions $u$ and $v$ for a sufficiently large $n$. Let $f(z)\in{A[z]}$ again be a non-constant polynomial. Under certain technical assumptions on the branch locus of the covering induced by a polynomial $f(z)$ (which are always met in our scenario), we then again find that $f(z)$ splits completely over the strict Henselization of this algebra, see \cref{lem:KummerExtensions}.    %
In general, this will work whenever our polynomial $f(z)$ satisfies the criteria of Abhyankar's lemma, see \cite[Proposition 5.2]{SGA1}. Put differently, this means that $f(z)$ splits in a \emph{log-geometric \'{e}tale localization}, see \cite{Kato89} and \cite{I2002}. We will not dwell on this matter too much here, but we note that these rings provide the natural \emph{generalized Newton-Puiseux rings} in which our computations take place. In practice however, working in these rings simply means that we take roots of the coordinate functions whenever a certain non-integral slope is encountered in the Newton polygon of $f(z)$. %
\end{remark} %
 Let us return to our algebra $A=K[x_{1},...,x_{n}]$ with maximal ideal $\mathfrak{m}=(x_{1},...,x_{n})$. Let %
 $f(z)$ be a polynomial that splits completely over the strict Henselization of $A$ at $\mathfrak{m}$. %
 We assume without loss of generality that $f(z)$ is monic. 
For every root $\alpha_{i}$ of $f(z)$, we can then write 
\begin{equation*}
\alpha_{i}=\sum_{m\in\mathbb{N}^{n}}c_{m}x^{m},
\end{equation*}
where $c_{m}\in K^{\sep}$, $x^{m}=x_{1}^{i_{1}}\cdot x_{2}^{i_{2}}\cdots x_{n}^{i_{n}}$ and $m=(i_{1},...,i_{n})$. We call this the $(x_{1},...,x_{n})$-adic power series expansions of the roots $\alpha_{i}$ of $f$. Our generalized Newton-Puiseux algorithm calculates the coefficients $c_{m}$ of the roots of $\alpha_{i}$. %
This algorithm forms the base case for various other generalizations, where $A$ is a more arithmetically flavored ring such as $\mathcal{O}_{K,\mathfrak{p}}[x_{1},...,x_{n}]$ and $\mathcal{O}_{K,\mathfrak{p}}$ is as before the localization of a number ring. 

The main ideas behind the discrete Newton-Puiseux algorithm %
are as follows:

\begin{center}
\boxed{\text{Core algorithm behind an individual Newton-Puiseux iteration}}
\end{center}
\begin{enumerate}
\item We calculate the reduction modulo $(x_{n})$ of $f(z)$. We denote this by $\overline{f}$. 
\item If $\overline{f}$ has an irreducible factor of degree larger than one, then we %
add a root $\gamma$ of this equation to our algebra $A$ by setting $A:=A[w]/(g(w))$. Here $g$ is an irreducible polynomial with root $\gamma$\footnote{This new algebra $A$ will be a domain again by the normality of the original $A$, since strict Henselizations of normal domains are normal domains. }. 
\item We then set $f_{1}=f(z+\gamma)$ and calculate the Newton polygon of $f$ with respect to $x_{n}$\footnote{In order for the algorithm to have better integrality properties, we will slightly modify this transformation, see \cref{sec:EtaleRoutine} and \cref{alg:EtaleRoutine}.}. 
\item For every segment of this polygon with slope $r\in\mathbb{N}$, we calculate $f_{2}=1/x_{n}^{c}\cdot f_{1}(x_{n}^{r}z)$, where $c$ is the $x_{n}$-content of $f_{1}(x_{n}^{r}z)$. 
\item Return to the first step. 
\end{enumerate}
Note that this can be seen as a generalized version of the ordinary Newton-Puiseux algorithm. %
In the new version, we calculate the $(x_{n})$-adic power series expansions of the roots of $f(z)$. That is, if we write $\alpha_{i}=\sum_{j=1}^{\infty}d_{i,j}x_{n}^{j}$, then it calculates the $d_{i,j}$. The $d_{i,j}$ turn out to be defined over the strict Henselization of $B=K[x_{1},...,x_{n-1}]$ with respect to $\mathfrak{p}=(x_{1},...,x_{n-1})$, see \cref{lem:IntegralityPowerSeries}. We can now reapply the algorithm above with $B$ and $x_{n-1}$ rather than $A$ and $x_{n}$. As before, this gives the $x_{n-1}$-adic expansions of the $d_{i,j}$, and the coefficients of this expansion are defined over the strict Henselization of $K[x_{1},...,x_{n-2}]$ with respect to $(x_{1},...,x_{n-2})$. By continuing in this way, we eventually end up with a polynomial in $K[x_{1},x_{2}]$, at which point the algorithm reduces to the classical Newton-Puiseux algorithm. By collecting the data from all these computations, we thus see that we are calculating a coefficient $c_{m}$ for every monomial $x^{m}$, where $m\in\mathbb{N}^{n}$.   %
We show in \cref{thm:MainThmNumber2} that this %
algorithm correctly computes the $\mathfrak{m}$-adic power series expansions of the roots of $f$ for $\mathfrak{m}=(x_{1},...,x_{n})$.    
\begin{reptheorem}{thm:MainThmNumber2}
Let $f(z)$ be a non-constant polynomial over $A=K[x_{1},...,x_{n}]$ that splits completely over the strict Henselization of $A$ with respect to $\mathfrak{m}=(x_{1},...,x_{n})$. 
By iteratively applying \cref{alg:DiscreteNPAlgorithm} to the minimal polynomials of the $x_{k}$-adic coefficients for $1\leq{k}\leq{n}$, we obtain the %
$\mathfrak{m}$-adic power series expansions of the roots $\alpha_{i}$ of $f(z)$ up to any height. 
\end{reptheorem}

We emphasize one important fact about this algorithm here. We say that an extension of algebras $A\to B$ is monogenic if $B\simeq A[z]/(f(z))$ for some univariate polynomial. In general, the extensions of algebras coming from the normalization morphism $X'\to X$ associated to a finite extension of function fields $K(X)\to K(X')$ is not monogenic, not even locally. This is for instance one of the main reasons that calculating with rings of integers of number fields is very costly in practice. One of the benefits of our algorithm is that it only relies on monogenic data, and it does not require the calculation of the full normalization of a scheme in a finite extension. %
\begin{remark}
We note that for our algorithms regarding dual intersection graphs, we need the intermediate power series data from \cref{thm:MainThmNumber2} and \cref{alg:DiscreteNPAlgorithm} as well, since these show how the intermediate primes split. 
\end{remark}

We illustrate the algorithm arising from \cref{thm:MainThmNumber2} in \cref{sec:CoveringsSemistableModels} by giving a full algorithm to calculate the map of dual intersection graphs associated to a sufficiently tame morphism of semistable models of curves over a discretely valued field $K$%
\footnote{As mentioned before, the techniques in \cite{Ruth2015} and MCLF using the Maclane-algorithm seem to only fully calculate the map of dual intersection graphs in cases where the covering is not too split. For instance, for the curve in \cref{exa:MainExample2}, the output of MCLF is enough to reconstruct the graph, but for the one in \cref{exa:PlaneQuartic}, it is not. }. This algorithm does not work with the equations of a model of the curve whose dual intersection graph we wish to find, but rather implicitly works with {{proxies}} that allow us to bypass these computationally expensive normalizations. To illustrate this, consider a plane curve $C$ given by an equation $f(x,y)=0$ with smooth completion $X$. We take the map $X\to \mathbb{P}^{1}$ induced by projection onto the $x$-axis. Correspondingly, we view %
$f(x,y)$ as a univariate polynomial in $y$ over the function field $K(x)$. %
The proxies mentioned above are then glued power series expansions for the roots $y_{1},...,y_{n}$ of this polynomial over the points of a semistable model $\mathcal{Y}$ of $\mathbb{P}^{1}$. Here $n$ is the $y$-degree of $f(x,y)$. The gluing maps between these locally defined power series expansions can be non-trivial, as  \cref{exa:PlaneQuartic} shows.    %
We note here that in the case where $f(x,y)$ defines a hyperelliptic or superelliptic curve, our algorithm becomes very similar to the ones described in \cite{BH2020}, \cite{H2022}, \cite{DDMM2023} and \cite{Dokchitser2021}. %

In \cref{sec:Examples}, 
we show in several examples %
how our algorithm works in practice,  %
and we give a sketch in \cref{sec:InterpretationAnalyticSpaces} of how it can be used to calculate with local equations of annuli in the sense of rigid analytic or Berkovich spaces, so that these techniques can be used in 
the future to carry out the Chabauty-Coleman method at places of bad reduction. We invite the reader to compare this with the material \cite{KK2022} and \cite{Kaya2022}, where the case of hyperelliptic curves is treated.

\begin{remark}
Before starting the paper, we would like to make one closing remark. Coverings $X'\to{X}$ of the type in this paper can in principle always be represented in terms of univariate polynomials $f(z)\in{K(X)}$ by the primitive element theorem. Our algorithms are stated in terms of these since many coverings in practice are given in this form. %
In certain circumstances however, coverings arise more naturally as the zero set of multiple \emph{multivariate} polynomials. For instance, one can consider parametrized polynomial systems of the form
\begin{equation*}
  f=a_{1} x^3 +a_{2}x+ a_{3} y^2 \quad\text{and}\quad b_{1} x^3+b_{2} y^2+b_{3}=g
\end{equation*}
over $X=\Spec(A)$, where $A=K[a_{1},a_{2},a_{3},b_{1},b_{2},b_{3}]$ and $K$ is a field. This is known as a \emph{Bernstein generic family}, see \cite{Bernstein1975}. The scheme $X_{0}=A[x,y]/(f,g)$ is then generically finite of degree $6$ over $X$, and this number is the mixed volume of the Newton polytopes of $f$ and $g$. This is the content of the \emph{BKK-Theorem}, see \cite[Theorem A]{Bernstein1975}. By taking normalizations, we then find a natural finite covering $X'\to X$, which is the same as the natural covering $X_{0}\to X$ on a Zariski dense open subset. %
The correspondences in \cref{fig:Correspondences} now still work, but one has to work with $D_{x}$-orbits of power series expansions of both $x$ \emph{and} $y$. Here we view $x$ and $y$ as elements of the function field $K(X_{0})$.  

For these types of systems, one has the same type of monodromy as described in \cref{rem:Monodromy}. We now compare this to the monodromy found in the method of \emph{homotopy continuation}, see \cite{DHJLLS19} and \cite{BBCHLS23}. There, one starts with a one-parameter solution $(x_{i}(t))$ of a system of polynomial equations as above, which can be viewed as a set of elements $x_{i}(t)\in \mathbb{C}((t))(t^{1/n})=K$ for some $n\geq{1}$. We will also view these as lying in a convergent power series subring of $K$. In terms of the set-up above, this gives a $K$-valued point of $X$, corresponding to a prime $\mathfrak{p}\subset A$. One now is interested in propagating the solutions from $t=0$, to $t=1$. We view these two points as prime ideals $\mathfrak{m}_{0}=(t)$ and $\mathfrak{m}_{1}=(t-1)$ that contain $\mathfrak{p}$ (we assume here that these two limit points also lie in $X$). Note that the corresponding poset is exactly as in \cref{exa:Poset1}. What happens quite often in homotopy continuation is that the branches switch, see \cite[Figure 1]{BBCHLS23} for an intuitive illustration. We view the phenomenon in \cref{fig:Covering} as an analogue of this. In a sense, the methods in this paper can be seen as a symbolic incarnation of homotopy continuation over general schemes. %

\end{remark}

     \subsection{Notation and terminology}\label{sec:Notation}
Fields will be denoted by $K$ or $L$. An algebraic closure of a field $K$ is denoted by $\overline{K}$ and a separable closure is denoted by $K^{\sep}$. The field of $p$-adic numbers will be denoted by $\mathbb{Q}_{p}$, and the field of Puiseux series over a field $K$ will be denoted by $K\{\{t\}\}$. Let $A$ be a unique factorization domain (UFD) and let $f(z)\in{A[z]}$. The content of $f$ is the greatest common divisor of its coefficients. This is well defined up to a unit. We denote it by $\mathrm{cont}(f)$. If $\mathrm{cont}(f)$ is a unit, then we say that $f$ is \emph{content-free}. Suppose that $g\in {A}$ is irreducible. The $g$-content of $f(z)$ is $g^{r}$, where $r$ is the minimum of the valuations of the coefficients of $f(z)$. We denote it by $\mathrm{cont}_{g}(f)$. Suppose that $f(z)\in{A[z]}$, where $A$ is a domain with field of fractions $K(A)$. We say that $f(z)$ splits completely over $K(A)$ if $f(z)=u\prod_{i=1}^{r} (z-\alpha_{i})$ for $u,\alpha_{i}\in{K(A)}$. The $\alpha_{i}$ are the \emph{roots} of $f(z)$. We say that $f(z)$ splits completely over $A$ if $\alpha_{i}\in{A}$ for all $i=1,...,r$.        %

If $G$ is a group acting on a set $S$, then we write $G\bs S$ for the corresponding orbit space, which we view as a subset of the power set $\mathcal{P}(S)$ of $S$.  For $s\in{S}$, we denote its $G$-orbit by $G(s)$. For a subgroup $H\subset{G}$, we use the notation $G/H$ for the left coset spaces of $H$ in $G$. For two subgroups $H$ and $N$ of a group $G$ and $g$ an element of $G$, we write 
$$HgN=\{g'\in{G}:\,\exists h\in{H},n\in{N} \text{ such that }g'=hgn\}$$
 for the corresponding double coset. The space of double cosets is denoted by $H\bs G/N$. It can be identified with the orbit space $H\bs (G/N)$ by mapping $HgN$ to the $H$-orbit of $gN$. We have a bijection $H\bs G/N\to N\bs G/H$ given by $HgN\mapsto Ng^{-1}H$.         

If $X$ and $X'$ are normal, connected Noetherian schemes\footnote{Note that normal locally Noetherian schemes are connected if and only if they are integral by \cite[\href{https://stacks.math.columbia.edu/tag/033N}{Lemma 033N}]{stacks-project}. } and $\phi: X'\to{X}$ is a finite dominant morphism, then we call this morphism a \emph{covering} of $X$. Note that we do not require this covering to be \'{e}tale. The function field of a normal, connected Noetherian scheme is denoted by $K(X)$ or $L$ if the scheme $X$ is clear from context. We will assume that coverings %
$X'\to{X}$ are separable, in the sense that they are \'{e}tale at the generic point of $X'$.  %
Equivalently, the extensions of function fields are all separable. Note that to give a covering of $X$, it suffices to give a separable irreducible polynomial $f(z)\in{K(X)}$, since the morphism $X'\to{X}$ can be recovered by taking the normalization of $X$ in the field extension $K(X)\to K(X)[z]/(f(z))$. %
If $\Sigma$ is a (multi)graph, then we denote its vertex set by $V(\Sigma)$ and its edge set by $E(\Sigma)$. %

\begin{remark}
We have implemented part of these algorithms in this paper in OSCAR, especially when $K$ is a field. For schemes over the valuation rings $R$ of $p$-adic fields, these implementations require a work-around in terms of splitting the reduction map $R\to k$, where $k$ is the residue field. For now, we have included an example of a $p$-adic curve for $p$ large that comes from a curve over $\mathbb{Q}(t)$, see \cref{rem:WorkAround}. %
A general $p$-adic implementation of the techniques in this paper is the subject of future work. %
\end{remark}

\subsection{Acknowledgements}

This paper has gone through many iterations over the years, and the author is indebted to many people for their remarks, comments and examples, including: Yue Ren, Jeffrey Giansiracusa, Kazuhiko Yamaki, Akio Tamagawa, Benjamin Collas, Bernd Sturmfels, Hannah Markwig, Michael Stillman, Roberto Gualdi, Yassine El Maazouz, Enis Kaya, Alexander Betts, Sachi Hashimoto, Ole Ossen and Adam Morgan. 

The author was supported by UK Research and
Innovation under the Future Leaders Fellowship programme MR/S034463/2  %
as a Postdoctoral Fellow at Durham University, and the JSPS Postdoctoral Fellowship with ID No. 23769 and KAKENHI 23KF0187 as a Postdoctoral Fellow at Tohoku University and the University of Tsukuba. A version of the implemented algorithms can be freely accessed on the author's website at \href{https://paulhelminck.wordpress.com/code}{https://paulhelminck.wordpress.com/code}. These include code for the two main examples of this paper: \cref{exa:MainExample1} and \cref{exa:PlaneQuartic}. It is written in OSCAR, see \cite{OSCAR} and \cite{OSCAR-book}. 

\tableofcontents

\section{Reconstructing poset structures}\label{sec:ReconstructionSection}

In this section we study the problem of reconstructing the relative poset structure of a covering of normal connected Noetherian schemes in terms of group-theoretic data. We first establish our notation for posets in \cref{sec:PreliminariesPosets}. We then define a group-theoretically enhanced poset in \cref{sec:Topoi}, together with its corresponding glued category. %
From \cref{sec:CovSchemes1} onward, we show how to apply this framework to the problem of reconstructing relative poset structures of coverings of normal connected Noetherian schemes. We will build up our reconstruction algorithm step-by-step, starting with posets consisting of a single point, then extending the result to edges and trees, and ending with general finite posets.  %

\subsection{Preliminaries on posets}\label{sec:PreliminariesPosets}
We first fix our notation for various set-theoretic quantities. 
\begin{definition}
Let $X$ be a topological space that is $T_{0}$, so that %
for every pair $(x,y)\in{X^{2}}$ of distinct points, there is an open neighborhood of either $x$ or $y$ that does not contain the other.  
Let $x_{1},x_{2}\in{X}$. %
The relation 
\begin{center}
$x_{1}\leq x_{2}$ $\Leftrightarrow$ $x_{2}\in \overline{\{x_{1}\}}$
\end{center} then defines a partial order on $X$. In particular, if $X$ is a scheme, then this endows $X$ with the structure of a poset. 
\end{definition}

\begin{remark}
We can view a $T_{0}$-space $X$ as an acyclic directed graph through its poset structure. %
Similarly, we can view $X$ as a $1$-category whose objects are the elements of $X$ and whose arrows $x\to y$ are given by relations $x\leq {y}$. We will do this without further mention. 
\end{remark}
\begin{definition}
Let $S\subset{X}$ be a subset of a $T_{0}$-space $X$ and endow $S$ with its induced partial order. We denote its associated \emph{Hasse diagram} by $S_{HD}$, which we consider as a directed acyclic graph. %
Here $e=xy\in{E(S_{HD})}$ for $x\neq{y}$ if and only if $x$ covers $y$, in the sense that there is no $z\in{S}$ with $x>z>y$. The barycentric subdivision $S_{HD,B}$ of $S_{HD}$ is the directed acyclic graph obtained by replacing every edge $e=xy\in E(S_{HD})$ by a vertex. The directed edges are given by $e_{x}=(x,e)$ and $e_{y}=(y,e)$. 
\end{definition}
\begin{remark} Note that %
 we have effectively removed edges from the directed acyclic graph of $S$ that can be deduced from reflexivity and transitivity in the Hasse diagram. In all cases in this paper, we can recover the poset structure of $S$ from the graph structure of its Hasse diagram (either because $S$ is finite, or because it is the inverse image of a finite set under an integral morphism). 
 \end{remark}

\subsection{Group-theoretically enhanced posets}\label{sec:Topoi}

In this section we define an enhanced version of a poset by endowing its elements with extra group-theoretic data. This extra data essentially consists of a profinite group $G$ attached to an element of the poset, together %
with comparison maps for adjacent elements. %
The local group-theoretic data introduced in \cref{fig:Correspondences} will be objects inside the category of topological spaces $X$ with an action of $G$ on $X$. We denote this by $G$-{Spaces}. For finite coverings, these spaces $X$ will simply be discrete spaces. To glue this group-theoretic data, we introduce a glued category (which is usually called the $2$-limit), which we will use to reconstruct poset structures.    
 The $2$-categorical phenomenona will only occur if the order complexes of the posets are sufficiently complicated. We will give an explicit algebraic example of this in terms of power series data in \cref{exa:ExplicitTransferMap}.      %

We start with the group-theoretic data over a single vertex.   %

\begin{definition}\label{def:SimpleTopoi}
Let $G$ be a profinite group. %
The objects of the category $G$-Spaces are topological spaces $X$ with a continuous action of $G$ on $X$. %
Morphisms in $G\mhyphen\text{Spaces}$ are $G$-equivariant continuous maps $X_{1}\to X_{2}$. %
We write $\mathcal{C}$ for the $2$-category of categories of the form $G$-Spaces for some profinite group $G$. %
The $1$-morphisms in this category are the functors arising from continuous homomorphisms of groups $G_{1}\to G_{2}$. The $2$-morphisms are natural transformations of functors. These will not play a large role in this paper.  %
The $1$-morphisms arising from continuous morphisms of profinite groups in this paper will either come from %
inclusions or conjugation isomorphisms.
\end{definition}

\begin{remark}
The $G$-spaces of interest to us in this paper are as follows. In the first case, we have %
$G=D_{x}$ and $X=V(f(z))=\{\alpha_{1},...,\alpha_{n}\}$, where $x$, $D_{x}$, $V(f(z))$ and the $\alpha_{i}$ are as in \cref{sec:TheoreticalOverview}. In this case we endow $X$ with the discrete topology. 

For our second example (again with the notation as in \cref{sec:TheoreticalOverview}), %
consider the action of $D_{x}$ on the ambient absolute Galois group $G_{\sep}:=\mathrm{Gal}(K(X)^{\sep}/K(X))$ by left-multiplication. We then take $G=D_{x}$ and $X=G_{\sep}$. Similarly, if $H$ is a closed subgroup of $G_{\sep}$, then we obtain a continuous action of $G=D_{x}$ on the set of left cosets $X=G_{\sep}/H$, where the latter has the quotient topology. Note that $X$ usually does not have the discrete topology (for instance for $H=(1)$). This explains our need for working with the category of $G$-Spaces rather than $G$-Sets. In our algorithms, we will however always work with finite sets. %
\end{remark}

To glue these sets with a group action, we introduce a group-theoretic enhancement of a poset.

\begin{definition}\label{def:GroupTheoreticEnhancement}
Let $I$ be a finite directed acyclic graph and let $\mathcal{G}$ be the category of profinite topological groups.  %
A group-theoretic enhancement of $I$ is a contravariant functor
\begin{equation*}
\mathcal{F}: I\to \mathcal{G}
\end{equation*}
In other words, for every vertex $i\in{I}$ we have a profinite group $\mathcal{F}(i)=G_{i}$, and for every directed edge $i\to j$ %
we have a continuous group homomorphism
\begin{equation*}
\psi^{*}_{i,j}:G_{j}\to G_{i}
\end{equation*}
These arrows satisfy a compatibility condition: for two arrows $i\to j$ and $j\to k$ with composition $i\to k$, we have that $\psi^{*}_{k,i}=\psi^{*}_{k,j}\circ{\psi^{*}_{j,i}}$. %
We write %
\begin{equation*}
\psi_{i,j}: G_{i}\mhyphen\text{Spaces}\to G_{j}\mhyphen\text{Spaces} %
\end{equation*}
for the induced functor of categories of topological spaces with an action by $G_{i}$ and $G_{j}$ respectively. Explicitly, this functor is obtained as follows. Let $X$ be a set with a $G_{i}$-action. This is given by %
a group homomorphism $G_{i}\to \mathrm{Aut}(X)$. We then compose this with the group homomorphism $G_{j}\to G_{i}$ to obtain an action of $G_{j}$ on $X$.  %

Suppose that $S$ is a finite poset with Hasse diagram $S_{HD}$, considered as a finite directed acyclic graph. A group-theoretic enhancement of $S$ is a group-theoretic enhancement of the graph $S_{HD}$ or its barycentric subdivision $S_{HD,B}$. 
\end{definition}

To every group-theoretic enhancement, we can now associate a new category $\mathcal{C}(\mathcal{F})$, %
which can be seen as containing all the glued data.

\begin{definition}\label{def:2Limit}
Let $\mathcal{F}$ be a group-theoretic enhancement of a finite directed acyclic graph $I$. We assume here that every arrow is terminal, in the sense that if $i\to j$ is a non-trivial arrow, then there is no $j\to k$ other than the identity map $j\to j$.   %
We follow \cite[Section 4]{Pirashvili2015} and construct a $1$-category $\mathcal{C}(\mathcal{F})$ %
as follows\footnote{Our functors are covariant, so some of the notation is different, see  \cite[Section 1.1]{Pirashvili2020}. }. An object in $\mathcal{C}(\mathcal{F})$ %
consists of the following:
\begin{enumerate}
\item An object $S_{i}\in G_{i}\mhyphen\text{Spaces}$ for every $i\in I$.
\item An isomorphism $\xi_{i,j}:\psi_{i,j}(S_{i})\to S_{j}$ in the category $G_{j}\mhyphen\text{Spaces}$ for every morphism $i\to j$. We call these isomorphisms %
 \emph{transfer maps}.
\end{enumerate}
 An object in $\mathcal{C}(\mathcal{F})$ %
 will be denoted by $(S_{i},\xi_{i,j})$. A morphism $(S_{i},\xi_{i,j})\to (T_{i},\nu_{i,j})$ is a collection of maps $\rho_{i}:S_{i}\to T_{i}$ that commute with the transfer maps $\xi_{i,j}$ and $\nu_{i,j}$, in the sense that the following diagrams commute:
\begin{displaymath}
\begin{tikzcd}
{\psi_{i,j}(S_{i})} \arrow[d,"{\psi_{*}(\rho_{i})}"] \arrow[r, "{\xi_{i,j}}"] & {S_{j}} \arrow[d,"\rho_{j}"] \\
{\psi_{i,j}(T_{i})}  \arrow[r,"{\nu_{i,j}}"]          & {T_{j}}
\end{tikzcd}
\end{displaymath}
We will also refer to this category $\mathcal{C}(\mathcal{F})$ as the $2$-limit of the induced diagram of $G\mhyphen\text{Spaces}$, see \cref{rem:2Limits}.   %
\end{definition}   

\begin{remark}
Normally, one also requires the transfer maps to satisfy a natural \emph{cocycle condition}. This is however a trivial condition by our assumption on $I$, so we omitted this here. The graphs arising from our constructions are all arrow-terminal in the sense of the definition, so this suffices for us. 
\end{remark}

\begin{remark}\label{rem:2Limits}
As mentioned earlier, we think of the $1$-category $\mathcal{C}(\mathcal{F})$ introduced in \cref{def:2Limit} as a $2$-limit. That is, there is an induced functor from $I$ to the $2$-category of categories equivalent to $G\mhyphen{\text{Spaces}}$ for some profinite topological group $G$. Here a functor from a $1$-category to a $2$-category is taken in the sense of \cite[\href{https://stacks.math.columbia.edu/tag/003N}{Definition 003N(1)}]{stacks-project}. The category $\mathcal{C}(\mathcal{F})$ is then the $2$-limit of this diagram in the sense of \cite[Section 4]{Pirashvili2015}. We will not make use of the $2$-categorical properties of this limit however, so the reader can simply think of the $1$-category $\mathcal{C}(\mathcal{F})$. As mentioned before, this category contains all of the data for our power series algorithms.   %
\end{remark}

\begin{remark} \label{rem:Tree1Limit}
Suppose that the undirected graph associated to $I$ is a connected tree.  %
Then we can think of the isomorphisms in \cref{def:2Limit} as equalities. %
If $I$ is not a connected tree however, then the additional isomorphisms can create new objects. %
Suppose for instance that we take the circle graph in \cref{fig:CircleDAG} and  %
assign the trivial group to every vertex. Over every vertex, we now take the set $\{1,2\}$. When we move from one vertex to the other, this set is unchanged by the pullback maps. %
In the $2$-limit, we can however use isomorphisms (which are simply bijections of sets) to twist objects. For instance, as in \cref{fig:CircleDAG}, we can twist using the non-trivial permutation $1\mapsto 2$ and $2\mapsto 1$. The induced object of $\mathcal{C}(\mathcal{F})$ %
can be seen as corresponding to the non-trivial degree-$2$ covering of the circle, see \cref{fig:CircleCovering}. We will see an explicit realization of this example in \cref{exa:ExplicitTransferMap}. There, the set $\{1,2\}$ will correspond to local power series expansions $\{\gamma_{1},\gamma_{2}\}$ of an irreducible polynomial $f(z)$ that gives a degree-$2$ covering of an elliptic curve.        
\end{remark}

\begin{figure}[h]
\scalebox{0.6}{
\begin{tikzpicture}[line cap=round,line join=round,>=triangle 45,x=1cm,y=1cm]
\clip(3,3) rectangle (17,13);
\draw [->,line width=0.1pt] (9,12) -- (9.1,12);
\draw [->,line width=0.1pt] (9,11) -- (9.1,11);
\draw [->,line width=1pt] (6,8) -- (8,7);
\draw [->,line width=1pt] (14,8) -- (12,7);
\draw [->,line width=1pt] (14,8) -- (12,9);
\draw [->,line width=1pt] (6,8) -- (8,9);
\draw [line width=1pt] (8,9)-- (10,10);
\draw [line width=1pt] (10,10)-- (12,9);
\draw [line width=1pt] (8,7)-- (10,6);
\draw [line width=1pt] (10,6)-- (12,7);
\draw [line width=1pt] (6,8)-- (10,6);
\draw [line width=1pt] (10,6)-- (14,8);
\draw [line width=1pt] (10,10)-- (14,8);
\draw [line width=1pt] (6,8)-- (10,10);
\draw [line width=1pt] (4,7)-- (5,7);
\draw [line width=1pt] (5,7)-- (5,9);
\draw [line width=1pt] (5,9)-- (4,9);
\draw [line width=1pt] (4,9)-- (4,7);
\draw [line width=1pt] (15,9)-- (15,7);
\draw [line width=1pt] (15,7)-- (16,7);
\draw [line width=1pt] (16,7)-- (16,9);
\draw [line width=1pt] (16,9)-- (15,9);
\draw [line width=1pt] (9.5,10.5)-- (10.5,10.5);
\draw [line width=1pt] (10.5,10.5)-- (10.5,12.5);
\draw [line width=1pt] (10.5,12.5)-- (9.5,12.5);
\draw [line width=1pt] (9.5,12.5)-- (9.5,10.5);
\draw [line width=1pt] (9.5,11.5)-- (10.5,11.5);
\draw [line width=1pt] (4,8)-- (5,8);
\draw [line width=1pt] (15,8)-- (16,8);
\draw [line width=1pt] (9.5,5.5)-- (9.5,3.5);
\draw [line width=1pt] (9.5,3.5)-- (10.5,3.5);
\draw [line width=1pt] (10.5,3.5)-- (10.5,5.5);
\draw [line width=1pt] (10.5,5.5)-- (9.5,5.5);
\draw [line width=1pt] (9.5,4.5)-- (10.5,4.5);
\draw (4.25,8.75) node[anchor=north west] {$1$};
\draw (4.25,7.75) node[anchor=north west] {$2$};
\draw (9.75,12.25) node[anchor=north west] {$1$};
\draw (9.75,11.25) node[anchor=north west] {$2$};
\draw (15.25,8.75) node[anchor=north west] {$1$};
\draw (15.25,7.75) node[anchor=north west] {$2$};
\draw (9.75,5.25) node[anchor=north west] {$1$};
\draw (9.75,4.25) node[anchor=north west] {$2$};
\draw [shift={(9,11.5)},line width=1pt]  plot[domain=1.5707963267948966:4.71238898038469,variable=\t]({1*0.5*cos(\t r)+0*0.5*sin(\t r)},{0*0.5*cos(\t r)+1*0.5*sin(\t r)});
\begin{scriptsize}
\draw [fill=ududff] (6,8) circle (3pt);
\draw [fill=ududff] (10,6) circle (3pt);
\draw [fill=ududff] (14,8) circle (3pt);
\draw [fill=ududff] (10,10) circle (3pt);
\draw [fill=ududff] (4,7) circle (3pt);
\draw [fill=ududff] (5,7) circle (3pt);
\draw [fill=ududff] (5,9) circle (3pt);
\draw [fill=ududff] (4,9) circle (3pt);
\draw [fill=ududff] (15,9) circle (3pt);
\draw [fill=ududff] (15,7) circle (3pt);
\draw [fill=ududff] (16,7) circle (3pt);
\draw [fill=ududff] (16,9) circle (3pt);
\draw [fill=ududff] (9.5,10.5) circle (3pt);
\draw [fill=ududff] (10.5,10.5) circle (3pt);
\draw [fill=ududff] (10.5,12.5) circle (3pt);
\draw [fill=ududff] (9.5,12.5) circle (3pt);
\draw [fill=ududff] (9.5,5.5) circle (3pt);
\draw [fill=ududff] (9.5,3.5) circle (3pt);
\draw [fill=ududff] (10.5,3.5) circle (3pt);
\draw [fill=ududff] (10.5,5.5) circle (3pt);
\end{scriptsize}
\end{tikzpicture}
}
\caption{\label{fig:CircleDAG}
The directed acyclic graph in \cref{rem:Tree1Limit}, with the trivial enhancement. Over each vertex, we can consider the set $\{1,2\}$, which can be thought of as the set of branches of a local covering space. In the category $\mathcal{C}(\mathcal{F})$, %
these branches can be glued using local twists. For instance, we can permute the set above by sending $1$ to $2$, and $2$ to $1$. This gives a non-trivial covering space of the circle, see \cref{exa:ExplicitTransferMap} and \cref{fig:CircleCovering}.     %
}
\end{figure}

\begin{example}\label{exa:TropicalCurve}
Let $\Sigma$ be a finite connected undirected multigraph with a weight function $w:V(\Sigma)\to \mathbb{N}$ on the vertices and a length function $\ell:E(\Sigma)\to \mathbb{R}_{>0}$ on the edges. We call this a tropical curve. We will assume withous loss of generality that $\Sigma$ is loopless. By subdividing the edges of $\Sigma$, we can view this as the Hasse diagram associated to the poset structure of the special fiber of a semistable model $\mathcal{X}$, see \cref{def:SFPoset} for the exact definition. %
Note that an intersection point of two irreducible components becomes a single element in terms of posets, so that the Hasse diagram $\Sigma_{HD}$ of the corresponding poset $\Sigma_{PS}\subset \mathcal{X}$ is a barycentric subdivision of the dual intersection graph of the special fiber.  For instance, the directed acyclic graph in \cref{fig:CircleDAG} is the Hasse diagram of two irreducible components meeting in two points, see \cref{exa:SemistableModelSection} for an explicit semistable model.   We turn $\Sigma_{HD}$ into a vertex-weighted graph by adding $g(x)=0$ for the newly introduced points; the genera of the other points remain the same.  

We can define a group-theoretic enhancement as follows. Let $x$ be an element of $\Sigma_{PS}$ and let $\val(x)$ be the valence of $x$. Let $(\overline{C}_{x},D)$ be a marked Riemann surface of genus $g=g(x)$ with $r=\val(x)$ marked points $D\subset \overline{C}_{x}$. We assign the profinite fundamental group $\pi(C_{x})$ of $C_{x}=\overline{C}_{x}\bs D$ to $x$. This group has the following well-known presentation. Let $G_{0}$ be the free group on the generators $a_{i}$ and $b_{i}$ for $1\leq{i}\leq{g}$, and the generators $c_{j}$ for $1\leq{j}\leq r$, and let %
$H\subset G$ be the subgroup generated by %
$(\prod_{i=1}^{g} [a_{i},b_{i}])\cdot (\prod_{j=1}^{r} c_{j})$, where $[a_{i},b_{i}]=a_{i}b_{i}a_{i}^{-1}b_{i}^{-1}$. We then have $\pi(C_{x})\simeq \widehat{G/H}$.   
Through this presentation, we obtain a natural 
injection 
\begin{equation*}
\hat{\Z}\to \pi(C_{x})
\end{equation*}  
for every marked point. %
This gives the desired comparison maps for the edges, so that we obtain a group-theoretic enhancement of $\Sigma_{PS}$. %
Note that the twisting maps over edges for sets with a trivial group action can be detected in terms of coverings of graphs as non-trivial maps of graphs. 
The non-trivial $\hat{\mathbb{Z}}$-twisting maps that we allow in the $2$-limit can then be seen as incarnations of the gluing data introduced in \cite{Helminck2023}.      %
We also invite the reader to compare the notion of a $2$-limit in this case with the graph of groups construction in algebraic topology, see  %
\cite[Section 1.A]{Hatcher2001}.
\end{example}

\begin{remark}
In terms of algebraic topology, we can view a group-theoretic enhancement as adding a $K(\pi,1)$-space   (namely, the one corresponding to the group $\mathcal{F}(i)$) to every vertex.  %
From this point of view, taking the $2$-limit is somewhat strange, as this removes some of the higher-dimensional topology. %
Our reason for considering $2$-limits in this paper is that these types of limits are well-suited for one-dimensional problems such as reconstructing poset structures. Here we view the latter as a one-dimensional problem by reducing a poset to its Hasse diagram. 
\end{remark}

\subsection{Reconstructing fibers}\label{sec:CovSchemes1}
For the remainder of \cref{sec:ReconstructionSection}, we focus on schemes and their associated poset structures. 
Let $L=K(X)$ be the function field of a normal connected Noetherian scheme $X$ and let $L^{\sep}$ be a fixed separable closure. We write $X^{\sep}$ for the normalization of $X$ inside $L^{\sep}$ and $\phi^{\sep}:X^{\sep}\to X$ for the corresponding morphism of normalizations. We denote the absolute Galois group of $L^{\sep}/L$ by $G$, which acts on $X^{\sep}$. Let $x^{\sep}\in X^{\sep}$ be a point lying over $x$ and let 
\begin{equation*}
D_{x^{\sep}/x}=\{\sigma\in{G}:\sigma(x^{\sep})=x^{\sep}\} 
\end{equation*}
be the decomposition group of $x^{\sep}$ over $x$. If $x^{\sep}$ is clear from context, we will also denote this by $D_{x}$. 
The double coset spaces associated to $D_{x^{\sep}/x}$ completely determine the different fibers over $x$ by the following. %

\begin{lemma}\label{lem:RecoveringPoints}
Let $L\to L'\to L^{\sep}$ be a fixed separable extension of $L$ with corresponding closed subgroup $H\subset {G}$, and let $\phi_{H}: X_{H}\to X$ be the induced morphism of normalizations.
There is a bijection %
\begin{equation*}
H\bs G/D_{x^{\sep}/x}\simeq H\bs (G/D_{x^{\sep}/x})\to %
\phi^{-1}(x). 
\end{equation*}
Explicitly, we send a double coset space $H\tau D_{x^{\sep}/x}$ to %
${\phi^{H}}(\tau({x}^{\sep}))$, where ${\phi^{H}}:{X^{\sep}}\to X'$ is the morphism induced from $L'\to L^{\sep}$. In term of reverse double coset spaces $D_{x^{\sep}/x}\bs G/H$, this sends $D_{x^{\sep}/x}\tau{H}$ to ${\phi^{H}}(\tau^{-1}({x}^{\sep}))$.  %
\end{lemma}  
\begin{proof}

See \cite[Lemma 2.12]{Helminck2024}. 
 The proof follows from 
$X^{\sep}/H=X'$ and the fact that 
$D_{x^{\sep}/x}$ acts transitively on the fibers of $X^{\sep}\to X$ (see \cite[\href{https://stacks.math.columbia.edu/tag/0BRK}{Lemma 0BRK}]{stacks-project}).
\end{proof}

\subsection{Reconstructing local specializations}

We now consider the problem of reconstructing the poset structure of a finite cover over finite subtrees of $X$. To do this, we will need the %
notion of being relatively unibranch, which is a condition on the singularities of $X$ that allows us to  globalize the theory of decomposition groups to larger posets. %
\begin{definition}
An affine scheme $X=\Spec(A)$ is \emph{unibranch} if $A$ has a unique minimal prime $\mathfrak{p}$ and $A/\mathfrak{p}$ is integrally closed in its field of fractions. 
\end{definition}
\begin{lemma}
The spectrum of a local affine ring $A$ with maximal ideal $\mathfrak{m}$ is unibranch if and only if the Henselization $A^{h}$ of $A$ with respect to $\mathfrak{m}$ has a unique minimal prime.   
\end{lemma} 
\begin{proof}
See \cite[\href{https://stacks.math.columbia.edu/tag/0BQ0}{Lemma 0BQ0}]{stacks-project}. 
\end{proof} 
\begin{definition}
A scheme $X$ is unibranch at $x\in{X}$ if its local ring $\mathcal{O}_{X,x}$ is unibranch.
\end{definition}  %
\begin{definition}\label{def:RelativelyUnibranch}
Let $X$ be a scheme and let $y$ be a generization of $x$, so that $x\in V(y)=\overline{\{y\}}$. We say that $X$ is relatively unibranch with respect to $x\geq{y}$ if the closed subscheme $V(y)$ is unibranch at $x\in V(y)$. Let $S$ be a subset. We say that $X$ is \emph{relatively unibranch} over $S$ if $X$ is relatively unibranch with respect to any pair of points in $S$.  
\end{definition}

\begin{lemma}
Let $x,y\in{X}$ be points such that $x\geq {y}$. 
We view $x$ and $y$ %
as points of $\mathrm{Spec}(\mathcal{O}_{X,x})$. Then $X$ is relatively unibranch with respect to the pair $x\geq{y}$ if and only if the fiber over $y$ of the Henselization morphism 
\begin{equation*}
\mathrm{Spec}(\mathcal{O}^{h}_{X,x})\to \Spec(\mathcal{O}_{X,x})
\end{equation*} 
consists of one point. 
\end{lemma}
\begin{proof}
This follows from \cite[\href{https://stacks.math.columbia.edu/tag/05WQ}{Lemma 05WQ}]{stacks-project}, which says that taking the Henselization is compatible with taking quotients. 
\end{proof}

\begin{example}\label{exa:NotUnibranch}
Let $K$ be a field and $X=\Spec(K[x,y])$ with $\mathfrak{m}=(x,y)$ and $\mathfrak{p}=(y^2-x^2(x+1))$. Then $X$ is not relatively unibranch with respect to the pair $\mathfrak{m}\geq\mathfrak{p}$. Indeed, if we write $z=\sqrt{x+1}\in\mathcal{O}^{h}_{X,\mathfrak{m}}$, then there are two prime ideals $\mathfrak{p}_{\pm}=(y\pm xz)$ lying over $\mathfrak{p}$. 

Birationally, we can however make this pair relatively unibranch again: consider the morphism $X_{0}=\Spec(K[x_{1},y_{1}])\to X$ given on the level of rings given $x\mapsto x_{1}$ and $y\mapsto y_{1}x_{1}$. Then the pair $\mathfrak{m}_{1}=(x_{1},y_{1})\supset (y_{1}^2-(x_{1}-1))=\mathfrak{p}_{1}$ is relatively unibranch, and it maps to $\mathfrak{m}\geq\mathfrak{p}$. Moreover, the extended poset where we add $(x_{1})$ is also relatively unibranch, so that the poset of the total transform is relatively unibranch.    %
\end{example}
\begin{example}\label{exa:RegularPointsRelUnibranch}
Suppose that $x$ is a regular point of $X$ with local system of parameters $z_{1},...,z_{k}$. For any subset of this system, we obtain a point $y$ with $x\geq{y}$ (see \cite[\href{https://stacks.math.columbia.edu/tag/00NQ}{Lemma 00NQ}]{stacks-project}), and $X$ is relatively unibranch with respect to this pair.  
\end{example}

\begin{remark}
For the remainder of this paper, we will fix a finite set $S\subset{X}$, and we will assume that $X$ is relatively unibranch with respect to $S$. 
\end{remark}
As we will see shortly, the notion of being relatively unibranch will allow us to extend various techniques from Galois theory over a single point to more general posets. 

\begin{definition}\label{def:AbsoluteSections}
Consider a finite subset $S\subset{X}$ and let $S_{R,HD}\subset S_{HD}$ be a subgraph of the Hasse diagram of $S$. A section of $S_{R,HD}$ in $X^{\sep}$ is a finite subset $S^{\sep}\subset X^{\sep}$ with a subgraph $S^{\sep}_{R,HD}\subset S^{\sep}_{HD}$ that induces an isomorphism of graphs $S^{\sep}_{R,HD}\to S_{R,HD}$. %
We will also refer to this as an extension of $S_{R,HD}$. If $S$ consists of two points $x$ and $y$ with $x\leq{y}$, then we call the corresponding graph $S_{R,HD}$ an edge. In this case we will also simply write $S$ or $e$ for either the graph or the subset.   
\end{definition}
\begin{lemma}\label{lem:SectionLemma}
Suppose that $S_{R,HD}$ is a tree. Then there is a section $S^{\sep}_{R,HD}$ of $S_{R,HD}$ in $X^{\sep}$.  
\end{lemma}   
\begin{proof}
This follows from the Cohen-Seidenberg theorems from commutative algebra. 
\end{proof}

\begin{example}\label{exa:SemistableModelSection}
For general subgraphs, this section need not exist. Consider for instance a discrete valuation ring $R$ with uniformizer $\pi$, and let $\mathcal{X}=\Spec(R[u,v,y]/(y^2-(1-u)(1+v),uv-\pi))$. This is a local affine chart of a semistable model of an elliptic curve $E$ over the fraction field of $R$. Let 
\begin{align*}
\mathfrak{m}_{+}=(u,v,y-1),& \qquad &\mathfrak{m}_{-}=(u,v,y+1),\\
\mathfrak{p}_{u}=(u,y^2-(1+v)),&\qquad &\mathfrak{p}_{v}=(v,y^2-(1-u)).
\end{align*} The Hasse diagram of the subset $S=\{\mathfrak{m}_{+},\mathfrak{m}_{-},\mathfrak{p}_{u},\mathfrak{p}_{v}\}$ %
forms a circle-like graph, corresponding to the dual graph of the special fiber of $\mathcal{X}$, see \cref{fig:CircleDAG}. %

The $2:1\mhyphen$covering induced by the polynomial $z^2-(1+v)$ gives a completely split covering over the special fiber $\mathcal{X}_{s}=\mathcal{X}\times_{R} k$ by a quick calculation, see \cref{exa:ExplicitTransferMap}. The inverse image of $S$ is moreover connected, so that we cannot find a section. Note that the $2$-torsion point used for the covering here reduces to the toric part of $E$. %
\end{example}

\begin{lemma}\label{lem:TransitiveOnChains}
Let $e_{i}={x}^{\sep}y^{\sep}_{i}$ be two extensions of the edge $e=xy$ to $X^{\sep}$. Then there is a $\sigma \in D_{x^{\sep}/x}$ such that $\sigma(y^{\sep}_{1})=y^{\sep}_{2}$. 
\end{lemma} 

\begin{proof}
Write $\phi^{h}:X^{\sep}\to X^{h}$ for the morphism induced by $(1)\subset D_{x^{\sep}/x}$. We can identify the local ring $\mathcal{O}_{X^{h},x^{h}}$ of $X^{h}$ with respect to $x^{h}:=\phi^{h}(x^{\sep})$ with the Henselization $\mathcal{O}^{h}_{X,x}$. 
The group $D_{x^{\sep}/x}$ now acts on $\mathcal{O}_{X^{\sep},x^{\sep}}$ and its ring of invariants is %
$\mathcal{O}_{X^{h},x^{h}}=\mathcal{O}^{h}_{X,x}$.  Moreover, $\mathcal{O}_{X^{\sep},x^{\sep}}$ is the integral closure of $\mathcal{O}^{h}_{X,x}$ in $L^{\sep}$, since $x^{\sep}$ is the only extension of $\phi^{h}(x^{\sep})$. Since $X$ is relatively unibranch with respect to $x$ and $y$, we find that $y^{\sep}$ and $y'^{\sep}$ are extensions of the same point in $\Spec(\mathcal{O}^{h}_{X,x})$.   %
The statement now follows from \cite[\href{https://stacks.math.columbia.edu/tag/0BRK}{Lemma 0BRK}]{stacks-project}. 

\end{proof}

\begin{example}
Consider again the pair in \cref{exa:NotUnibranch}. We now consider the finite covering $X'\to X$ induced by the polynomial $f(w)=w^2-(x+1)$. There are two extensions over both $\mathfrak{m}=(x,y)$ and $\mathfrak{p}=(y^2-x^2(x+1))$, but there are four extensions of the pair. Since the Galois group is $\Z/2\Z$, there does not exist a $\sigma$ that acts transitively on extensions as in \cref{lem:TransitiveOnChains}.    
\end{example}

\begin{definition}
Let $S^{\sep}=e^{\sep}=x^{\sep}y^{\sep}$ be a section of an edge $e=xy$ in $X^{\sep}$. %
We define 
\begin{equation*}
D_{e^{\sep}/e}=D_{x^{\sep}/x}\cap D_{y^{\sep}/y}
\end{equation*} 
to be the relative decomposition group. 
\end{definition}

\begin{lemma}\label{lem:ReconstructChains2}
Let $E_{e}$ be the set of extensions of $e=xy$ to $X^{\sep}$. The map 
\begin{equation*}
G/D_{e^{\sep}/e}\to E_{e}
\end{equation*}
given by 
$$\sigma D_{e^{\sep}/e}\mapsto \{\sigma(x^{\sep})\supset \sigma(y^{\sep})\}$$ is surjective. 
Two cosets $\sigma_{i}D_{e^{\sep}/e}$ give the same extension if and only if their images under $G/D_{e^{\sep}/e}\to G/D_{x^{\sep}/x}$ and $G/D_{e^{\sep}/e}\to G/D_{y^{\sep}y}$ are the same.   
\end{lemma}
\begin{proof}
Let $e'^{\sep}=x'^{\sep}y'^{\sep}$ be an extension of $e=xy$. Let $\sigma$ be an element such that $\sigma(x^{\sep})=x'^{\sep}$. We then have the extension $x'^{\sep}\supset \sigma(y^{\sep})$. By \cref{lem:TransitiveOnChains}, there is a $\tau\in D_{x'^{\sep}/x}$ such that $\tau(\sigma(y^{\sep}))=y'^{\sep}$. The left coset $\tau\sigma D_{e^{\sep}/e}$ then maps onto the given extension. The last part is immediate. 
\end{proof}
\begin{lemma}\label{lem:RecoveringEdges}
Let $L\subset L'\subset L^{\sep}$ be a separable extension with morphisms $X^{\sep}\to X'\to X$, and let $E_{e,\phi}$ be the set of extensions of $e$ under the map $\phi:X'\to X$. 
The map $$H \tau D_{e^{\sep}/e}\mapsto \{\phi^{\sep}(\tau(x^{\sep}))\supset\phi^{\sep}(\tau(y^{\sep}))\}$$ induces a surjective map $H\bs G/D_{e^{\sep}/e}\to E_{e,\phi}$. Two double cosets give the same extension if and only if their images in $H\bs G/D_{x^{\sep}/x}$ and $H\bs G/D_{y^{\sep}/y}$ are the same. 
\end{lemma}
\begin{proof}
The map is well defined since $X^{\sep}\to X'=X^{\sep}/H$. %
To see that it is surjective, we take an extension $e'$ of $e$ and extend it to an edge $e'^{\sep}$ in $X^{\sep}$. By \cref{lem:TransitiveOnChains}, there is a $\sigma$ such that $\sigma(e^{\sep})=e'^{\sep}$. From this, we directly find that $H\sigma D_{e^{\sep}/e}$ maps to $e'$, and the last part also immediately follows.               
\end{proof}

\begin{proposition}\label{pro:RecoveringTrees}
Let $\phi:X'\to X$ be as in the beginning of \cref{sec:CovSchemes1} and let $S\subset X$ be a finite set with a subtree $S_{T,HD}\subset S_{HD}$. We write $S^{\sep}_{T,HD}$ for an absolute section as in %
\cref{def:AbsoluteSections}. The graph $\phi^{-1}(S_{T,HD})$ can be recovered from the double coset diagrams 
  \begin{equation*}
 \begin{tikzcd}
{} & H\bs G/D_{e^{\sep}/e} \arrow[dl] \arrow[dr] & {} \\
H\bs G/D_{x^{\sep}/x} & {} & H\bs G/D_{y^{\sep}/y}
\end{tikzcd}
 \end{equation*}
 arising from \cref{lem:RecoveringEdges}. In particular, if $S_{HD}$ is already a tree, then this recovers the poset structure of $\phi^{-1}(S)$\footnote{As in the first part of  \cref{lem:RecoveringPoints}, we used the double cosets $H\bs G/D_{e^{\sep}/e}$, $H\bs G/D_{x^{\sep}/x}$ and $H\bs G/D_{y^{\sep}/y}$ here. We will use the corresponding version with reversed double cosets later on in \cref{sec:FiberPosets}.     }.  
\end{proposition}
\begin{proof}
This directly follows from \cref{lem:RecoveringEdges}. 
\end{proof}

\begin{remark}
If $S$ is a single point, then we can replace the decomposition group $D_{x^{\sep}/x}$ in \cref{pro:RecoveringTrees} by a conjugate $\sigma D_{x^{\sep}/x}\sigma^{-1}=D_{\sigma(x^{\sep})/x}$ without changing the reconstruction result. If $S$ however is a more complicated tree, then the specific group $D_{x^{\sep}/x}$ is needed to reconstruct the relative poset structure. %

This difference is the reason for the different coverings in \cref{fig:Covering}, as the local groups are all conjugate, but the coverings are different. For instance, we can decorate our tree with the groups $D_{x}$ indicated by the $2$-cycles in \cref{fig:TreeWithGroups}. 

\begin{figure}[h]
\begin{center}
\scalebox{0.5}{
 \begin{tikzpicture}[line cap=round,line join=round,x=1cm,y=1cm]
\clip(0,-0.5) rectangle (10.5,8.5);
\draw [line width=1.5pt] (4,3)-- (2,5);
\draw [line width=1.5pt] (4,3)-- (1,2);
\draw [line width=1.5pt] (4,3)-- (7,3);
\draw [line width=1.5pt] (7,3)-- (10,3);
\begin{scriptsize}
\draw [fill=qqqqff] (4,3) circle (5pt);
\draw [fill=qqccqq] (2,5) circle (5pt);
\draw[color=black] (2,5.8) node {\Large{$(12)$}};
\draw [fill=ffqqtt] (1,2) circle (5pt);
\draw[color=black] (1,2.8) node {\Large{$(23)$}};
\draw [fill=qqqqff] (7,3) circle (5pt);
\draw[color=black] (7,3.8) node {\Large{$(12)$}};
\draw [fill=qqqqff] (10,3) circle (5pt);
\draw[color=black] (10,3.8) node {\Large{$(23)$}};
\end{scriptsize}
\end{tikzpicture}
}
\end{center}
\caption{\label{fig:TreeWithGroups}A tree together with a set of decomposition groups. The local structure of the induced finite covering only depends on the conjugacy class of the group. The global structure of the finite covering depends heavily on the conjugacy class however, see \cref{fig:Covering}.}
\end{figure}
Here the edges are assumed to have trivial groups. The vertex without a label also has the trivial group. 
 The double coset spaces $D_{x}\backslash S_{3}/\langle(12)\rangle$ then glue to give the first covering in \cref{fig:Covering}. However, if we change the group $D_{x}=\langle(23)\rangle$ on the right vertex to $D_{x}=\langle(12)\rangle$, then we obtain the second covering. Note here that the local behavior of the covering does not change, since we have just moved to a different conjugacy class.  All of the coverings in \cref{fig:Covering} occur as $3:1$ coverings of the projective line by the results in \cite{ABBR2015} (see \cite{Helminck2024} for an interpretation in terms of Galois categories).  %
In our algorithms we reconstruct the local diagrams in %
\cref{pro:RecoveringTrees} rather than the local groups, which are enough to reconstruct poset structures. %
\end{remark}

We conclude this section with a set of examples to show that the results above are sharp. 

\begin{example}\label{exa:CounterexampleExtensions}
In general, the map $H\bs G/D_{e^{\sep}/e}\to E_{e}$ is not a bijection. %
Consider for instance the ring $B=\mathbb{Q}[X_{1},X_{(12)},X_{(13)},X_{(23)},X_{(123)},X_{(132)},W]$, where the subscripts are interpreted as elements of the group $G=S_{3}$. We endow $B$ with the standard grading $\mathrm{deg}(X_{\sigma})=\mathrm{deg}(W)=1$. 
 The group $G$ acts on $B$ through $\sigma(X_{\tau})=X_{\sigma\tau}$ and $\sigma(W)=W$. Define $\overline{X}=\mathrm{Proj}(B)=\mathbb{P}^{6}_{\Q}$. We then also have an action of $G$ on $\overline{X}$, with quotient map $\overline{X}\to X$. 
 Set $H=\angles{(12)}$ and 
$N=\angles{(123)}$, and let 
 \begin{align*}
 \mathfrak{q}_{1}&=(X_{1}+X_{(12)}),\\
 \mathfrak{q}_{2}&=(X_{1}+X_{(12)},X_{(123)}+X_{(13)},X_{(132)}+X_{(23)},X_{1}X_{(123)}X_{(132)}+W^3),
 \end{align*}
so that $\mathfrak{q}_{1}\subset\mathfrak{q}_{2}$. These are homogeneous prime ideals, corresponding to points $x_{1}$ and $x_{2}$ in $\overline{X}$. 
Let $y_{1}$ and $y_{2}$ be the images of $x_{1}$ and $x_{2}$ in $X$, and let $e_{x}=x_{1}x_{2}$ and $e_{y}=y_{1}y_{2}$. We then have
\begin{align*}
D_{x_{1}/y_{1}}&=H,\\
D_{x_{2}/y_{2}}&=N,\\
D_{e_{x}/e_{y}}&=(1).
\end{align*}
We now consider the different extensions of the points in $X$ in $X_{H}=\overline{X}/H$. Note that there are two points lying over $y_{1}$ as $|D_{x_{1}/y_{1}}\bs G/H|=2$, and only one lying over $y_{2}$ since $|D_{x_{2}/y_{2}}\bs G/H|=1$. There are thus two extensions of $e_{y}$, even though 
$|D_{e_{x}/e_{y}}\bs G/H|=3$. 

We can also use this example to construct extensions of triples on which the action of $G$ is not transitive, so that the conditions of \cref{lem:ReconstructChains2} are sharp.
Consider the homogeneous prime ideal 
$$I=(X_{1},
    X_{(123)} + X_{(13)},
    X_{(132)} + X_{(23)},
    X_{(12)},
    X_{(13)}^2 + X_{(23)}^2, W)$$
    with corresponding points $x_{3}\in \overline{X}$ and $y_{3}\in{X}$. Note that $x_{1}\leq x_{2}\leq x_{3}$. %
The stabilizer of $x_{3}$ is easily seen to be $H$. %
The map of posets can be found in \cref{fig:NonTransitive}. Note that there are $18$ extensions of the triple $y_{0}\leq y_{1}\leq y_{2}$, so that the action is not transitive. This also implies that the action of $G=\mathrm{Gal}(L^{\sep}/L)$ is not transitive on extensions in $X^{\sep}$ of this triple.    
\begin{figure}[ht]
\scalebox{0.4}{
\begin{tikzpicture}[line cap=round,line join=round,>=triangle 45,x=1cm,y=1cm]
\clip(4,1.46) rectangle (14,14.14);
\draw [line width=2.4pt] (6,13)-- (9,11);
\draw [line width=2.4pt] (6,7)-- (9,9);
\draw [line width=2.4pt] (6,10)-- (9,11);
\draw [line width=2.4pt] (6,10)-- (9,9);
\draw [line width=2.4pt] (9,11)-- (12,13);
\draw [line width=2.4pt] (9,9)-- (12,7);
\draw [line width=2.4pt] (9,11)-- (12,10);
\draw [line width=2.4pt] (9,9)-- (12,10);
\draw [line width=2.4pt,loosely dashed] (6,13)-- (9,9);
\draw [line width=2.4pt,loosely dashed] (6,7)-- (9,11);
\draw [line width=2.4pt,loosely dashed] (12,13)-- (9,9);
\draw [line width=2.4pt,loosely dashed] (12,7)-- (9,11);
\draw [line width=2.4pt] (6,4)-- (9,4);
\draw [line width=2.4pt] (9,4)-- (12,4);
\draw [->,line width=2pt] (9,7.5) -- (9,5.5);
\begin{scriptsize}
\draw [fill=ududff] (6,13) circle (5pt);
\draw [fill=ududff] (9,11) circle (5pt);
\draw [fill=ududff] (6,7) circle (5pt);
\draw [fill=ududff] (9,9) circle (5pt);
\draw [fill=ududff] (6,10) circle (5pt);
\draw [fill=ududff] (12,13) circle (5pt);
\draw [fill=ududff] (12,7) circle (5pt);
\draw [fill=ududff] (12,10) circle (5pt);
\draw [fill=ududff] (6,4) circle (5pt);
\draw [fill=ududff] (9,4) circle (5pt);
\draw [fill=ududff] (12,4) circle (5pt);
\end{scriptsize}
\end{tikzpicture}
}
\caption{\label{fig:NonTransitive}
The map of Hasse diagrams in \cref{exa:CounterexampleExtensions}. There is an action of $S_{3}$ on the top graph, and this action is not transitive on pairs of horizontal edges, corresponding to chains of length two. 
}
\end{figure}    
\end{example}

\subsection{Reconstructing posets}\label{sec:FiberPosets}

Let $S_{HD,B}$ be the barycentric subdivision of the Hasse diagram $S_{HD}$ of $S$, and let $\mathcal{G}$ be the category of profinite groups.  
We now extend the results from the previous section to more general subgraphs %
using a group-theoretic enhancement
\begin{equation*}
\mathcal{F}: S_{HD,B}\to \mathcal{G},
\end{equation*}
see  \cref{def:GroupTheoreticEnhancement}.   %
We first fix a spanning tree $S_{T,HD}$ of the Hasse diagram $S_{HD}$ of $S$, together with an absolute section $S^{\sep}_{T,HD}$ of $S_{T,HD}$. We set $\mathcal{F}(x)=D_{x^{\sep}/x}$ for %
every corresponding vertex in %
the barycentric subdivision $S_{HD,B}$. For every edge $e\in E(S_{T,HD})$, we set $\mathcal{F}(e)=D_{e^{\sep}/e}$, %
where $e^{\sep}$ is the corresponding section of $e$ in $S^{\sep}_{T,HD}$. The corresponding homomorphisms are %
the natural ones induced from the inclusions $D_{e^{\sep}/e}\subset D_{x^{\sep}/x},D_{y^{\sep}/y}$.  %

Let $e=xy\in E(S_{HD})\bs E(S_{T,HD})$, and let $x^{\sep}$ and $y^{\sep}$ be the two points lying over $x$ and $y$ in $S^{\sep}_{T,HD}$. Note that these do not necessarily form an edge in $S^{\sep}$. %
Let $e'=x^{\sep}y'^{\sep}$ be an extension of $e$ starting at $x^{\sep}$. By \cref{lem:TransitiveOnChains}, we can find an element in $D_{x^{\sep}/x}$ such that $\sigma(y'^{\sep})=y^{\sep}$. From this, we obtain the equality $\sigma D_{y'^{\sep}/y}\sigma^{-1}=D_{y^{\sep}/y}$, which we interpret as an isomorphism  %
$\gamma_{\sigma}: D_{y'^{\sep}/y}\to D_{y^{\sep}/y}$. 
We set $\mathcal{F}(e)=D_{e'^{\sep}/e}$ %
for every vertex $e$ in $S_{HD,B}$ corresponding to $e\in E(S_{HD})\bs E(S_{T,HD})$. We have an inclusion %
$D_{x^{\sep}/x}\supset D_{e'^{\sep}/e},$ which gives the first map of profinite groups for the functor $\mathcal{F}$. %
The other map $D_{e'^{\sep}/e}\to D_{y^{\sep}/y}$ %
is then obtained by composing the inclusion $D_{e'^{\sep}/e}\subset D_{y'^{\sep}/y}$ with %
the conjugation isomorphism $\gamma_{\sigma}:D_{y'^{\sep}/y}\to D_{y^{\sep}/y}$. %
By doing this for every missing edge in $E(S_{HD})\bs E(S_{T,HD})$, we obtain the desired functor $\mathcal{F}: S_{HD,B}\to \mathcal{G}$. 
\begin{definition}
The functor $\mathcal{F}:S_{HD,B}$ contructed above is a group-theoretic enhancement of $S\subset X$. We call $\mathcal{C}(\mathcal{F})$ %
the glued category associated to this enhancement. %
\end{definition}

Recall that an object in $\mathcal{C}(\mathcal{F})$ consists of a collection of topological spaces $X_{i}$ with a continuous action of the group $G_{i}=\mathcal{F}(i)$, and a set of gluing data for adjacent spaces.   %
We now construct a universal covering, which will be %
a single object in the category $\mathcal{C}(\mathcal{F})$. To construct the various transition maps in the $2$-limit, we first need a simple lemma whose proof we leave to the reader. 
\begin{lemma}\label{lem:EquivarianceTransfer}
Let $D,H\subset {G}$ be two subgroups and let $\sigma\in {G}$. 
Consider the map $\phi_{\sigma}: G/H\to G/H$ given by $\tau{H}\mapsto \sigma\tau{H}$, and the left-actions 
\begin{align*}
s_{i}:D\to \mathrm{Aut}(G/H)
\end{align*}
for $i=1,2$ given by $s_{1}(d,gH)=dgH$ and %
$s_{2}(d,g)=\sigma d \sigma^{-1}gH$.  
Endow the first copy of $G/H$ in $\phi_{\sigma}$ with the $D$-structure induced by $s_{1}$ and the second with the $D$-structure induced by $s_{2}$.  %
Then $\phi_{\sigma}$ is $D$-equivariant.  
\end{lemma}

We can now define the universal poset covering as an object in $\mathcal{C}(\mathcal{F})$. %
\begin{definition}\label{def:UniversalPosetCovering}
The universal poset covering is the following object of $\mathcal{C}(\mathcal{F})$. For every vertex $x$ of $S_{HD,B}$, the corresponding decomposition group $D_{x}$ acts continuously on the topological group $G$ on the left, so that we can consider $G$ as an object of the corresponding category $D_{x}\mhyphen\text{Spaces}$. For the transfer maps, we do the following.    
For edges in $S_{HD,B}$ arising from $S_{T,HD}$, we take the identity. %
For an edge $e=xy$ not in $E(S_{T,HD})$, we take the identity on the directed edge $e_{x}=(x,e)$ in $E(S_{HD,B})$. For the directed edge $e_{y}=(y,e)$, we apply the $D_{e'^{\sep}/e}$-equivariant isomorphism from %
\cref{lem:EquivarianceTransfer} to $G$ with $\sigma$ as in the definition of $e'^{\sep
}=x^{\sep}y'^{\sep}$, so that  %
$\sigma(y'^{\sep})=y^{\sep}$. %
We denote the corresponding universal poset covering by $S_{univ}\in \text{Ob}(\mathcal{C}(\mathcal{F}))$.  By replacing $G$ by $G/H$ for a closed subgroup $H$, we then also obtain the object $S_{H}$ corresponding to $H$. %
\end{definition} 

\begin{remark}
Note that for a given edge $e$ or tree $S_{T,HD}$ inside $S_{HD}$, we can always trivialize the isomorphisms in  \cref{def:UniversalPosetCovering} and view the data as a group $G$ together with subgroups $D_{x^{\sep}/x},D_{e^{\sep}/e}\subset G$ acting by left multiplication. This is for instance the point of view taken in \cite{Helminck2024} on the level of Berkovich spaces, where a continuous version for coverings of curves is studied. %
It is only when we consider loops inside the Hasse diagram that we have to be careful about gluing the different actions. 
\end{remark}

We now define a geometric realization of $S_{univ}$, which can be seen as the underlying poset of the object in \cref{def:UniversalPosetCovering}. By construction, we have an object $G$ over every vertex of $S_{HD,B}$, and we can compare the actions of the decomposition groups on $G$ over neighboring vertices. In comparing the actions of these decomposition groups, we emphasize that one first has to apply the transfer maps from \cref{lem:EquivarianceTransfer}.   
\begin{definition}\label{def:GeometricRealization}
Consider the set $|S_{univ}|=\bigsqcup_{x\in S}D_{x^{\sep}/x}\bs G$. We endow $|S_{univ}|$ with a poset structure by describing its Hasse diagram. %
Let $e=xy$ be an edge in the Hasse diagram of $S$, and let $\tau_{x}$ and $\tau_{y}$ be representatives of $D_{x^{\sep}/x}$-orbits and $D_{y^{\sep}/y}$-orbits in $G$ respectively. %
There is an edge between the $D_{x^{\sep}/x}$-orbit of $\tau_{x}$ and the $D_{y^{\sep}/y}$-orbit of $\tau_{y}$ if and only if there exists a $\tau_{e}\in{G}$ such that the $D_{x^{\sep}/x}$-orbit of $\tau_{e}$ is the $D_{x^{\sep}/x}$-orbit of $\tau_{x}$, and the $D_{y^{\sep}/y}$-orbit of $\tau_{e}$ is the $D_{y^{\sep}/y}$-orbit of $\tau_{y}$. 
One easily sees that this is independent of the chosen representatives. Note that the actions of $D_{x^{\sep}/x}$ and $D_{y^{\sep}/y}$ are calculated after applying the transfer maps.

More generally, let $H$ be a closed subgroup of $G$. %
We then similarly define the geometric realization $|S_{H}|$ of $S_{H}$ by replacing $G$ by $G/H$. The elements of this poset are $D_{x^{\sep}/x} \bs (G/H)$,  %
and the relations are as follows: there is an edge between the $D_{x^{\sep}/x}$-orbit of $\tau_{x}H$ and the $D_{y^{\sep}/y}$-orbit of $\tau_{y}H$ if and only if there exists a $\tau_{e}\in{G}$ such that the $D_{x^{\sep}/x}$-orbit of $\tau_{e}H$ is the $D_{x^{\sep}/x}$-orbit of $\tau_{x}H$, and the $D_{y^{\sep}/y}$-orbit of $\tau_{e}H$ is the $D_{y^{\sep}/y}$-orbit of $\tau_{y}H$.  
Equivalently, we have the equalities of double cosets 
$D_{x^{\sep}/x}\tau_{e}H=D_{x^{\sep}/x}\tau_{x}H$ and 
$D_{y^{\sep}/y}\tau_{e}H=D_{y^{\sep}/y}\tau_{y}H$. 
As before, we calculate the action of the decomposition groups after applying the transfer maps.  %
\end{definition} 

\begin{remark}\label{rem:RephraseReconstruction}
From the equalities $D_{x^{\sep}/x}\tau_{e}H=D_{x^{\sep}/x}\tau_{x}H$ and 
$D_{y^{\sep}/y}\tau_{e}H=D_{y^{\sep}/y}\tau_{y}H$ in \cref{def:GeometricRealization}, one obtains the equations 
$d_{x}\tau_{e}{h_{x}}=\tau_{x}$ and 
$d_{y}\tau_{e}h_{y}=\tau_{y}$.
By rewriting this, we obtain the single equality $d_{x}^{-1}\tau_{x}h_{x}^{-1}=d_{y}^{-1}\tau_{y}h_{y}^{-1}$, and thus $\tau_{x}=d_{x}d_{y}^{-1}\tau_{y}h_{y}^{-1}h_{x}$. One might be tempted to take the latter as the defining equation for an edge over $e$, in analogy with the equations arising from double coset spaces. That is, there is an edge between the orbits of $\tau_{x}H$ and $\tau_{y}H$ if and only if there exist $d_{x}\in D_{x^{\sep}/x}$, $d_{y}\in D_{y^{\sep}/y}$ and $h\in{H}$ such that $\tau_{x}=d_{x}d_{y}^{-1}\tau_{y}h$. We note however that this in general does {not} define an equivalence relation, due to the possible non-normality of the decomposition groups.   
\end{remark}

\begin{remark}
If we interpret $|S_{univ}|$ as the set of right coset classes of $D_{x^{\sep}/x}$ in $G$, then $H$ acts on the right.  The poset $|S_{H}|$ is then the natural quotient of $|S_{univ}|$ under this action. Alternatively, we could also have worked with the set of left-coset classes $\bigsqcup_{x\in S} G/D_{x^{\sep}/x}$ in \cref{def:GeometricRealization}. The subgroup $H$ acts on the left on these classes and %
we again obtain $|S_{H}|$ as a quotient of $|S_{univ}|$. 
\end{remark}

The following is the discrete poset version of \cite[Theorem 2.17]{Helminck2024} for Berkovich spaces.

\begin{theorem}\label{thm:MainThmv2}
Let $H$ be a closed subgroup of $G$ with map of normalizations $\phi_{H}: X_{H}\to X$. There is an isomorphism $|S_{H}|\to \phi^{-1}_{H}(S)$ of partially ordered sets. For an inclusion of closed subgroups $H_{1}\to H_{2}$, we have a commutative diagram 
\begin{equation*}
\begin{tikzcd}
{|S_{H_{1}}|} \arrow[d] \arrow[r] & {\phi^{-1}_{H_{1}}(S)} \arrow[d] \\
{|S_{H_{2}}|}  \arrow[r]           & {\phi^{-1}_{H_{2}}(S)}    
\end{tikzcd}
\end{equation*}
\end{theorem}
\begin{proof}
We note that throughout the proof, we use $\phi_{H}:X_{H}\to X$ for the morphism to $X$, and $\phi^{H}:X^{\sep}\to X_{H}$ for the morphism from $X^{\sep}$. 
By \cref{rem:RephraseReconstruction} and 
\cref{pro:RecoveringTrees}, it follows that the data over $S_{T,HD}$ recovers the structure of the graph $\phi_{H}^{-1}(S_{T,HD})$. We have to check that this procedure also recovers the remaining parts of $\phi^{-1}_{H}(S)$.  %
We recall that $\sigma(y'^{\sep})=y^{\sep}$. Consider the diagram
\begin{equation*}
\begin{tikzcd}
{D_{y'^{\sep}/y}\bs G/H}%
\arrow[d] \arrow[r] & {{D_{y^{\sep}/y}\bs G/H}}%
\arrow[d] \\
{\phi_{H}^{-1}(y)} \arrow[r]           & {\phi_{H}^{-1}(y)}          
\end{tikzcd},
\end{equation*}
where the top horizontal arrow is induced from the map $G/H\to G/H$ defined by $\tau{H}\mapsto \sigma\tau{H}$ as in \cref{lem:EquivarianceTransfer}, the bottom horizontal arrow is the identity and the vertical arrows are induced from \cref{lem:RecoveringPoints}. We claim that this diagram is commutative. The second vertical arrow sends the $D_{y^{\sep}/y}$-orbit of $\tau$ to $\phi^{H}(\tau^{-1}(y^{\sep}))$ and the first vertical arrow sends the $D_{y'^{\sep}/y}\tau{H}$ to $\phi^{H}(\tau^{-1}(y'^{\sep}))$. %
We then find that $D_{y'^{\sep}/y}\sigma\tau{H}$ is sent to $\phi^{H}(\tau^{-1}\sigma^{-1}(y^{\sep}))=\phi^{H}(\tau^{-1}(y'^{\sep}))$, as desired. It now follows from this that the data in the construction of $|S_{H}|$ (see \cref{def:GeometricRealization}) recovers the double coset diagrams corresponding to $D_{x^{\sep}/x}$ and $D_{y'^{\sep}/y}$, which exactly give the different edges over $e=xy$ by \cref{lem:RecoveringEdges} since the local section is given by $e'^{\sep}=x^{\sep}y'^{\sep}$. The compatibilities in the theorem are  immediate.     
\end{proof}

\begin{remark}
Note that if we instead used the original sets without the map from \cref{lem:EquivarianceTransfer}, then the diagram in the proof of \cref{thm:MainThmv2} would not commute. 
\end{remark}

\section{Power series algorithms}\label{sec:Algorithms}

In this section we explore various algorithmic aspects of \cref{thm:MainThmv2}. We first show how one can interpret the group-theoretic data from the previous sections in terms of roots of polynomials and their approximations in \cref{sec:FromTopoiToPowerSeries}. We review power series over regular local rings in \cref{sec:PowerSeries}, and in \cref{sec:NPAlgorithm} we give a symbolic Newton-Puiseux algorithm to calculate power series expansions of the roots of a polynomial $f(z)$ that splits completely over the strict Henselization of $A=K[x_{1},...,x_{n}]$ with respect to $(x_{1},...,x_{n})$. These algorithms will also work over arithmetic rings such as $\mathcal{O}_{K,\mathfrak{p}}[x_{1},...,x_{n}]$ and $\mathcal{O}_{K,\mathfrak{p}}[x_{0},...,x_{n}]/(x_{0}\cdots x_{r}-\pi^{k})$, where $\mathcal{O}_{K,\mathfrak{p}}$ is the localization of the ring of integers $\mathcal{O}_{K}$ of a number field $K$ at a prime $\mathfrak{p}$, or a completion thereof. %

\subsection{From group theory to power series}\label{sec:FromTopoiToPowerSeries}
A finite covering $\phi: X'\to X$ of normal connected Noetherian schemes corresponds to a finite separable extension $L=K(X)\to K(X')=L'$. %
We can represent this finite extension using an irreducible separable polynomial $f(z)\in L[z]$ through an   isomorphism $L'\simeq L[z]/(f(z))$. We fix this isomorphism throughout this section. Let $L^{\sep}$ be a separable closure of $L$. As in \cref{lem:RecoveringPoints}, we will fix a specific embedding of $L'$ into $L^{\sep}$. For a polynomial $f(z)$ as above, we can consider the set of roots $V(f(z))\subset L^{\sep}$ inside a separable closure $L^{\sep}$ of $L$.  %
Our choice of an
embedding of $L'$ into $L^{\sep}$ then gives a specific root $\alpha\in L^{\sep}$ of $f(z)$. We will denote this specific root by $\alpha_{1}$ sometimes.  %
Note that the decomposition groups $D_{x^{\sep}/x}$ for $x\in{X}$ act on the zero set $V(f(z))$. We have the following characterization of the fibers of $X'\to X$ in terms of $V(f(z))$.      

\begin{lemma}\label{lem:RootsDecomposition}
Let $\phi:X'\to X$ be a finite covering with embedding of function fields $L'\to L^{\sep}$, generating polynomial $f(z)\in L[z]$, subgroup $H$, specified root $\alpha$ and zero set $V(f(z))\subset L^{\sep}$, as above.  %
There is a bijection $D_{x^{\sep}/x}\bs V(f(z))\simeq \phi^{-1}(x)$. Explicitly, we write $\alpha_{i}=\tau(\alpha)$ for some $\tau\in G$, and we send this root to $\phi_{H}(\tau^{-1}(x^{\sep}))$.  %
\end{lemma}
\begin{proof}
We can send $\tau(\alpha)$ to $\tau{H}$ to obtain a bijection $V(f(z))\to G/H$. Note that this is independent of our choice of $\tau$, as any two automorphisms $\tau_{1}$ and $\tau_{2}$ with $\tau_{1}(\alpha)=\tau_{2}(\alpha)$ differ by an element of $H$ (namely, $\tau^{-1}_{2}\tau_{1}$). Together with \cref{lem:RecoveringPoints}, this directly gives the desired statement.  %
\end{proof}

In other words, in order to find the different extensions of $x$, we only have to find the orbits of the roots of $f(z)$ under the action of a decomposition group $D_{x^{\sep}/x}$. To find this, we can do the following. Let $M\supset L$ be a finite extension with map of normalizations $X_{M}\to X$. %
Suppose that $f(z)$ splits completely over $M^{sh}$, which is the fraction field of the strict Henselization of a point $x_{M}$ lying over $x$\footnote{In practice, we can for instance use Hensel's lemma to conclude that the polynomial splits completely over the strict Henselization.  }. We then have to find the action of $\mathrm{Gal}(M^{sh}/L^{h})$ on $V(f(z))$. 
For morphisms of semistable models, it suffices to take Kummer extensions for $M$, see \cref{lem:KummerExtensions}.  %

To calculate the action of $G_{0}=\mathrm{Gal}(M^{sh}/L^{h})$, it suffices to calculate the action of this group on sufficiently precise finite approximations $\gamma_{i}$ of the roots $\alpha_{i}$. We explain this in more detail here. 
\begin{definition}[Separating approximations]\label{def:SeparatingApproximations}
Let $f(z)\in K(X)[z]$ be a separable irreducible polynomial that generates a field extension $K(X)\to K(X')$ with map of normalizations $X'\to X$, and fix a point $x\in{X}$. Let $M\supset K(X)$ be a fixed finite extension with map of normalizations $X_{M}\to X$ as above. 
Let $A$ be the strict Henselization of a point $x_{M}$ of $X_{M}$ lying over $x$ with maximal ideal $\mathfrak{m}$. We will assume without loss of generality that the $\alpha_{i}$ already lie in $A$, rather than the fraction field. Let $I\subset A$ be a nonzero ideal. We say that $\gamma_{i}\in{A}$ is an $I$-approximation of height $k\in\mathbb{N}$ of a root $\alpha_{i}$ if
\begin{equation*}
\alpha_{i}-\gamma_{i}\in{I^{k}}. %
\end{equation*}
Let $S\subset A$ be a finite set. We say that $S$ is an $I$-approximation set of height $k$ for $f(z)$ if for every root $\alpha_{i}$, there is a $\gamma_{i}$ such that $\gamma_{i}$ is an $I$-approximation of height $k$.  
Let $S=\{\gamma_{1},...,\gamma_{n}\}$ be an $I$-approximation set of height $k$ for $f(z)$. %
We say that $k$ is a separating height if the images of the $\gamma_{i}$ in $A/I^{k}$ are distinct. In particular, the $\gamma_{i}$ are all distinct. In this case we call $S$ a separating approximation set of height $k$.  
\end{definition}
\begin{remark}
One can take $\alpha_{i}=\gamma_{i}$ in this definition, but this is of course not desirable. We will see in the next section how we can construct suitable $\gamma_{i}$ for regular local rings.  
\end{remark}
Let $G$ be a group acting on $A$, and let $I$ be an ideal that is $G$-invariant, so that $\sigma(I)=I$. This implies that $\sigma(I^{m})=I^{m}$ for any $m\geq{1}$.  
Let $S$ be an $I$-separating approximation set of height $k$ for a polynomial $f(z)\in K(X)[z]$. We define an action of $G$ on these approximations as follows. Suppose that 
$$\alpha_{i}-\gamma_{i}\in I^{k}.$$
Applying $\sigma\in G$ to this and noting that $\sigma(I)=I$, we then find 
\begin{equation*}
\sigma(\alpha_{i})%
-\sigma(\gamma_{i})\in I^{k}. %
\end{equation*}  
We then know that $\sigma(\gamma_{i})-\gamma_{\sigma(i)}$ for a unique $\sigma(i)\in[n]=\{1,...,n\}$ by our definition of a separating approximation set. %
We then send $\gamma_{i}$ to $\gamma_{\sigma(i)}$. We now see that we have an isomorphism of $G$-sets %
\begin{equation*}
\{\alpha_{1},...,\alpha_{n}\}\to \{\gamma_{1},...,\gamma_{n}\}.
\end{equation*} 

Suppose that $G=G_{0}$, as before. To calculate the $D_{x^{\sep}/x}$-orbits of the $\alpha_{i}$, we now see that we have to calculate the $G$-orbits of the $\gamma_{i}$. %
Thus, to calculate the data in \cref{fig:Correspondences}, it suffices to calculate the action on a separating approximation set. Explicitly calculating this action can be very hard, as it requires the calculation of a local Galois group. For our applications however, one does not need the full group, as it suffices to know whether two approximations are in the same orbit. This usually only requires a single irreducibility check. We will focus more on these issues for semistable models in \cref{sec:FindingGaloisOrbits}. We will see in \cref{sec:PowerSeries} how power series expansions can be used to make the approximations $\gamma_{i}$ more explicit when $A$ is a regular local ring. In our applications to semistable models, we can always reduce to this case.  

\subsubsection{Relative approximations}

We now discuss the relative version of the above. %
Let $x\geq{y}$ be two points in $X$. We assume as always that $X$ is relatively unibranch with respect to these. We fix our notation as in \cref{def:SeparatingApproximations} and let $x_{M}\geq y_{M}$ be an extension of this chain to $X_{M}$. Let $A$ be the strict Henselization of $X_{M}$ at $x_{M}$. We write $\mathfrak{m}$ and $\mathfrak{p}$ for the two prime ideals in $A$ corresponding to $x_{M}$ and $y_{M}$ respectively. 
\begin{definition}\label{def:CompatibleSeparatingSets}
Let $S=\{\alpha_{1},...,\alpha_{n}\}\subset A$ be the roots of a polynomial $f(z)$.  
Let $\gamma_{i,\mathfrak{p}}$ be a separating set of $\mathfrak{p}$-approximations for $S$ of height $k$, and let %
$\gamma_{i,\mathfrak{m}}$ be a separating set of $\mathfrak{m}$-adic approximations for $S$ of height $k$. We assume that the labeling is preserved, so that $\gamma_{i,\mathfrak{p}}-\gamma_{i,\mathfrak{m}}\in\mathfrak{m}^{k}$. We call this a compatible separating set of $(\mathfrak{m},\mathfrak{p})$-approximations of height $k$.     %
\end{definition}

To calculate the diagram of orbits as in \cref{pro:RecoveringTrees}, let $\sigma\in D_{e^{\sep}/e}$. Here $e^{\sep}$ is an edge in $X^{\sep}$ that maps to $e_{M}=x_{M}y_{M}$. Note that both $\mathfrak{p}$ and $\mathfrak{m}$ are invariant with respect to this action. We can now construct the diagram in \cref{pro:RecoveringTrees} as follows. Let $(\gamma_{i,\mathfrak{p}},\gamma_{i,\mathfrak{m}})$ be a compatible separating set of $(\mathfrak{m},\mathfrak{p})$-approximations of height $k$ (in principle, we could again take the $\alpha_{i}$ here, but this is not desirable in practice). For $D_{e}\backslash G/H$, we the $D_{e^{\sep}/e}$-sets $S_{\mathfrak{p}}=\{\gamma_{i,\mathfrak{p}}\}$ and $S_{\mathfrak{m}}=\{\gamma_{i,\mathfrak{m}}\}$, with the isomorphism coming from \cref{def:CompatibleSeparatingSets}. We now consider $S_{\mathfrak{p}}$ as a $D_{y^{\sep}/y}$-set, and $S_{\mathfrak{m}}$ as a $D_{x^{\sep}/x}$-set. This diagram then automatically reconstructs the poset structure over $e$ by \cref{pro:RecoveringTrees}.

\begin{remark}
For coverings of semistable models, we will have that the action of $D_{e^{\sep}/e}$ is the same as the action of $D_{x^{\sep}/x}$, see \cref{lem:KummerOrbitRemark}. %
\end{remark}  

\subsection{Power series in regular local rings}\label{sec:PowerSeries}

The output of our algorithms in the upcoming sections will be certain power series approximations of the roots of a polynomial. We review these notions here for general regular local Noetherian rings.   
Let $(A,\mathfrak{m})$ be a regular local Noetherian ring with residue field $k:=A/\mathfrak{m}$. %
Suppose that we are given a set of elements $x_{1},...,x_{r}\in\mathfrak{m}$ whose images in the $k$-vector space %
$\mathfrak{m}/\mathfrak{m}^{2}$ form a basis of that vector space. %
Here $r=\mathrm{dim}(A)$ by regularity. This basis can be used to give an isomorphism between the graded ring $\mathrm{gr}_{\mathfrak{m}}(A)$ and the polynomial ring over $k$ in $r$ variables, see \cite[\href{https://stacks.math.columbia.edu/tag/00NO}{Lemma 00NO}]{stacks-project}. We will interpret this result using the concept of monomials. %
\begin{definition}\label{MonomialElements}
Let $(A,\mathfrak{m})$ be a regular local Noetherian ring and let $\{x_{1},...,x_{r}\}$ be a set of generators for $\mathfrak{m}$. %
An element $m\in{A}$ of the form
\begin{equation*}
m=x_{1}^{i_{1}}\cdots{x_{r}^{i_{r}}}
\end{equation*}
for $i_{1},...,i_{r}\in\mathbb{N}$ 
is a {\it{monomial}} in the $x_{i}$ of total degree $\mathrm{deg}(m):=i_{1}+...+i_{r}$. We will also think of these monomials as being represented by vectors in $\mathbb{N}^{r}$. That is, every monomial $x_{1}^{i_{1}}\cdots{x_{r}^{i_{r}}}$ corresponds to the vector $(i_{1},...,i_{r})$. %
\end{definition}
The set of all monomials in $A$ is denoted by $M$. For every $j\in\mathbb{N}$, we can consider the set of monomials of total degree $j$: $M_{j}=\{m\in{M}:\mathrm{deg}(m)=j\}$. These monomials are elements of the ideal $\mathfrak{m}^{j}$. Moreover, using the isomorphism $\text{gr}_{\mathfrak{m}}(A)\simeq k[z_{1},...,z_{r}]$ one easily sees that they form a basis of the quotient $\mathfrak{m}^{j}/\mathfrak{m}^{j+1}$.   %
\begin{lemma}\label{BasisVectorSpaces}
Consider the set $M_{j}$ of all monomials in the $x_{i}$ of total degree $j$ and let $\overline{M}_{j}$ be the image of this set in the $k$-vector space $\mathfrak{m}^{j}/\mathfrak{m}^{j+1}$. Then $\overline{M}_{j}$ is a basis for this vector space.  %
\end{lemma}

For the upcoming proposition, recall that a set of representatives $S$ for the residue field $k:=A/\mathfrak{m}$ is a set of elements $S$ in $A$ such that the restriction of the quotient map $A\rightarrow{A/\mathfrak{m}}$ to $S$ yields a bijection.  %
In other words, for every element of the residue field, we have chosen exactly one representative in $A$. We will assume that $0\in{S}$, so that it represents $\overline{0}$. We will view $S$ as a set-theoretic section of the quotient map $A\to A/\mathfrak{m}$. 

We now recall that we %
can write down $\mathfrak{m}$-adic power series in a regular local Noetherian ring using the $\mathfrak{m}$-adic topology on $A$. 
Namely, for a fixed set of generating monomials $\{x_{1},...,x_{r}\}$ for $A$, %
we write %
\begin{equation*}%
z=\sum_{{i}\in\mathbb{N}^{r}}c_{i}\cdot{}x_{1}^{i_{1}}\cdots{x_{r}^{i_{r}}}
\end{equation*}
for $z\in{A}$, ${i}=(i_{1},...,i_{r})\in\mathbb{N}^{r}$ and $c_{{i}}\in{S}$ if the sequence %
$(z_{n})$ defined by taking all terms of total degree less than $n$ converges to $z$. For every $z\in{A}$, we can find a unique power series expansion of this form by the following proposition. 
\begin{proposition}\label{prop:RegularExpansions}
Let $(A,\mathfrak{m})$ be a regular Noetherian local ring. %
 Let $x_{1},...,x_{r}\in{\mathfrak{m}}$ be elements that map to a basis of the $k$-vector space $\mathfrak{m}/\mathfrak{m}^{2}$ and let $S\subset{A}$ be a set of representatives of the residue field $A/\mathfrak{m}$. Then every element $z\in{A}$ can be uniquely written as
\begin{equation*}
z=\sum_{{i}\in\mathbb{N}^{r}}c_{{i}}\cdot{}x_{1}^{i_{1}}\cdots{x_{r}^{i_{r}}},
\end{equation*}
where ${i}=(i_{1},...,i_{r})$ and $c_{{i}}\in{S}$. %
\end{proposition}
\begin{proof}
The proof follows from Lemma \ref{BasisVectorSpaces} by set-theoretically splitting the exact sequences 
\begin{equation*}
0\rightarrow{\mathfrak{m}^{i+1}}\to\mathfrak{m}^{i}\to\mathfrak{m}^{i}/\mathfrak{m}^{i+1}\rightarrow{0}
\end{equation*}
using the given basis $\{x_{1},...,x_{r}\}$ and the section $A/\mathfrak{m}\to{A}$ provided by $S$. 
We leave the details to the reader. 
 
 \end{proof}

\begin{remark}
The proposition gives an injective map $A\to \text{Hom}(\mathbb{N}^{r},S)$. This map is generally not surjective, but it is if we take the $\mathfrak{m}$-adic completion of $A$. %
\end{remark} 

\begin{definition}\label{def:Truncations}
Let $$z=\sum_{{i}\in\mathbb{N}^{r}}c_{{i}}\cdot{}x_{1}^{i_{1}}\cdots{x_{r}^{i_{r}}}$$
be the unique power series representation of an element $z\in A$ obtained from \cref{prop:RegularExpansions}. %
 Let $n\in\mathbb{N}$. The $n$-th order truncation of $z$ is the polynomial %
\begin{equation*}
\mathrm{trun}_{n}(z)=\sum_{{i}\in\mathbb{N}^{r}}d_{{i}}\cdot{}x_{1}^{i_{1}}\cdots{x_{r}^{i_{r}}},
\end{equation*} 
where $d_{i}=c_{i}$ if $\mathrm{deg}(i)\leq{n}$, and $d_{i}=0$ if $\mathrm{deg}(i)>{n}$. %
We have $z-\text{trun}_{n}(z)\in\mathfrak{m}^{n+1}$. Let $f(z)\in A[z]$ be a separable polynomial that is completely split with roots $\alpha_{i}\in A$. We say that $n$ is a separating height for $f(z)$ if the $(n-1)$-truncations of the $\alpha_{i}$ are distinct.  
\end{definition}
\begin{remark}
Note that if $n$ is a separating height for $f(z)$, then the $n$-truncations $\gamma_{i}=\text{trun}_{n}(\alpha_{i})$ give a separating set of approximations for $f(z)$ of height $(n-1)$ in the sense of \cref{def:SeparatingApproximations}.  
\end{remark}

\subsection{A symbolic algorithm to calculate power series expansions}\label{sec:NPAlgorithm}

In this section we give a symbolic algorithm to calculate power series expansions over the strict Henselization of $A=K[x_{1},...,x_{n}]$ with respect to $\mathfrak{m}=(x_{1},...,x_{n})$. By a small modification, the algorithm will also work over arithmetic rings such as $A=\mathcal{O}_{K,\mathfrak{p}}[x_{1},...,x_{n}]$, where $\mathcal{O}_{K,\mathfrak{p}}$ is a localization of the ring of integers of a number field. We refer the reader to \cref{rem:NumberFieldAlgorithm} and \cref{rem:ModifiedAlgorithm} for more on these modifications.
\begin{remark}
The algorithm in this section will be presented in several steps. We first show that the coefficients of the $\mathfrak{p}$-adic power series expansions of the roots of $f(z)$ are defined over the strict Henselization of $B=K[x_{1},...,x_{n-1}]$. By iterating this procedure, we then find that the $\mathfrak{m}$-adic power series expansions of the roots $f(z)$ can be recovered by calculating succesive $\mathfrak{p}$-adic power series expansions. %
In the second part of this section, we show how to explicitly calculate these coefficients. We start by explaining the \'{e}tale and tropical subroutines, which form the heart of the main discrete Newton-Puiseux algorithm which can be found in \cref{sec:DNPAlgorithm}. The formal algorithms behind the subroutines can be found in \cref{sec:Algorithms2}. The multivariate Newton-Puiseux algorithm can be found in \cref{thm:MainThmNumber2}. An implementation of these can be found on the \href{https://paulhelminck.wordpress.com/code}{author's website}. We will often refer to this code here. %
We conclude this section with \cref{exa:MainExample1}.   
\end{remark}    

\subsubsection{Strict Henselizations}\label{sec:StrictHenselizations1}

We start with a general remark on strict Henselizations. Let $A$ be a Noetherian normal domain with corresponding scheme $X=\Spec(A)$ and prime ideals $\mathfrak{m}\supset \mathfrak{p}$. We write $A^{\sep}$ for the normalization of $A$ in $L^{\sep}$.  Then we obtain a morphism of localizations $A_{\mathfrak{m}}\to A_{\mathfrak{p}}$. We denote the strict Henselizations of $A$ with respect to $\mathfrak{m}$ and $\mathfrak{p}$ by $A^{\sh}_{\mathfrak{m}}$ and $A^{\sh}_{\mathfrak{p}}$ respectively.  We can then extend the morphism of localizations $A_{\mathfrak{m}}\to A_{\mathfrak{p}}$ to a morphism of strict Henselizations $A^{\sh}_{\mathfrak{m}}\to A^{\sh}_{\mathfrak{p}}$. We note here that this extension is not necessarily unique. We give two ways to see this. Namely, if we use the definition of the strict Henselization as a colimit of pointed \'{e}tale algebras as in \cite[\href{https://stacks.math.columbia.edu/tag/04GP}{Lemma 04GP}]{stacks-project}, then we can choose a set of compatible extensions of $\mathfrak{p}\subset \mathfrak{m}$ (note that the extension of $\mathfrak{m}$ is already fixed by definition, but this does not fix the extension of $\mathfrak{p}$). Since there are often many different extensions of $\mathfrak{p}\subset \mathfrak{m}$, this gives rise to different morphisms.    

We can also view this choice as follows. 
 We pick an extension $\mathfrak{m}^{\sep}\supset \mathfrak{p}^{\sep}$ of $\mathfrak{m}\supset \mathfrak{p}$ to $A^{\sep}$ as in \cref{def:AbsoluteSections} and \cref{lem:SectionLemma}. We have $I_{\mathfrak{p}^{\sep}/\mathfrak{p}}\subset I_{\mathfrak{m}^{\sep}/\mathfrak{m}}$ by $z-\sigma(z)\in \mathfrak{p}^{\sep}\subset \mathfrak{m}^{\sep}$. Using infinite Galois theory, we then obtain %
an inclusion 
\begin{equation*}
A^{\sh}_{\mathfrak{m}}\simeq (A_{\mathfrak{m}^{\sep}}^{\sep})^{I_{\mathfrak{m}^{\sep}/\mathfrak{m}}}\subset (A^{\sep}_{\mathfrak{p}^{\sep}})^{I_{\mathfrak{p}^{\sep}/\mathfrak{p}}}\simeq A^{\sh}_{\mathfrak{p}}.
\end{equation*}
Here the isomorphisms follow from \cite[\href{https://stacks.math.columbia.edu/tag/0BSW}{Lemma 0BSW}]{stacks-project}. As before, there are possibly many extensions for $\mathfrak{p}\subset \mathfrak{m}$ to $A^{\sep}$, so this gives many different homomorphisms. Explicit examples of this can be found in \ref{exa:MainExample1} and \ref{exa:PlaneQuartic}.      

Similarly, if we have a chain of prime ideals $\mathfrak{m}_{1}\supset \mathfrak{m}_{2}\supset ...\supset\mathfrak{m}_{n}$, then we obtain a chain of inclusions 
\begin{equation*}
A^{\sh}_{\mathfrak{m}_{1}}\subset A^{\sh}_{\mathfrak{m}_{2}}\subset ...\subset A^{\sh}_{\mathfrak{m}_{n}}
\end{equation*}%
in this way. We will use this set-up throughout this section. The specific choice of an extension will not matter to our computations.

\subsubsection{Maps of strict Henselizations}\label{sec:MapsSH}

Let $A={K}[x_{1},..,x_{n}]$ with $\mathfrak{m}=(x_{1},...,x_{n})$ and $\mathfrak{p}=(x_{n})$. The corresponding residue fields are $k(\mathfrak{m})=K$ and $k(\mathfrak{p})=K(x_{1},...,x_{n-1})$ respectively. We fix a set of %
separable closures $K^{\sep}$ and $K(x_{1},...,x_{n-1})^{\sep}$, together with isomorphisms
\begin{align*}
A^{\sh}_{\mathfrak{m}}/\mathfrak{m}^{\sh}&\simeq K^{\sep},\\
A^{\sh}_{\mathfrak{p}}/\mathfrak{p}^{\sh}&\simeq K(x_{1},...,x_{n-1})^{\sep}
\end{align*} 
and an embedding 
\begin{equation*}
K^{\sep}\to K(x_{1},...,x_{n-1})^{\sep}
\end{equation*}
 such that %
the diagram 
\begin{equation*}
 \begin{tikzcd}
A^{\sh}_{\mathfrak{m}} \arrow[r] \arrow[d]& K^{\sep} \arrow[d]\\
A^{\sh}_{\mathfrak{p}} \arrow[r] & K(x_{1},...,x_{n-1})^{\sep}
\end{tikzcd}
 \end{equation*}
 commutes. %
Let $B={K}[x_{1},...,x_{n-1}]$, $\overline{\mathfrak{m}}=(x_{1},...,x_{n-1})\subset{B}$ and $\overline{\mathfrak{p}}=(0)\subset{B}$. We identify the separable closure of $k(\overline{\mathfrak{m}})$ with the separable closure of $k(\mathfrak{m})$, and similarly for $k(\mathfrak{p})$ and $k(\overline{\mathfrak{p}})$.  %
We have natural ring homomorphisms 
\begin{align*}
r:B&\to A,\\
s:A&\to B,
\end{align*}
with $s\circ r=\text{id}$. From this, we obtain a %
commutative diagram 
  \begin{equation*}
 \begin{tikzcd}
A_{\mathfrak{m}}\arrow[r] \arrow[d] & B_{\overline{\mathfrak{m}}} \arrow[r] \arrow[d] & A_{\mathfrak{m}} \arrow[d] \\
A_{\mathfrak{p}} \arrow[r] & B_{\overline{\mathfrak{p}}} \arrow[r] & A_{\mathfrak{p}}.
\end{tikzcd}
 \end{equation*}
 Here the horizontal arrows are local, and the vertical arrows are the localization morphisms. We can now turn this diagram into a diagram of strict Henselizations as follows. For the horizontal maps, we use the functoriality %
 of strict Henselizations as in \cite[\href{https://stacks.math.columbia.edu/tag/04GU}{Lemma 04GU}]{stacks-project}. The induced map of strict Henselizations on the vertical side comes from our section %
 $\mathfrak{m}^{\sep}\supset \mathfrak{p}^{\sep}$ of $\mathfrak{m}\supset{\mathfrak{p}}$. %
 From this, it is then easy to see that we obtain a commutative diagram 
   \begin{equation}\label{eq:CommutativeDiagram}
 \begin{tikzcd}
A^{\sh}_{\mathfrak{m}} \arrow[d] \arrow[r] & B^{\sh}_{\overline{\mathfrak{m}}} \arrow[d]\arrow[r] & A^{\sh}_{\mathfrak{m}} \arrow[d] \\
A^{\sh}_{\mathfrak{p}} \arrow[r]           & {K(x_{1},...,x_{n-1})^{\sep}} \arrow[r] &   A^{\sh}_{\mathfrak{p}}.
\end{tikzcd}
 \end{equation}
 Note that the middle vertical map is essentially unique, since there is only one extension of $\overline{\mathfrak{m}}\supset (0)$ if $\overline{\mathfrak{m}}^{\sep}$ is fixed.  
The vertical arrows are all injective. 
Note that $ {K(x_{1},...,x_{n-1})^{\sep}}$ is the residue field of the discrete valuation ring $A^{\sh}_{\mathfrak{p}}$, and the map on the bottom right gives a natural ring-theoretic section of the quotient map $A^{\sh}_{\mathfrak{p}}\to A^{\sh}_{\mathfrak{p}}/\mathfrak{p}^{\sh}\simeq {K(x_{1},...,x_{n-1})^{\sep}}$.  %
We use the image of this map as our set of representatives and find using \cref{prop:RegularExpansions} that we can uniquely write any element $\alpha\in {A^{\sh}_{\mathfrak{p}}}$ as 
\begin{equation*}%
\alpha= \sum d_{j}x_{n}^{j},
\end{equation*}
where $d_{j}\in {K(x_{1},...,x_{n-1})^{\sep}}$. Note that $d_{0}$ is exactly the image of $\alpha$ under the ring homomorphisms 
\begin{equation*}
A^{\sh}_{\mathfrak{p}}\to {K(x_{1},...,x_{n-1})^{\sep}} \to  A^{\sh}_{\mathfrak{p}}
\end{equation*}
from %
\cref{eq:CommutativeDiagram}.  
If $\alpha\in A^{\sh}_{\mathfrak{m}}\subset A^{\sh}_{\mathfrak{p}}$, then we moreover have the following: %
\begin{lemma}\label{lem:IntegralityPowerSeries}
Let $\alpha\in A^{\sh}_{\mathfrak{m}}$ and consider its unique $\mathfrak{p}$-adic power series expansion
\begin{equation*}
\alpha= \sum d_{i}x_{n}^{i}
\end{equation*}
arising from the section ${K(x_{1},...,x_{n-1})^{\sep}} \to  A^{\sh}_{\mathfrak{p}}$, see 
 \cref{prop:RegularExpansions}. 
Then $d_{i}\in B^{\sh}_{\overline{\mathfrak{m}}}$ for every $i\geq{0}$. 
\end{lemma}
\begin{proof}
Using the commutative diagram in \cref{eq:CommutativeDiagram}, we immediately find that $d_{0}\in B^{\sh}_{\overline{\mathfrak{m}}}$. By recursively applying the same reasoning to the higher approximations %
\begin{equation*}
\alpha'_{j}=(\alpha-\sum_{i=0}^{j} d_{i}x_{2}^{i})/x_{2}^{r}
\end{equation*}  
for $r=v_{\mathfrak{p}}(\alpha-\sum_{i=0}^{j} d_{i}x_{2}^{i})$, we then find that all the $d_{i}$ are defined over $B^{\sh}_{\overline{\mathfrak{m}}}$. 
\end{proof}

Consider the integral coefficient $d_{0}$ obtained from \cref{lem:IntegralityPowerSeries}. We can now restart our argument with $A=K[x_{1},...,x_{n-1}]$, $B=K[x_{1},...,x_{n-2}]$, $\mathfrak{m}=(x_{1},...,x_{n-1})$, $\mathfrak{p}=(x_{n-1})$ and $\alpha=d_{0}\in{A^{\sh}_{\mathfrak{m}}}$. As before, we find that the coefficients in its $\mathfrak{p}$-adic power series expansions are defined over $B^{\sh}_{\overline{\mathfrak{m}}}$. By continuing this process, we eventually find a set of coefficients that are defined over the strict Henselization of $K$, which is $K^{\sep}$. By retracing our steps and combining our findings, this gives the $(x_{1},...,x_{n})$-adic power series expansions for an element $\alpha\in K[x_{1},...,x_{n}]^{\sh}_{\mathfrak{m}}$. That is, we obtain the unique coefficients $c_{i}\in K^{\sep}$ such that 
\begin{equation}\label{eq:FinalCoefficients}
\alpha = \sum_{i\in\mathbb{N}^{n}}c_{i}x^{i},
\end{equation}
where $x^{i}=x_{1}^{i_{1}}\cdots x_{n}^{i_{n}}$. In the remaining sections, we give algorithms to compute these coefficients. 
\subsubsection{Notation for our algorithms}\label{sec:NotationAlgorithms}

To make the transition to our implementation as smooth as possible, we introduce some notation for our algorithms. The main objects in our code are globally saved in a large array denoted by \texttt{C}. An element of this array consists of a vector \texttt{C[i]} with entries \texttt{C[i][j]} for $j=1,...8$. We will call these \emph{state vectors}. We will sometimes denote these by \texttt{w} as well, with entries \texttt{w[j]}. We discuss some of the types that occur in a state vector: %
\begin{itemize}
\item A polynomial $f(z)\in{R=K[z,w_{1},...,w_{m},x_{1},...,x_{n}]}$. This is called the \emph{current polynomial}. It is stored in \texttt{C[i][1]}. 
\item The ring $R=K[z,w_{1},...,w_{m},x_{1},...,x_{n}]$. This is the \emph{ambient ring}. It is stored in \texttt{C[i][2]}. 
\item The variable $z$ is the \emph{covering variable}. It always comes first in our algebras $R$. %
\item The variables $w_{i}$ are the \emph{coefficient variables}.
\item The variables $x_{i}$ are the \emph{parameter variables}. 
\item A finite set $S\subset R$ that encodes the relations of \'{e}tale algebras. This is stored in \texttt{C[i][3]}. 
\item The integer $m$ as above is denoted by \texttt{c1} in the code. It is the number of coefficient variables. It is stored in \texttt{C[i][4]}.  
\item The integer $n$ as above is denoted by \texttt{c2} in the code. It is the number of parameter variables. It is stored in \texttt{C[i][5]}.   %
\item A local approximation height $r_{0}$. It is stored in \texttt{C[i][6]}.  
\item An array $V$ of approximations. Each array $V$ will consist of arrays $\{P,r\}$, where $P=[a:b]$ is a projective zero (so that the affine zero is $a/b$), and $r$ is the height of the coefficient. Every entry $\{P,r\}$ can be viewed as a \emph{term} $(a/b)x_{n}^{r}$. We note here that $b$ is never zero for us. These approximations are stored in \texttt{C[i][7]}. 
\item A tag that is either \texttt{etale} or \texttt{tropical}. It is stored in \texttt{C[i][8]}. 
\end{itemize}

The main routines of our algorithm are the \emph{\'{e}tale} and \emph{tropical} routines. The terminology will be explained later on. During each such routine, a given polynomial \texttt{C[i][1]} is taken, and a set of new polynomials is produced (together with other data). The vector \texttt{C[i]} is then replaced by a new set of global vectors using      
\texttt{splice!()}. In terms of memory, this is not entirely efficient, but it will suffice for our purposes.    

\subsubsection{How to calculate power series expansions: the \'{e}tale routine}\label{sec:EtaleRoutine}

We now explain how one can calculate the coefficients $d_{i,j}$ from \cref{lem:IntegralityPowerSeries} in practice. %
The main idea is to generalize %
the classical discrete Newton-Puiseux algorithm, which can be distilled from the proof of \cite[Theorem 2.1.5]{MS15}. %
To explain this, we will heavily use the fact that $A^{\sh}$ is a unique factorization domain, see \cite[\href{https://stacks.math.columbia.edu/tag/06LN}{Lemma 06LN}]{stacks-project}. 

We assume without loss of generality that $f(z)\in A[z]=K[x_{1},...,x_{n}][z]$. Suppose that $f(z)$ splits completely over $A^{\sh}$, so that 
\begin{equation*}
f(z)=u \prod (z-\alpha_{i})
\end{equation*}            
for $\alpha_{i}\in A^{\sh}$ and $u$ a unit in $A^{\sh}$. %
To calculate the first $\mathfrak{p}$-adic coefficient of a root for $\mathfrak{p}=(x_{n})$, we calculate the reduction of $f(z)$ modulo $(x_{n})$. This gives a polynomial $\overline{f(z)}\in B[z]$ and we calculate an irreducible factorization %
\begin{equation*}
\overline{f}(z)=\overline{u} \prod g_{i}^{r_{i}}
\end{equation*}   
of $\overline{f(z)}$
over the fraction field $K(B)$. We will only be interested in calculating the $g_{i}$ for now, and not the $r_{i}$.  
\begin{remark}
Let $I$ be a prime ideal in $K[x_{1},...,x_{m}]$ with quotient $C$. Suppose we are given a (non-constant) polynomial $f(z)\in C[z]$. 
To calculate an %
irreducible factorization of $f(z)$ over $C$, let %
$R=K[z,x_{1},...,x_{m}]$. We write $I_{R}$ for the ideal generated by $I$ in $R$. We now calculate a primary decomposition of the ideal $I':=I_{R}+f(z)\cdot{R}\subset R$, and then return a Gr\"{o}bner basis of each of these factors with respect to a suitable elimination order. We then take %
the element of the Gr\"{o}bner basis with the smallest nonzero $z$-degree. The corresponding Julia file in our implementation that calculates this irreducible factorization is called \texttt{FactorDomain.jl}. For instance, if we take $C=\mathbb{Q}[x,y]/(y^2-x^3)$ and $f=z^2-x$, then we obtain the output 
\begin{lstlisting}[basicstyle=\ttfamily\scriptsize, escapeinside={(*@}{@*)}]
julia> FactorDomain(S,I,f2)
2-element Vector{Any}:
 -z*y + x^2
 z*y + x^2
\end{lstlisting}
This also shows that it is not necessary in this algorithm to assume that the domain $C$ has any special properties such as being smooth over $K$ or normal. 
\end{remark}
We now choose an irreducible factor $g_{i}$. If $\mathrm{deg}(g_{i})=1$, then we do not construct a new algebra.  %
Otherwise, we construct the \'{e}tale algebra $B[z]/(g_{i})$. %
This algebra non-canonically embeds into the strict Henselization $B^{\sh}_{\overline{\mathfrak{m}}}$, as we have different choices for the roots of $g_{i}$ in $B^{\sh}_{\overline{\mathfrak{m}}}\subset A^{\sh}_{\mathfrak{m}}$. For the algorithm, we can however simply work with $A[d_{0}]=A[z]/(g_{i})$. We view the above as calculating the zeroth-order power series expansions of the roots of $f(z)$. 
\begin{remark}
The corresponding algebras are created in \texttt{Code1.jl} using the function \texttt{AlgebraFromPolynomial(R,c1,c2,S,f)}. Here $R$ is the ambient polynomial ring, $c_{1}$ is the number of coefficient variables, $c_{2}$ is the number of parameter variables, $S$ is the set of current relations (which will be updated after this step) and $f$ is the given polynomial arising from the irreducible factorization.  The new variables used for different polynomials will often overlap. This will not cause any errors however, as we keep track of the corresponding \'{e}tale relations in \texttt{C[i][3]}.  
\end{remark} 
 If $\mathrm{deg}(g_{i})=1$, then we write $g_{i}=b_{i}z+a_{i}$ and let  %
 $\gamma_{i,P}=[a_{i}:b_{i}]$ be the projective root of $g_{i}$ in our \'{e}tale algebra. If $\mathrm{deg}(g_{i})>1$, then we take $a_{i}=d_{0}$ and $b_{i}=1$, so that $\gamma_{i,P}=[d_{0}:1]$.  
We now consider $f(z)$ as a polynomial in $A[d_{0}][z]$ and calculate  %
\begin{equation*}
f_{i,1}(z):=b_{i}^{n}f(z+a_{i}/b_{i}),
\end{equation*} 
where $n$ is the degree of $f$ as a polynomial in $z$. By construction, we have that there is at least one root of $f_{i,1}$ that is divisible by $x_{n}$, so that its valuation with respect to $\mathfrak{p}=(x_{n})$ (see \cref{def:TropicalRoot}) is positive. Note that we can define $f_{i,1}(z)$ without passing to the fraction field.  
\begin{remark}
This completes the \'{e}tale routine. It can be found in our code under \texttt{RetrieveNewPolynomialsEtale()}. We have included a short technical summary in \cref{alg:EtaleRoutine}.  %
\end{remark}

\subsubsection{How to calculate power series expansions: the tropical routine}\label{sec:TropicalRoutine}

We now continue with one of the polynomials $f_{i,1}(z)$ obtained in the \'{e}tale routine and turn to the second part of the algorithm. We know by construction that there is a root of $f_{i,1}$ that is divisible by $x_{n}$, so that the valuation of this root with respect to $x_{n}$ is positive. 
\begin{definition}\label{def:TropicalRoot}
Let $r$ be an integer. We say that $r$ is a tropical root of $f(z)$ is there exists a root $\alpha\in K(A^{\sh}_{\mathfrak{m}})$ of $f(z)$ such that $v_{x_{n}}(\alpha)=r$. Here $v_{x_{n}}:A^{\sh}_{\mathfrak{m}}\to \mathbb{Z}$ is the normalized valuation on $A^{sh}_{\mathfrak{m}}$ with $v_{x_{n}}(x_{n})=1$.  
\end{definition}
To calculate the tropical roots of a polynomial, we use the Newton polygon theorem, see \cite[Chapter II, Proposition 6.3]{Neukirch1999}. In terms of our set-up where $f(z)$ is a polynomial that splits completely over the strict Henselization $A^{\sh}_{\mathfrak{m}}$, this says that the tropical roots of a polynomial are exactly the slopes of its Newton polygon. %
Let $r>0$ be the slope of a non-trivial line segment in the Newton polygon of $f_{1}(z)$, %
corresponding to a root $\alpha$ of valuation $r$. In our code, we convert $f_{1}(z)$ to a tropical polynomial, and we use the built-in command for finding roots of tropical polynomials, see \texttt{TropicalRoots()}.   
Let 
\begin{equation*}%
f_{2}(z)=1/x_{n}^{k}f_{1}(x_{n}^{r}z),
\end{equation*}   
where $k$ is the content of $f_{1}(x_{n}^{r}z)$ with respect to $x_{n}$. Note here that $A^{\sh}_{\mathfrak{m}}$ is a unique factorization domain, so that the content is well defined. We remove this content using the function \texttt{RemoveContent()}. The polynomial $f_{2}(z)$ is calculated in \texttt{ValuationScaling()}. In $f_{2}(z)$, the corresponding transformed root $\alpha'=\alpha/x_{n}^{r}$ now has valuation $0$. %
At this point, we update our approximation height \texttt{C[i][6]} for each tropical root, and we replace the current polynomial $f_{1}(z)$ by the translated polynomial $f_{2}(z)$. This concludes the tropical routine. It can be found in \texttt{RetrieveNewPolynomialsTropical()}. We have included a technical summary of this algorithm in \cref{alg:TropicalRoutine}.

After this, we restart our algorithm and perform the \'{e}tale routine. That is, we calculate an irreducible factorization of the reduction of $f_{2}(z)$ modulo $x_{n}$. By iterating this process, we obtain the $\mathfrak{p}$-adic power series expansions of the roots of $f(z)$. 

\subsubsection{The discrete Newton-Puiseux algorithm}\label{sec:DNPAlgorithm}

We summarize our findings in an algorithm here. Throughout the algorithm, we will make use of the notation introduced in \cref{sec:NotationAlgorithms}. The corresponding code can be found in \texttt{IterationNP()}. In the algorithm, we write $[\cdot]$ for the empty array. 

\begin{algorithm}[h]
\caption{The generalized discrete Newton-Puiseux algorithm}\label{alg:DiscreteNPAlgorithm}
\begin{algorithmic}[1]
\Require A polynomial $f(z)\in A[z]$ for $A=K[x_{1},...,x_{n}]$. A height $r\in\mathbb{N}$. %
The polynomial $f(z)\in A[z]$ splits completely over the %
strict Henselization of $A$ with respect to $\mathfrak{m}=(x_{1},...,x_{n})$.
\Ensure The $x_{n}$-adic power series expansions of the roots of $f(z)$ up to height $r$.  
\State Initialize \texttt{C} with a single state vector \texttt{w}=$[f,A[z],[0],0,n,0,[\cdot],\texttt{etale}]$.
\State Run \texttt{InitialTropicalization()} to calculate the tropical roots of $f(z)$ with their corresponding state vectors, see \cref{alg:TropicalRoutine}. For every tropical root $k_{j}$, this defines a new state vector \texttt{w[j]} with approximation height \texttt{w[j][6]} $=k_{j}$. We also scale $f(z)$ according to the $k_{j}$. %
\State Replace \texttt{C[1]} with the set of \texttt{w[j]} obtained in Step 2.
\State $s_{i}=$\texttt{C[i][6]} for $i=1,...,c=$\texttt{length(C)}. 
\State $r_{0}=\mathrm{min}\{s_{i}\}_{i=1}^{c}$. %
\State $k=1$. 
\While{$r_{0}<r$}
\If{\texttt{C[k][6]}$\geq{r}$}%
\State $k=k+1$.
\Else
\If{\texttt{C[k][8]=etale}}
\State Run the \'{e}tale routine for \texttt{C[k]}, see \cref{alg:EtaleRoutine}. 
\Else
\State Run the tropical routine for \texttt{C[k]}, see \cref{alg:TropicalRoutine}. 
\EndIf
\State $r_{0}=\mathrm{min}\{\texttt{C[i][6]}\}$.  %
\EndIf
\EndWhile
\State \Return \texttt{C[i][7]} for $i=1,...,\texttt{length(C)}$. 
\end{algorithmic}
\end{algorithm}

\begin{theorem}\label{thm:DiscreteNPAlgorithm}
\cref{alg:DiscreteNPAlgorithm} correctly computes the $x_{n}$-adic power series expansions of $f(z)$ up to height $r$. 
\end{theorem}
\begin{proof}
The zeroth-order coefficients of the roots of $f(z)$ give the roots of the reduced polynomial $\overline{f}$ by the discussion in \cref{sec:MapsSH}. We thus see that the \'{e}tale routine correctly calculates these zeroth-order approximations. We now fix a root $\alpha_{i}$ with zeroth-order approximation $\gamma_{i}=a_{i}/b_{i}$. %
The transformation in the \'{e}tale routine sends a root $\alpha_{j}$ %
$\alpha_{j,1}:=\alpha_{j}-\gamma_{i}$. For $j=i$, this automatically has positive valuation $r_{i}$, so that $\alpha_{i,1}=\alpha_{i}-\gamma_{i}=x_{n}^{r_{i}}u$. The integer $r_{i}$ is calculated by the tropical routine. The next transformation sends a root $\alpha_{j,1}$ to $\alpha_{j,2}:=\alpha_{j,1}/x_{n}^{r_{i}}$. Note that if $z-\alpha_{j,2}$ is not integral, then the algorithm multiplies this term by a suitable factor of $x_{n}^{m}$, so that the $x_{n}^{m}z-\alpha_{j,2}x_{n}^{m}$ is integral. Its reduction modulo $x_{n}$ is then simply the reduction of the constant factor $\alpha_{j,2}x_{n}^{m}$.    %

The roots of the reduction of $f_{i,2}$ correspond to all of the $\alpha_{j,2}$ whose $x_{n}$-valuation is zero. We focus again on $j=i$, where the reduction of the root is the reduction of $u$. The next iteration of the \'{e}tale routine calculates the reduction of $u$, say $\delta_{i}$. We then directly find that $\alpha_{i}-(\gamma_{i}+\delta_{i}x_{n}^{r_{i}})$ has valuation greater than $r_{i}$, so that  %
$\gamma_{i}+x_{n}^{r_{i}}\delta_{i}$ is the first term in the power series expansion of $\alpha_{i}$. By continuing in this way, it now easily follows that the algorithm correctly calculates the $x_{n}$-adic power series expansions of all the roots of $f(z)$, up to any height. %
\end{proof}

\subsubsection{The full multivariate Newton-Puiseux algorithm}

We now indicate how \cref{alg:DiscreteNPAlgorithm} gives rise to an algorithm that calculates the $\mathfrak{m}$-adic power series expansions of all the roots of $f(z)$. Here, as before, $f(z)$ is a non-constant polynomial that splits completely over the strict Henselization $A^{\sh}_{\mathfrak{m}}$ of $A=K[x_{1},...,x_{n}]$ with respect to $\mathfrak{m}=(x_{1},...,x_{n})$. Recall that \cref{alg:DiscreteNPAlgorithm} calculates a series of coefficients $d_{i}\in{B^{\sh}_{\overline{\mathfrak{m}}}}$, where $B^{\sh}_{\overline{\mathfrak{m}}}$ is the strict Henselization of $B=K[x_{1},...,x_{n-1}]$ with respect to $\overline{\mathfrak{m}}=(x_{1},...,x_{n-1})$. We can calculate their $(x_{n-1})$-adic power series expansions using the following observation.  %
\begin{lemma}\label{lem:GaloisConjugates}
Every Galois conjugate of $d_{i}$ over $K(x_{1},...,x_{n-1})$ occurs in the $\mathfrak{p}$-adic power series expansion of a root $\alpha_{i}\in{V(f(z))}$. 
\end{lemma}
\begin{proof}
Since $f(z)$ splits over the strict Henselization, the action of the decomposition group reduces to the action of %
$\mathrm{Gal}(K(x_{1},...,x_{n-1})^{\sep}/K(x_{1},...,x_{n-1}))$. %
 Note that a conjugate of $d_{j}$ comes from an automorphism $\sigma$ of the strict Henselization by \cite[\href{https://stacks.math.columbia.edu/tag/09ZL}{Lemma 09ZL}]{stacks-project}. %
 One directly finds that the power series expansion of $\sigma(\alpha)$ is 
$\sum \sigma(d_{i})x_{n}^{i}$
since the action on $x_{n}$ is trivial and the set of representatives is invariant under the action of the decomposition group. This gives the statement of the lemma.  
\end{proof}

Using \cref{lem:GaloisConjugates}, we find that every conjugate of $d_{i}$ is also a coefficient in the power series of a root of $f(z)$. To calculate the $x_{n-1}$-adic power series of the $d_{i}$, we can thus continue with the minimal polynomial $h_{i}$ of $d_{i}$.
\begin{remark}
To calculate the minimal polynomial $h_{i}$ of $d_{i}$ over $K(x_{1},...,x_{n-1})$, we can eliminate variables as in \cite[Remark 2.6]{ABPR20}. This will be implemented in the near future. 
\end{remark}

We can assume without loss of generality that $h_{i}$ is defined over $B$. We now apply \cref{alg:DiscreteNPAlgorithm} to the $h_{i}$ to calculate the $(x_{n-1})$-adic power series expansions of the roots of the $h_{i}$, after which we calculate the $(x_{n-2})$-adic expansions of the coefficients in these, and so on. Eventually, the polynomials $h_{i}$ will be defined over $K$, and the roots of these give the desired $(x_{1},...,x_{n})$-adic coefficients of the roots of $f(z)$ as in \cref{eq:FinalCoefficients}. %

\begin{theorem}\label{thm:MainThmNumber2}
Let $f(z)$ be a non-constant polynomial over $A=K[x_{1},...,x_{n}]$ that splits completely over the strict Henselization of $A$ with respect to $\mathfrak{m}=(x_{1},...,x_{n})$. 
By iteratively applying \cref{alg:DiscreteNPAlgorithm} to the minimal polynomials of the $x_{k}$-adic coefficients, we obtain the %
$\mathfrak{m}$-adic power series expansions of the roots $\alpha_{i}$ of $f(z)$ up to any height. 
\end{theorem}
\begin{proof}
By the considerations in \cref{sec:MapsSH}, we find that the $x_{k}$-adic coefficients at any individual step are again defined over $A^{\sh}_{k-1}$, which is the strict Henselization of $A_{k}=K[x_{1},...,x_{k-1}]$ with respect to $\mathfrak{m}_{k}=(x_{1},...,x_{k-1})$. By \cref{lem:GaloisConjugates}, every root of a minimal polynomial of a coefficient is the coefficient of another power series. Using \cref{thm:DiscreteNPAlgorithm} on these minimal polynomials, we now inductively find that the $\mathfrak{m}$-adic expansions of the roots are computed correctly.  
\end{proof}

\begin{example}\label{exa:MainExample1}
The computations done in this example can all be found in the accompanying \texttt{Example3.20.jl} and \texttt{Example3.20Additional.jl}. 
Let $A=\overline{\mathbb{Q}}[x_{1},x_{2}]$ with $\mathfrak{m}=(x_{1},x_{2})$, $\mathfrak{p}_{1}=(x_{1})$, $\mathfrak{p}_{2}=(x_{2})$ and $t=x_{1}x_{2}$. 
Consider the polynomial %
\begin{align*}
f(z)& = (t^2+1)z^4+(t^4x_{1}+t^4)z^3+(t^{10}x_{1}^2+(t^{12}-2)x_{1}+t^{10})z^2+\\
& +(t^{14}x_{1}^3+t^{16}x_{1}^2+t^{16}x_{1}+t^{14})z+
 (t^{22}x_{1}^4+t^{8}x_{1}^3+(1+t^6)x_{1}^2+t^8x_{1}+t^{22}).
\end{align*}
 If we view $f(z)$ as a polynomial in the variables $x_{1}$ and $z$ over the field $K=\overline{\mathbb{Q}}\{\{t\}\}$ of Puiseux series over $\overline{\mathbb{Q}}$, %
then it gives a smooth plane quartic $X$ in $\mathbb{P}^{2}_{K}$. Moreover, the algebra $A$ will correspond to a modified version of the coordinate ring of an annulus, see   %
\cref{exa:MainExample2}. The calculation given here will be used in that example to calculate the dual intersection graph of a semistable model of $X$.  %
We first apply the Newton-Puiseux algorithm with respect to $\mathfrak{p}_{2}=(x_{2})$\footnote{It will follow from our calculations that $f$ splits completely over the strict Henselization of $A$ with respect to $\mathfrak{m}$, but we will also give an abstract reason for this in \cref{sec:CoveringsSemistableModels}.}. The first iteration of \texttt{InitialTropicalization()} does nothing, as the Newton polygon is trivial. In the accompanying example code, we now simply run \texttt{IterationNP(C,1,5)}, which calculates the power series expansions of the first root up to order $5$. We will go through each of the individual steps here.  We apply the \'{e}tale routine and calculate the reduction of $f$ modulo $(x_{2})$:
\begin{lstlisting}[basicstyle=\ttfamily\scriptsize, escapeinside={(*@}{@*)}]
julia> f2=red(f)
z^4 - 2*z^2*x1^4 + x1^8
\end{lstlisting}
We now calculate an irreducible factorization of \texttt{f2}:
\begin{lstlisting}[basicstyle=\ttfamily\scriptsize, escapeinside={(*@}{@*)}]
julia> S=FactorDomain(R3,J,f2)
2-element Vector{Any}:
 -z + x1^2
 z + x1^2
\end{lstlisting}
We thus have two roots, $z=\pm x_{1}^{2}$. Both of these are of multiplicity $2$. We then translate $f$ over each of these to obtain two new polynomials $f_{\pm,1}$. This is done using \texttt{RetrieveNewPolynomialsEtale(S,C,1)}. We then redefine \texttt{C} by removing the old polynomial, and adding these two new polynomials.

We now move to the tropical routine. We calculate the positive tropical roots of $f_{\pm,1}$ using \texttt{TropicalRoots()}. This gives
\begin{lstlisting}[basicstyle=\ttfamily\scriptsize, escapeinside={(*@}{@*)}]
julia> TropicalRoots(R3,f,3)
2-element Vector{QQFieldElem}:
 0
 1
\end{lstlisting}
 In both cases, $1$ is the only positive root. %
We then scale $f_{\pm,1}$ using \texttt{ValuationScaling(R,f,1,3)}. Here the $3$ is used to denote the current parameter variable. This gives us the polynomials $f_{\pm,2}$. As before, we then apply the \'{e}tale routine to find an irreducible factorization of the reduction modulo $(x_{2})$ of $f_{\pm,2}$. The code is as follows:
\begin{lstlisting}[basicstyle=\ttfamily\scriptsize, escapeinside={(*@}{@*)}]
julia> f=ValuationScaling(R3,f,1,3);
julia> g=red(f);
julia> S=FactorDomain(R3,J,g);
julia> S2=FilterFactor(R3,S);
julia> S2
1-element Vector{Any}:
 4*z^2 + x1^6
\end{lstlisting}
Here we use \texttt{FilterFactor()} to remove any polynomials of degree $0$ ($x_{1}$ divides the polynomial, so we remove this). The polynomial $4z^2 + x_{1}^6$ is irreducible since we are working over $\mathbb{Q}$. It has the distinct roots $\pm{i}x_{1}^{3}$ over $\overline{\mathbb{Q}}$ (this also come out of the discrete Newton-Puiseux algorithm applied to this polynomial with respect to $x_{1}$, see the accompanying code). We note that these roots are distinct, so that we can apply Hensel's lemma to conclude that the roots lift uniquely, as expected. Continuing in this way, we eventually find the $x_{2}$-adic power series expansions of the roots up to height $5$. This yields the following output for the first root:
\begin{lstlisting}[basicstyle=\ttfamily\scriptsize, escapeinside={(*@}{@*)}]
julia> C[1][7]
5-element Vector{Any}:
 Any[QQMPolyRingElem[x1^2 1], 0]
 Any[QQMPolyRingElem[w1, 1], 1]
 Any[QQMPolyRingElem[-3//8*x1^4 1], 2]
 Any[QQMPolyRingElem[-x1^10 + 5//4*x1^8 - x1^6 8*w1], 3]
 Any[QQMPolyRingElem[-1//4*x1^8 + 35//128*x1^6 - 1//4*x1^4 1], 4]
\end{lstlisting}
    Here $4w_{1}^2+x_{1}^6=0$. For the second root, we find
    \begin{lstlisting}[basicstyle=\ttfamily\scriptsize, escapeinside={(*@}{@*)}]
julia> C[2][7]
5-element Vector{Any}:
 Any[QQMPolyRingElem[-x1^2 1], 0]
 Any[QQMPolyRingElem[w1, 1], 1]
 Any[QQMPolyRingElem[3//8*x1^4 1], 2]
 Any[QQMPolyRingElem[x1^10 + 5//4*x1^8 + x1^6 8*w1], 3]
 Any[QQMPolyRingElem[-1//4*x1^8 - 35//128*x1^6 - 1//4*x1^4 1], 4]
\end{lstlisting}
The relation for $w_{1}$ is the same here. 
This gives the $(x_{2})$-adic power series expansions 
    \begin{align*}
\gamma_{1}&=x_{1}^2+1/2ix_{1}^3x_{2}+(3/8 \cdot x_{1}^4)x_{2}^2+...\\
\gamma_{2}&=x_{1}^2-1/2ix_{1}^3x_{2}+(3/8 \cdot x_{1}^4)x_{2}^2+...\\
\gamma_{3}&=-x_{1}^2+1/2ix_{1}^3x_{2}+(-3/8\cdot x_{1}^4)x_{2}^2+...\\
\gamma_{4}&=-x_{1}^2-1/2ix_{1}^3x_{2}+(-3/8\cdot x_{1}^4)x_{2}^2+...
\end{align*} 
    These then also yield the $\mathfrak{m}$-adic power series expansions, modulo $(x_{2})^3$.

We now calculate the $\mathfrak{p}_{1}$-adic expansions. We find 
\begin{lstlisting}[basicstyle=\ttfamily\scriptsize, escapeinside={(*@}{@*)}]
julia> C[1][7]
3-element Vector{Any}:
 Any[QQMPolyRingElem[1 1], 2]
 Any[QQMPolyRingElem[w1, 1], 3]
 Any[QQMPolyRingElem[-1//4*x2^4 - 3//8*x2^2 1], 4]
\end{lstlisting}
for the first polynomial. Here $4w_{1}^2 + x_{2}^4 + x_{2}^2=0$, which is irreducible over $\mathbb{Q}(x_{2})$.  Note that this corresponds to two roots of $f(z)$. Indeed, for each root $\alpha_{i}$ of $4w_{1}^2 + x_{2}^4 + x_{2}^2=0$, we obtain one root of $f(z)$. %
For the second polynomial in \texttt{C}, we find %
\begin{lstlisting}[basicstyle=\ttfamily\scriptsize, escapeinside={(*@}{@*)}]
julia> C[2][7]
3-element Vector{Any}:
 Any[QQMPolyRingElem[-1 1], 2]
 Any[QQMPolyRingElem[w1, 1], 3]
 Any[QQMPolyRingElem[-1//4*x2^4 + 3//8*x2^2 1], 4],
\end{lstlisting}
where $-4w_{1}^2 + x_{2}^4 - x_{2}^2=0$. We write $\beta_{i}$ for the two roots of this polynomial. All in all, we have the following $\mathfrak{p}_{1}$-adic power series expansions   
\begin{align*}
\delta_{1}&=x_{1}^2+\alpha_{1} x_{1}^3+(-1/4\cdot x_{2}^4 - 3/8\cdot x_{2}^2)x_{1}^4,\\
\delta_{2}&=x_{1}^2-\alpha_{2} x_{1}^3+(-1/4\cdot x_{2}^4 - 3/8\cdot x_{2}^2)x_{1}^4,\\
\delta_{3}&=-x_{1}^2+\beta_{1} x_{1}^3+(-1/4\cdot x_{2}^4 + 3/8\cdot x_{2}^2)x_{1}^4,\\
\delta_{4}&=-x_{1}^2+\beta_{2} x_{1}^3+(-1/4\cdot x_{2}^4 + 3/8\cdot x_{2}^2)x_{1}^4.
\end{align*}
We note here that the coefficients for $x_{1}^{5}$ again depend on the $\alpha_{i}$ and $\beta_{i}$, and this periodic pattern repeats. 

To calculate the $\mathfrak{m}$-adic expansions, we now apply the Newton-Puiseux algorithm to the $\alpha_{i}$ and $\beta_{i}$ with respect to %
$(x_{2})$. For the $\alpha_{i}$, this yields 
\begin{lstlisting}[basicstyle=\ttfamily\scriptsize, escapeinside={(*@}{@*)}]
julia> C[1][7]
5-element Vector{Any}:
 Any[QQMPolyRingElem[w1, 1], 1]
 Any[QQMPolyRingElem[-1 8*w1], 3]
 Any[QQMPolyRingElem[-4 -128*w1], 5]
 Any[QQMPolyRingElem[-8192 524288*w1], 7]
 Any[QQMPolyRingElem[-351843720888320 -36028797018963968*w1], 9]
\end{lstlisting}
where $4w_{1}^2+1=0$. In other words, we have 
\begin{align*}
\alpha_{i}=w_{1}x_{2}-(1/(8w_{1}))x_{2}^3+(4/(128\cdot w_{1}))x_{2}^5+...
\end{align*}
Similarly, we find 
\begin{lstlisting}[basicstyle=\ttfamily\scriptsize, escapeinside={(*@}{@*)}]
julia> C[1][7]
5-element Vector{Any}:
 Any[QQMPolyRingElem[w1, 1], 1]
 Any[QQMPolyRingElem[-1 -8*w1], 3]
 Any[QQMPolyRingElem[4 128*w1], 5]
 Any[QQMPolyRingElem[-8192 -524288*w1], 7]
 Any[QQMPolyRingElem[351843720888320 36028797018963968*w1], 9],
\end{lstlisting}
so that 
\begin{align*}
\beta_{i}=w_{1}x_{2}+(1/(8w_{1}))x_{2}^3+(4/(128w_{1}))x_{2}^5+...
\end{align*}
Here $4w_{1}^2+1=0$ again. 
By comparing the $\mathfrak{m}$-adic coefficients of $\gamma_{i}$ and $\delta_{i}$, we see that this identifies every $\gamma_{i}$ with $\delta_{i}$, so that the labeling is preserved. For an %
example where the labeling changes, see \cref{exa:PlaneQuartic}.     %
\end{example}

\begin{remark}\label{rem:NumberFieldAlgorithm}
Let $K$ be a number field, let $\mathcal{O}_{K}$ be its ring of integers and let $\mathfrak{p}$ be a nonzero prime ideal with localization $\mathcal{O}_{K,\mathfrak{p}}$. %
Since $\mathcal{O}_{K,\mathfrak{p}}$ is a discrete valuation ring, we have that $\mathcal{O}_{K,\mathfrak{p}}[x_{1},..,x_{n}]$ is again a unique factorization domain. Moreover, the procedure given above works verbatim till the last step where we have an element $\alpha$ that splits completely over $\mathcal{O}_{K,\mathfrak{p}}^{\sh}$. For these, one can use similar variant of the one-dimensional Newton-Puiseux algorithm to calculate the power series expansions of $\alpha$, after choosing a {set-theoretic} splitting of the residue field map $\mathcal{O}_{K,\mathfrak{p}}\to k(\mathfrak{p})$. 
\end{remark}

\begin{remark}\label{rem:GeneralizationNumberFields}
Let $R$ be a discrete valuation ring with uniformizer $\pi$ and let  %
$Y=\Spec(A)$, where 
$$A=R[y_{0},...,y_{n}]/(y_{0}\cdots y_{r}-\pi).$$
We write $\mathfrak{m}=(y_{0},...,y_{n},\pi)$. The space $Y$ is a standard \'{e}tale-local chart of a (regular) semistable model over $R$. If $R=k[\pi]_{(\pi)}$ and $f(z)\in{A}[z]$ is a polynomial that splits completely over the strict Henselization $A^{\sh}$, then we can automatically use our algorithm since the base ring $A$ is isomorphic to a localization of the polynomial ring $k[y_{0},...,y_{n},\pi]$. %

Suppose however that $R=\mathcal{O}_{K,\mathfrak{p}}$ as in \cref{rem:NumberFieldAlgorithm} and let $\pi$ again be a uniformizer. The above modification then doesn't work, since we have prime ideals that are defined over fields of different characteristics. We can now modify our algorithm using  
\cite[\href{https://stacks.math.columbia.edu/tag/04D1}{Lemma 04D1}]{stacks-project} as follows. First, we reduce to the case where %
$r=n$. We then have $A/(y_{n})\simeq k(\mathfrak{p})[y_{0},...,y_{n-1}]=:\overline{A}$. As before, we obtain an induced map of strict Henselizations 
\begin{equation*}
r: A^{\sh}_{\mathfrak{m}} \to \overline{A}^{\sh}_{\overline{\mathfrak{m}}}.
\end{equation*} 
To lift elements from $\overline{A}^{\sh}_{\overline{\mathfrak{m}}}$ to $A^{\sh}_{\mathfrak{m}}$, %
consider an $\overline{A}$-algebra $\overline{B}$ that is \'{e}tale over $\overline{\mathfrak{m}}=(y_{1},...,y_{n-1})$. This algebra $\overline{B}$ can be lifted to an $A$-algebra $B$ that is \'{e}tale over $\mathfrak{m}$, see the proof of \cite[\href{https://stacks.math.columbia.edu/tag/04D1}{Lemma 04D1}]{stacks-project}. We then choose a set-theoretic section $\overline{B}\to{B}$ of the quotient map $B\to \overline{B}$.  %
By a standard set-theoretic limit argument,
we can combine these to obtain a section 
\begin{equation*}
s: \overline{A}^{\sh}_{\mathfrak{m}}\to A^{\sh}_{\mathfrak{m}}
\end{equation*}
of $r$. 
Moreover, we can extend this splitting to $\overline{A}^{\sh}_{\overline{\mathfrak{p}}}=k(y_{1},...,y_{n-1})^{\sep}\to A^{\sh}_{\mathfrak{p}}$, so that we again obtain a commutative diagram as in \cref{eq:CommutativeDiagram}. To calculate the $\mathfrak{m}$-adic power series expansions of an element in $s(\overline{A}^{\sh})$, note that the corresponding equation can be defined over $\mathcal{O}_{K,\mathfrak{p}}[y_{1},...,y_{n-1}]$. We can then use the modification in \cref{rem:NumberFieldAlgorithm} to calculate the corresponding $\mathfrak{m}$-adic power series expansions. 
\end{remark}

\begin{remark}\label{rem:PowerSeriesVSCompletions}
Since the power series constructed in \cref{sec:PowerSeries} are naturally elements of a certain completion, one might wonder if the material in this section can be rephrased using completions. Note however that for primes $\mathfrak{m}\supset \mathfrak{p}$ in a commutative ring $A$, one in general does not obtain a morphism of completions $\hat{A}_{\mathfrak{m}}\to \hat{A}_{\mathfrak{p}}$. Suppose for instance we first localize to obtain a morphism $A_{\mathfrak{m}}\to A_{\mathfrak{p}}$. %
If we endow these rings with their respective adic topologies, then this map %
is not continuous, %
so it does not extend to a map of completions. Heuristically, one can also see this as follows. %
Let $A=K[x_{1},x_{2}]$ and $\mathfrak{m}=(x_{1},x_{2})\supset \mathfrak{p}=(x_{2})$. Then a power series
\begin{equation*}
\alpha_{i}=\sum c_{i,j}x_{1}^{i}x_{2}^{j}
\end{equation*}
would have to be sent to 
\begin{equation*}
\alpha_{i}=\sum d_{j}x_{2}^{j},
\end{equation*}
which in general is impossible since the individual sums $\sum_{i=0}^{\infty} c_{i,j}x_{1}^{i}$ are not algebraic. %
Note that a similar problem arises if one does not first localize. Working with strict Henselizations allows us to surpass these difficulties, as we have maps $A^{\sh}_{\mathfrak{m}}\to A^{\sh}_{\mathfrak{p}}$. %
\end{remark}

\begin{remark}\label{rem:ModifiedAlgorithm}
A modified version of the algorithm also works if $f(z)$ is defined over %
$A^{\mathrm{sh}}$. Namely, we can represent the elements in this ring symbolically using elements in \'{e}tale algebras. %
 One has to be slightly more careful with Galois conjugates as in \cref{lem:GaloisConjugates} however, as they have to be taken relative to the field of definition of $f(z)$. One can also simply try all of the different orbits.   %
\end{remark}

\begin{remark}
In many situations in practice, one can modify the local algebra to obtain an algebra $A$ and polynomial $f(z)$ of the type considered in this section. For instance, for morphisms of semistable models this can be achieved using %
using Kummer extensions, see \cref{lem:KummerExtensions}. 
\end{remark}

\begin{remark}
Let $A=K[x_{1},...,x_{n}]$ and let $S\subset \Spec(A)=X$ be the poset generated by monomial prime ideals, so that $\mathfrak{m}=(x_{1},...,x_{n})$ is the maximal element of $S$. Note that the order complex of this poset is the barycentric subdivision of the ordinary $n$-simplex. 
In this section, we discussed a general algorithm for calculating the power-series expansions over prime ideals in $S$. %
It follows from the construction in \cref{thm:MainThmNumber2} that the individual approximations are in fact compatible sets of approximations in the sense of \cref{def:CompatibleSeparatingSets}. Thus, if one can calculate the actions of $D_{e^{\sep}}$, $D_{x^{\sep}}$ and $D_{y^{\sep}}$ on these approximations, then one can calculate the poset structure over $S$. We will see how to calculate this Galois action for $n=2$ in \cref{sec:FindingGaloisOrbits}.   %
\end{remark}

\section{Coverings of semistable models of curves}\label{sec:CoveringsSemistableModels}

In this section we apply the techniques from this paper to obtain a full algorithm to calculate the maps of dual intersection graphs associated to tame coverings of semistable models of curves. This includes equations for the residue curves and the lengths of the edges. We also indicate how these calculations allow us to naturally work with analytic equations of annuli, with an eye towards future applications in $p$-adic integration.    

\subsection{A review}\label{sec:ReviewSemistable}
We first review some terminology and well-known results on coverings of semistable models. Throughout this section, $K$ will be a discretely valued complete field of characteristic zero with valuation ring $R$, uniformizer $\pi$ and normalized valuation $v(\pi)=1$. We will assume for simplicity that the residue field $k$ is algebraically closed. 
All curves in this section are smooth, proper and geometrically connected. A marked curve $(X,D)$ is a curve $X/K$ together with a finite set of closed points $D\subset{X}$.   
Let $(X',D')$ and $(X,D)$ be marked curves. %
We assume for simplicity that $D'\subset X'(K)$ and $D\subset X(K)$. %
A covering of marked curves is a covering $X'\to{X}$ with $\phi^{-1}(D)=D'$ such that the ramification locus of $\phi$ is contained in $D'$. A strongly semistable model of a marked curve $(X,D)$ is a semistable model $\mathcal{X}$ of $X$ (see \cite[Chapter 10, Definition 3.14]{liu2}) such that every irreducible component of the special fiber $\mathcal{X}_{s}$ is smooth, and every point of $D$ reduces to a smooth point of $\mathcal{X}_{s}$ under the reduction map 
\begin{equation*}
\mathrm{red}:X(K)\to \mathcal{X}_{s}(k).
\end{equation*}
We recall the definition of the reduction map here. %
From the open immersion $X=\mathcal{X}_{\eta}\to \mathcal{X}$, we obtain the map $X(K)=\mathcal{X}_{\eta}(K)\to \mathcal{X}(K)=\mathcal{X}(R)$, where the last equality follows from the valuative criterion of properness. We then compose this map with $\mathcal{X}(R)\to \mathcal{X}(k)=\mathcal{X}_{s}(k)$ to obtain the desired reduction map.      
\begin{definition}\label{def:SFPoset}
Let $\mathcal{X}$ be a strongly semistable model for $(X,D)$, and 
let $S_{\mathcal{X}}\subset \mathcal{X}$ be the subset consisting of the images of the following points under the closed embedding %
$\mathcal{X}_{s}\to\mathcal{X}$: %
\begin{enumerate}
\item The generic points of $\mathcal{X}_{s}$. %
\item The intersection points of the components of $\mathcal{X}_{s}$. 
\item %
The reductions of the points in $D$ in $\mathcal{X}_{s}$. 
\end{enumerate}
We endow $S_{\mathcal{X}}$ with the poset structure induced from $\mathcal{X}$. 
We call $S_{\mathcal{X}}$ the poset associated to the model $\mathcal{X}$ of $(X,D)$. %
The corresponding dual intersection graph $\Sigma(\mathcal{X})$ will be referred to as the skeleton of $\mathcal{X}$. A skeleton of a marked curve $(X,D)$ is a skeleton of a strongly semistable model $\mathcal{X}$ of $(X,D)$. 
\end{definition}

\begin{lemma}
Let $\mathcal{X}$ be a strongly semistable model of a curve $X$. Then $\mathcal{X}$ is relatively unibranch with respect to any pair of points $x\geq{y}$ (see \cref{def:RelativelyUnibranch}). 
\end{lemma}
\begin{proof}
We go over the different options for points of $\mathcal{X}$. For pairs $x\supset \eta$ where $\eta$ is the generic point of $\mathcal{X}$, this follows from the fact that normal schemes are geometrically unibranch. Let $x\in\mathcal{X}$ be a closed point that is smooth over $\text{Spec}(R)$. Then the local ring $\mathcal{O}_{\mathcal{X},x}$ is a unique factorization domain, and a generator of a codimension one point $y$ can be extended to a system of parameters for $x$. We then use \cref{exa:RegularPointsRelUnibranch}. Suppose that $x$ is an ordinary double point. If $y$ is the generic point of one of the irreducible components of the special fiber, then this follows from our assumption that $\mathcal{X}$ be strongly semistable. If $y$ is a closed point of the generic fiber, then one can use \cite[Chapter 10, Proposition 1.40]{liu2} to conclude that there is only one point over $y$ in the completion, and thus also in the Henselization.           
\end{proof}

To find skeleta in practice, we use the following well-known result: 

\begin{proposition}\label{pro:SimultaneousSemistable}
Let $(X',D')\to{(X,D)}$ be a covering of marked curves and let $\mathcal{X}$ be a strongly semistable model of $(X,D)$. Suppose that the morphism of normalizations 
\begin{equation*}
\mathcal{X}'\to\mathcal{X}
\end{equation*} 
is tame over every point of codimension one. Then there is a finite tame extension $R'$ of $R$ such that the normalized base change $\mathcal{X}'_{R'}$ is strongly semistable. Moreover, ordinary double points are sent to ordinary double points under the map $\mathcal{X}'_{R'}\to \mathcal{X}_{R'}$.  %
\end{proposition}
\begin{proof}
See \cite[Theorem 2.3]{liulorenzini1999} and \cite[Proposition 4.30]{liu2} for the case when $\phi$ is a Galois covering. The general case easily follows from the results in \cite{SGA1} or the considerations in \cref{lem:KummerExtensions}. %
We leave the details to the reader. They can also be found in an older version of this manuscript \cite{H2021}.   
\end{proof}

\begin{remark}
Throughout the rest of \cref{sec:CoveringsSemistableModels}, we will write $\mathcal{X}'\to \mathcal{X}$ for a morphism of semistable models of general curves, and $\mathcal{X}\to \mathcal{Y}$ for a morphism of semistable models associated to a covering $X\to \mathbb{P}^{1}$.
We will also write $x_{e}\in S_{\mathcal{X}}$ for an ordinary double point. Here we think of the ordinary double point as corresponding to an edge $e$ in the dual intersection graph $\Sigma(\mathcal{Y})$ of $\mathcal{Y}$, or an %
open line segment in the Berkovich analytification of $X$. We will try to make it clear when $e$ is used for an edge in the Hasse diagram of $S_{\mathcal{Y}}$ and when $e$ is used in the sense above. 
\end{remark}

\subsection{Outline of the algorithm}\label{sec:OutlineAlgorithm}

In this section we give an outline of our algorithm to find the dual intersection graph of a semistable model for a curve $X$. The individual parts of the algorithm will be explained in the upcoming sections. The input of the algorithm is a finite separable morphism $\phi: X\to \mathbb{P}^{1}$ given by a polynomial $f(y)\in K(x)[y]$ that is sufficiently tame in the sense of \cref{pro:SimultaneousSemistable}. The output is the poset $S_{\mathcal{X}}$ of a semistable model $\mathcal{X}$ of $X$. This also directly gives the dual intersection graph.   
\vspace{0.1cm}
\begin{center}
{\underline{Algorithm for finding dual intersection graphs of semistable models}}
\end{center}
\vspace{0.1cm}
\begin{enumerate}
\item Calculate sufficiently precise approximations of the branch locus $D$ of $\phi$ and deduce a semistable model $\mathcal{Y}$ of $(\mathbb{P}^{1},D)$ with poset $S_{\mathcal{Y}}$ from this data. Write $\phi_{\mathcal{O}_{K}}:\mathcal{X}\to \mathcal{Y}$ for the induced morphism of semistable models arising from \cref{pro:SimultaneousSemistable}\footnote{Here we assume for simplicity that the finite extension in \cref{pro:SimultaneousSemistable} has already been made. In the algorithms, one has to make a tamely ramified extension whenever one encounters a rational slope when computing Newton polygons as in \cref{sec:NPAlgorithm}.  }. %
\item For every pair $(x_{1},x_{e})\in S^{2}_{\mathcal{Y}}$ of a codimension $1$ point $x_{1}$ and an ordinary double point %
$x_{e}$ in $S_{\mathcal{Y}}$ with $x_{1}\leq x_{e}$, take a local chart around $x_{e}$ and calculate the $\mathfrak{m}_{x_{e}}$-adic power series expansions of the roots of $f(y)$ up to a separating height. Here $\mathfrak{m}_{x_{e}}$ is the maximal ideal corresponding to $x_{e}$ in the chosen chart.  %
\item Compare the $\mathfrak{m}_{x_{e}}$-adic power series expansions of pairs $(x_{1},x_{e})$ and $(x_{2},x_{e})$ as above that share the same intersection point $x_{e}$. 
\item Calculate the $D_{x_{1}^{\sep}/x_{1}}$ and $D_{x_{e}^{\sep}/x_{e}}$-orbits of the approximations for pairs  %
$(x_{1},x_{e})$ in $S^2_{\mathcal{Y}}$ as above. %
\item Use \cref{thm:MainThmv2} to deduce the structure of $S_{\mathcal{X}}=\phi^{-1}_{\mathcal{O}_{K}}(S_{\mathcal{Y}})$ from this data\footnote{The Hasse diagram of $S_{\mathcal{Y}}$ will be a tree, so that we do not need any transfer maps. See \cref{exa:ExplicitTransferMap} however for an example with transfer maps.  }. 
\end{enumerate}

The first part, including explicit equations for the charts, will be discussed in \cref{sec:ModelsP1}. The second and third parts then follow from the algorithm given in \cref{sec:NPAlgorithm}. The remaining parts will be discussed in \cref{sec:FindingGaloisOrbits}. We note that the data calculated here will also be sufficient to deduce the equations of the residue curves corresponding to generic points of the special fiber (see \cref{rem:ResidueCurves}), and the edge lengths (or: thicknesses) of the ordinary double points. This will also be explained in \cref{sec:FindingGaloisOrbits}.  We will discuss several examples in \cref{sec:Examples}. 

\subsection{Finding a model for the base curve}\label{sec:ModelsP1}

Let $\phi:X\to\mathbb{P}^{1}$ be a covering that is sufficiently tame in the sense of \cref{pro:SimultaneousSemistable} with branch locus $D$.
We write $U$ and $V$ for homogeneous coordinates on $\mathbb{P}^{1}$. %
We similarly use the coordinate $x=U/V$ on the local affine chart $D_{+}(V)$. %
We start by calculating sufficiently precise $\pi$-adic power series expansions with respect to $x$ for the points $P$ in the branch locus $D$ of $\phi$. Here sufficiently precise means that the $\pi$-adic power series expansions in the sense of \cref{prop:RegularExpansions} are distinct. %
We write points in $D$ as $P_{i}=[U_{i}:V_{i}]$. %
Consider the matrix $A=(a_{i,j})$ of tropical Pl\"{u}cker vectors corresponding to $D$, given by 
\begin{equation*}
a_{i,j}=\mathrm{val}(%
U_{i}V_{j}-U_{j}V_{i}).
\end{equation*}
We view the matrix $A$ as a point in a suitable tropical projective space, as in \cite[Section 4.3]{MS15}. Note that we can recover this matrix from our power series expansions. This data gives rise to a natural metric tree $\Sigma_{\mathrm{Berk}}$ in the Berkovich analytification $\mathbb{P}^{1,\mathrm{an}}$. Abstractly, this is the minimal skeleton of the marked curve $(\mathbb{P}^{1},D)$, see \cite[Corollary 4.23]{BPRa1}. It can also be reconstructed explicitly using the gluing technique in %
\cite[Section 2.1]{H2022}, or the cluster picture technique developed in  %
\cite{DDMM2023}.

This tree gives a canonical strongly semistable model $\mathcal{Y}$ for $(\mathbb{P}^{1},D)$ by the equivalence in \cite[Lemma 5.1]{ABBR2015}.
Explicitly, this model can be found as follows. We take the canonical minimal vertex set for $\Sigma_{\mathrm{Berk}}$, which gives rise to a natural set of vertices $V(\Sigma_{\mathrm{Berk}})$ and closed edges $E(\Sigma_{\mathrm{Berk}})$.  %
We can represent a finite closed edge $e\in{E(\Sigma_{\mathrm{Berk}})}$ by an inequality $a\leq v(x-c_{e})\leq b$ for $c_{e}\in{K}$ and $a,b\in\mathbb{Q}$. %
We can assume here that $a,b\in\mathbb{Z}$ after extending the field and rescaling the valuation. Let $m=b-a$. We associate a coordinate ring $A_{e}=R[u,v]/(uv-\pi^{m})$ to each of these edges, and we embed this ring into $K(\mathbb{P}^{1})=K(x)$ using the map $\psi_{e}:A_{e}\to{K(x)}$ given by %
\begin{align*}
\psi_{e}(u)&= \dfrac{x-c_{e}}{\pi^{a}},\\
\psi_{e}(v)&=\dfrac{\pi^{b}}{x-c_{e}}.
\end{align*}
Let $U_{e}=\Spec(\psi_{e}(A_{e}))$.
If we have two edges $e_{1}$ and $e_{2}$ that share a common vertex, then we obtain a canonical affine open subscheme $U_{e_{1},e_{2}}$ of $U_{e_{1}}$ and $U_{e_{2}}$, see \cref{exa:ExplicitModel} for an explicit example.
By iterating this process and gluing the $U_{e_{i}}$ along these open subschemes, %
 we obtain a global model $\mathcal{Y}$ of $\mathbb{P}^{1}$, which is easily seen to be a strongly semistable model for $(\mathbb{P}^{1},D)$.   %

\begin{example}\label{exa:ExplicitModel}
Consider the closed edges $e_{1}: 1\leq v(x-\pi)\leq 2$ and $e_{2}:1\leq v(x-\pi^{2})\leq 2$. These two edges meet at the vertex given by the equations $1=v(x-\pi)$ and $1=v(x-\pi^{2})$. We have  
\begin{align*}
\psi_{e_{1}}(A_{e_{1}})&=R[{(x-\pi)}/{\pi},{\pi^2}/{(x-\pi)}],\\
\psi_{e_{2}}(A_{e_{2}})&=R[{(x-\pi^{2})}/{\pi},{\pi^{2}}/{(x-\pi^{2})}].
\end{align*}The affine subscheme $U_{e_{1},e_{2}}=\Spec(B)$ of $U_{e_{1}}$ and $U_{e_{2}}$ is given by the composite ring %
$$B=R[{(x-\pi)}/{\pi},{\pi}/{(x-\pi)},{(x-\pi^{2})}/{\pi},{\pi}/{(x-\pi^{2})}].$$ %
Note that $B$ can be obtained from either $\psi_{e_{1}}(A_{e_{1}})$ or $\psi_{e_{2}}(A_{e_{2}})$ by a suitable localization.     
\end{example}

\begin{definition}
Let $A=R[u,v]/(uv-\pi^{n})$ be the standard coordinate ring of an annulus of length $n$. Let $A_{reg}=R[u_{1},v_{1}]/(u_{1}v_{1}-\pi)$. We define the regularization of $A$ to be the embedding $A\to A_{reg}$ given by 
\begin{align*}
u&\mapsto u_{1}^{n},\\
v&\mapsto v_{1}^{n}.
\end{align*}
Note that this defines a function field extension of degree $n$. 
If the embedding is clear from context, then we also refer to $A_{reg}$ as the regularization of $A$. If $A_{0}$ is an algebra isomorphic to $A$, then we similarly refer to the induced extension of $A_{0}$ as the regularization of $A_{0}$. 
\end{definition}

\begin{lemma}\label{lem:KummerExtensions}
Let $\mathcal{X}'\to \mathcal{X}$ be a finite dominant separable morphism of semistable models such that the closure of the generic branch locus is supported on the smooth locus of $\mathcal{X}\to \Spec(R)$, and write $f(z)$ for a polynomial that generates the function field extension $K(X)\to K(X')$. We suppose that $\mathcal{X}'\to \mathcal{X}$ is \'{e}tale over the generic points of the special fiber. Let $U\subset \mathcal{X}$ be an open affine that is \'{e}tale-locally isomorphic to the standard closed annulus $\Spec(A_{e})$. %
Then $f(z)$ splits completely over the field of fractions of the strict Henselization of the regularization of $A_{e}$ with respect to the maximal ideal $\mathfrak{m}=(u_{1},v_{1},\pi)=(u_{1},v_{1})$.    
\end{lemma} 
\begin{proof}
By assumption, there is no ramification over primes of codimension one arising from the special fiber. Moreover, there are no prime ideals of codimension $1$ contained in $\mathfrak{m}$ lying over the generic point of $\Spec(R)$ over which the covering is ramified by construction. We conclude using purity of the branch locus that the induced covering of the regularization is \'{e}tale over $\mathfrak{m}$. It quickly follows that $f(z)$ splits completely over the field of fractions of the strict Henselization with respect to $\mathfrak{m}$.     
\end{proof}

\subsection{Finding the Galois orbits}\label{sec:FindingGaloisOrbits}

In this section we give a description of how to find the Galois orbits necessary for computing skeleta of coverings of semistable models. The technique works for general coverings of semistable models as in \cref{pro:SimultaneousSemistable}, so we use the notation $\mathcal{X}'\to \mathcal{X}$ for the covering under consideration. We start with points of codimension one. %
Let $x'\in\mathcal{X}'$ be the image of a generic point of $\mathcal{X}'_{s}$ under the closed immersion $\mathcal{X}'_{s}\to\mathcal{X}'$, and let $x$ be its image in $\mathcal{X}$ under $\phi:\mathcal{X}'\to\mathcal{X}$. By our assumption on $\mathcal{X}'\to\mathcal{X}$, we have that %
a unifomizer $\pi\in{R}$ is again a local generator of the ideals corresponding to $x'$ and $x$. We will assume for simplicity that $f(z)$ is defined over $\mathcal{O}_{\mathcal{X},x}$, so that the roots are integral.   %
We can then express the $x$-adic power series expansions\footnote{Here, $x$-adic power series expansions means that we pick a local chart and calculate the  $\mathfrak{p}_{x}$-adic power series expansions, where $\mathfrak{p}_{x}$ is the corresponding prime ideal. } of a root %
$\alpha\in{V(f(z))}$ %
in terms of $\pi$ as follows: 
\begin{equation*}
\alpha=\sum_{i=0}^{\infty}c_{i}\pi^{i}.
\end{equation*}
Let $r>0$ be large enough so that the $r$-truncations of the roots (see \cref{def:Truncations}) are distinct. 
Suppose that $\alpha$ lies in the orbit corresponding to $x$. The reductions of the $c_{i}$ give a finite extension of fields 
\begin{equation*}
k(x)\subset \ell:=k(x)(\overline{c}_{0},...,\overline{c}_{r})\subset k(x)^{\sep}.
\end{equation*}   

We claim that $\ell$ is $k(x)$-isomorphic to $k(x')$. To that end, consider the induced extension of Henselizations $A:=\mathcal{O}^{\h}_{\mathcal{X},x}\subset B:=\mathcal{O}^{\h}_{\mathcal{X}',x'}$,  which is \'{e}tale by assumption. Note that the residue fields are not changed after passing to the %
Henselization. %
We now recall that there is an equivalence between  %
the category of finite \'{e}tale extensions of $A$ and the category of finite \'{e}tale extensions of $k(x)$, see \cite[\href{https://stacks.math.columbia.edu/tag/04GK}{Lemma 04GK}]{stacks-project}. Let $C$ be the \'{e}tale extension of $A$ corresponding to $\ell$ under this equivalence. Using Hensel's lemma and the fact that the power series expansions are unique, one directly sees that $\alpha\in{C}$. Suppose that $K(B)=K(A)(\alpha)$ is a strict subextension of $K(C)$, corresponding to an extension $k(x)\subset \ell'\subset \ell$. Let $m$ be the smallest integer such that $\overline{c}_{m}\notin \ell'$ and consider the element $z:=1/\pi^{m}(\alpha-\sum_{j=0}^{m-1}c_{i}\pi^{i})=c_{m}+\pi{h}\in K(C)\cap B=C$ for a suitable $h\in{B}$. We then have that $\overline{z}=\overline{c}_{m}\in\ell'$, a contradiction. We conclude that $C=B$. Note that this also implies that $n_{\alpha}:=[K(A)(\alpha):K(A)]=[\ell:k(x)]$, so that there are $n_{\alpha}$ embeddings of $K(A)(\alpha)$ into $K(A)^{\sep}$. The image of $\alpha$ under these embeddings are exactly the roots of $f(z)$ that are in the same $D_{x}$-orbit.

We now give an explicit criterion to see whether two roots are in the same $D_{x}$-orbit. 
We suppose without loss of generality that our set of representatives of the residue field $k(x)^{\sep}$ in $\mathcal{O}^{\sh}_{\mathcal{X},x}$ is invariant under $D_{x}$\footnote{This is automatic from the construction in \cref{sec:NPAlgorithm} if $K=k[t]_{(t)}$. If we are in the situation in \cref{rem:GeneralizationNumberFields}, then this can be done element-wise: we lift a single element, and then take the lifts of its conjugates to be the conjugates of the given lift.}. 
Since the action of $D_{x}$ on $\pi$ is trivial, we then have %
\begin{equation*}
\sigma(\alpha)=\sum_{i=0}^{\infty} \sigma(c_{i})\pi^{i}. 
\end{equation*}
In other words, we can calculate the Galois action by calculating it on the coefficients. 
Consider two roots $\alpha,\beta\in V(f(z))$ with power series expansions %
\begin{align*}
\alpha= \sum_{i=0}^{\infty} c_{i}\pi^{i},\\
\beta= \sum_{i=0}^{\infty} d_{i}\pi^{i}.
\end{align*}
We suppose that we have calculated the $r$-truncations of $\alpha$ and $\beta$ for some $r>0$, and we suppose that $r$ is sufficiently large, so that the $r$-truncations of all the roots are distinct. 
We then have that $\alpha$ and $\beta$ lie in the same $D_{x}$-orbit if and only if inductively, for $i\leq{r-1}$, the minimal polynomial of $\overline{c}_{i+1}$ over $k(x)(\overline{c}_{0},...,\overline{c}_{i})$ is the same as the minimal polynomial of $\overline{d}_{i+1}$ over $k(x)(\overline{d}_{0},...,\overline{d}_{i})$ under the isomorphism given by the polynomial in the previous step. This again follows from the equivalence of categories in \cite[\href{https://stacks.math.columbia.edu/tag/04GK}{Lemma 04GK}]{stacks-project} and standard Galois theory. 

\begin{remark}\label{rem:ResidueCurves}
Every point $x'\in S_{\mathcal{X}'}$ of codimension one lying over $x\in S_{\mathcal{X}}$ gives a morphism of smooth curves $C_{x'}\to C_{x}$ over the residue field $k$ of $K$. By the equivalence of categories between finite extensions of function fields and morphisms of smooth curves, we can recover $C_{x'}\to C_{x}$ from the extension of residue fields. By the above, we can recover this information from the algorithm, so that we can also recover the equations for the residue curves $C_{x'}$.    
\end{remark}
  
  We now discuss the remaining data in \cref{thm:MainThmv2}. We start with a lemma on the Galois action.
  
  \begin{lemma}\label{rem:KummerOrbitRemark}\label{lem:KummerOrbitRemark}
Let $\mathcal{X}'\to \mathcal{X}$ be a morphism of semistable models arising from \cref{pro:SimultaneousSemistable} and let $\overline{\mathcal{X}}\to \mathcal{X}'\to \mathcal{X}$ be the morphism of semistable models associated to the Galois closure of $K(X)\to K(X')$\footnote{Note that the Galois closure also satisfies all the necessary conditions, so that this is well defined. }. 
Let $x_{e}\in{S_{\mathcal{X}}}$ be an ordinary double point and let $x_{e}\in{S_{\mathcal{X}}}$ be an adjacent vertex with edge $s=x_{1}x_{e}$ in the Hasse diagram of $S_{\mathcal{X}}$. Write $\overline{s}=\overline{x}_{1}\overline{x}_{e}$ for a section of $s$ in $\overline{\mathcal{X}}$.  %
Then $D_{\overline{x}_{1}/x_{1}}=D_{\overline{s}/s}\subset D_{\overline{x}_{e}/x_{e}}$. Moreover, $D_{\overline{x}_{1}/x_{1}}$ is cyclic.   %
\end{lemma}
\begin{proof}
Suppose there exists a $\sigma\in D_{\overline{x}_{1}/x_{1}}$ that fixes $x_{1}$ but not $x_{e}$. Then $\sigma(\overline{x}_{e})$ is connected to the generic point of another component.  %
This contradicts the final part of \cref{pro:SimultaneousSemistable}, as we would obtain an ordinary double point in the quotient $\mathcal{X}$ contained in only one component. For the last part, we note that $f(z)$ splits completely over the strict Henselization of the regularization of a local chart by \cref{lem:KummerExtensions}. The extension induced by the regularization is a cyclic Kummer extension, so that we obtain the desired statement on the Galois group. 
\end{proof} 

  We retain the notation from \cref{rem:KummerOrbitRemark} for the ordinary double point $x_{e}$ with adjacent generic point $x_{1}$. %
  We now have to know the action of the individual $D_{\overline{x}_{1}/x_{1}}$ and $D_{\overline{x}_{e}/x_{e}}$  to reconstruct the inverse image of $S_{\mathcal{X}}$. The action of $D_{\overline{x}_{1}/x_{1}}$ was explained above, so we now explain how to find the action of $D_{\overline{x}_{e}/x_{e}}$.  
   \vspace{0.1cm}
Let $A_{e}=R[u,v]/(uv-\pi^{n})$ be the corresponding coordinate ring (see \cref{sec:ModelsP1}) %
with regularization $A_{e,\mathrm{reg}}=R[u_{1},v_{1}]/(u_{1}v_{1}-\pi)$ and map $A\to A_{\mathrm{reg}}$ given by $u\mapsto u_{1}^{n}$ and $v\mapsto v_{1}^{n}$. Then $f(z)$ splits completely over $L^{\sh}_{x_{e},\mathrm{reg}}$, which is the fraction field of the strict Henselization of $A_{e,\mathrm{reg}}$. As we saw earlier, the extension $L^{\sh}_{x_{e}}\subset L^{\sh}_{x_{e},\mathrm{reg}}$ is Galois with Galois group $\mathbb{Z}/n\mathbb{Z}$. To find the $D_{x_{e}}$-orbits of the roots, we calculate the action of this group on the approximations calculated by the algorithm. Let $x'_{e}\in \mathcal{X}'$ be a point lying over $x_{e}$ corresponding to the orbit of a root $\alpha_{i}$. The length or thickness of the corresponding edge (see \cite[Chapter 10, Definition 3.23]{liu2})  is then given by the length of $e$ (which is $n$), divided by $[L^{\sh}_{x_{e}}(\alpha_{i}):L^{\sh}_{x_{e}}]$. The latter in turn is the order of the orbit of $\alpha_{i}$ by basic field theory. 

\subsection{Examples}\label{sec:Examples}

We now show how the algorithms work in three examples. In the first and third example, we show how to calculate the covering of dual intersection graphs for coverings of the projective line of degrees four and three respectively. In the second example, we review the covering of elliptic curves in \cref{exa:SemistableModelSection} to show %
how the transfer maps in the $2$-limit work in terms of power series. To recover coverings of more general dual intersection graphs (for which the Hasse diagram of the base poset is not necessarily a tree), one has to apply similar monodromy techniques to recover the poset structure.  %

\begin{example}\label{exa:MainExample2}
Consider the polynomial 
\begin{align*}
f(z)& = (t^2+1)z^4+(t^4x+t^4)z^3+(t^{10}x^2+(t^{12}-2)x+t^{10})z^2+\\
& (t^{14}x^3+t^{16}x^2+t^{16}x+t^{14})z+
 (t^{22}x^4+t^{8}x^3+(1+t^6)x^2+t^8x+t^{22})
\end{align*} 
in $K(x)[z]$, where $K=\overline{\mathbb{Q}}(t)$\footnote{The computations automatically give the skeleton over the field of Puiseux series $\overline{\mathbb{Q}}\{\{t\}\}$ as well. We use the field $K$ here for computational simplicity, keeping in mind that it carries the standard $t$-adic valuation with $v(t)=1$. }. The extension $K(x)\subset{K(x)[z]/(f(z))}$ gives rise to a morphism of curves $\phi: X\to\mathbb{P}^{1}$ of degree four. 
The branch locus $D$ of $\phi$ consists of $12$ points. Let $\mathcal{Y}$ be the minimal semistable model of %
$(\mathbb{P}^{1},D)$. Using our algorithms, we then find that the associated covering of dual intersection graphs is as in \cref{fig:BExample}. 

\begin{figure}[h]
\scalebox{0.7}{

\begin{tikzpicture}[line cap=round,line join=round,>=triangle 45,x=1cm,y=1cm]
\clip(-0.21694666022106102,0.3379031991666796) rectangle (14.82865533183592,9.277554607003223);
\draw [line width=1.2pt] (2,2)-- (3,2);
\draw [line width=1.2pt] (3,2)-- (5,2);
\draw [line width=1.2pt] (5,2)-- (7,2);
\draw [line width=1.2pt] (7,2)-- (9,2);
\draw [line width=1.2pt] (9,2)-- (11,2);
\draw [line width=1.2pt] (11,2)-- (12,2);
\draw [line width=1.2pt] (2,2)-- (1.5,2.5);
\draw [line width=1.2pt] (2,2)-- (1.5,1.5);
\draw [line width=1.2pt] (3,2)-- (3,3);
\draw [line width=1.2pt] (3,2)-- (3,1);
\draw [line width=1.2pt] (3,2)-- (3.8,2.4);
\draw [line width=1.2pt] (5,2)-- (5,3);
\draw [line width=1.2pt] (12,2)-- (12.5,2.5);
\draw [line width=1.2pt] (12,2)-- (12.5,1.5);
\draw [line width=1.2pt] (11,2)-- (11,3);
\draw [line width=1.2pt] (11,2)-- (11,1);
\draw [line width=1.2pt] (11,2)-- (10.2,2.4);
\draw [line width=1.2pt] (9,2)-- (9,3);
\draw [line width=1.2pt] (2,8)-- (3,7);
\draw [line width=1.2pt] (3,7)-- (5,7);
\draw [line width=1.2pt] (9,7)-- (11,7);
\draw [line width=1.2pt] (11,7)-- (12,8);
\draw [line width=1.2pt] (11,7)-- (12,6);
\draw [line width=1.2pt] (3,7)-- (2,6);
\draw [line width=1.2pt] (2,8)-- (1.5,8.5);
\draw [line width=1.2pt] (2,8)-- (1.5,7.5);
\draw [line width=1.2pt] (2,6)-- (1.5,6.5);
\draw [line width=1.2pt] (2,6)-- (1.5,5.5);
\draw [line width=1.2pt] (12,8)-- (12.5,8.5);
\draw [line width=1.2pt] (12,8)-- (12.5,7.5);
\draw [line width=1.2pt] (12,6)-- (12.5,6.5);
\draw [line width=1.2pt] (12,6)-- (12.5,5.5);
\draw [line width=1.2pt] (3,7)-- (3,8);
\draw [line width=1.2pt] (3,7)-- (3.8,7.4);
\draw [line width=1.2pt] (3,7)-- (3,6);
\draw [line width=1.2pt] (5,7)-- (5,8);
\draw [line width=1.2pt] (5,7)-- (4.7,8.2);
\draw [line width=1.2pt] (5,7)-- (5.3,8.2);
\draw [line width=1.2pt] (9,7)-- (9,8);
\draw [line width=1.2pt] (9,7)-- (9.3,8.2);
\draw [line width=1.2pt] (9,7)-- (8.7,8.2);
\draw [line width=1.2pt] (2,8)-- (1.6,8.8);
\draw [line width=1.2pt] (2,8)-- (1.6,7.2);
\draw [shift={(7,5)},line width=1.2pt]  plot[domain=0.7853981633974483:2.356194490192345,variable=\t]({1*2.8284271247461903*cos(\t r)+0*2.8284271247461903*sin(\t r)},{0*2.8284271247461903*cos(\t r)+1*2.8284271247461903*sin(\t r)});
\draw [shift={(7,9)},line width=1.2pt]  plot[domain=3.9269908169872414:5.497787143782138,variable=\t]({1*2.8284271247461903*cos(\t r)+0*2.8284271247461903*sin(\t r)},{0*2.8284271247461903*cos(\t r)+1*2.8284271247461903*sin(\t r)});
\draw [line width=1.2pt] (3,7)-- (3.3,8.2);
\draw [line width=1.2pt] (3,7)-- (2.7,8.2);
\draw [line width=1.2pt] (3,7)-- (2.7,5.8);
\draw [line width=1.2pt] (3,7)-- (3.3,5.8);
\draw [line width=1.2pt] (3,7)-- (4,7.2);
\draw [line width=1.2pt] (3,7)-- (3.9,7.7);
\draw [line width=1.2pt] (11,7)-- (11,8);
\draw [line width=1.2pt] (11,7)-- (11.3,8.2);
\draw [line width=1.2pt] (11,7)-- (10.7,8.2);
\draw [line width=1.2pt] (11,7)-- (10.2,7.4);
\draw [line width=1.2pt] (11,7)-- (10.1,7.7);
\draw [line width=1.2pt] (11,7)-- (10,7.2);
\draw [line width=1.2pt] (11,7)-- (11,6);
\draw [line width=1.2pt] (11,7)-- (10.7,5.8);
\draw [line width=1.2pt] (11,7)-- (11.3,5.8);
\draw [line width=1.2pt] (12,8)-- (12.4,8.8);
\draw [line width=1.2pt] (12,8)-- (12.4,7.2);
\draw [->,line width=1.2pt,loosely dashed] (7,5) -- (7,3);
\draw (3.25,6.8) node[anchor=north west] {$1$};
\draw (10,6.8) node[anchor=north west] {$1$};
\begin{scriptsize}
\draw [fill=black] (2,2) circle (3pt);
\draw [fill=black] (3,2) circle (3pt);
\draw [fill=black] (5,2) circle (3pt);
\draw [fill=black] (7,2) circle (3pt);
\draw [fill=black] (9,2) circle (3pt);
\draw [fill=black] (11,2) circle (3pt);
\draw [fill=black] (12,2) circle (3pt);
\draw [fill=ududff] (1.5,2.5) circle (3pt);
\draw [fill=ududff] (1.5,1.5) circle (3pt);
\draw [fill=ududff] (3,3) circle (3pt);
\draw [fill=ududff] (3,1) circle (3pt);
\draw [fill=ududff] (3.8,2.4) circle (3pt);
\draw [fill=ududff] (5,3) circle (3pt);
\draw [fill=ududff] (12.5,2.5) circle (3pt);
\draw [fill=ududff] (12.5,1.5) circle (3pt);
\draw [fill=ududff] (11,3) circle (3pt);
\draw [fill=ududff] (11,1) circle (3pt);
\draw [fill=ududff] (10.2,2.4) circle (3pt);
\draw [fill=ududff] (9,3) circle (3pt);
\draw [fill=black] (2,8) circle (3pt);
\draw [fill=black] (3,7) circle (3pt);
\draw [fill=black] (5,7) circle (3pt);
\draw [fill=black] (9,7) circle (3pt);
\draw [fill=black] (11,7) circle (3pt);
\draw [fill=black] (12,8) circle (3pt);
\draw [fill=black] (12,6) circle (3pt);
\draw [fill=black] (2,6) circle (3pt);
\draw [fill=ududff] (1.5,8.5) circle (3pt);
\draw [fill=ududff] (1.5,7.5) circle (3pt);
\draw [fill=ffffff] (1.5,6.5) circle (2pt);
\draw [fill=ffffff] (1.5,5.5) circle (2pt);
\draw [fill=ududff] (12.5,8.5) circle (3pt);
\draw [fill=ududff] (12.5,7.5) circle (3pt);
\draw [fill=ffffff] (12.5,6.5) circle (2pt);
\draw [fill=ffffff] (12.5,5.5) circle (2pt);
\draw [fill=ududff] (3,8) circle (3pt);
\draw [fill=ududff] (3.8005858112108353,7.395112663219037) circle (3pt);
\draw [fill=ududff] (3,6) circle (3pt);
\draw [fill=ududff] (5,8) circle (3pt);
\draw [fill=ffffff] (4.7,8.2) circle (2pt);
\draw [fill=ffffff] (5.3,8.2) circle (2pt);
\draw [fill=ududff] (9,8) circle (3pt);
\draw [fill=ffffff] (9.3,8.2) circle (2pt);
\draw [fill=ffffff] (8.7,8.2) circle (2pt);
\draw [fill=white] (1.6,8.8) circle (2pt);
\draw [fill=ffffff] (1.6,7.2) circle (2pt);
\draw [fill=ffffff] (3.3,8.2) circle (2pt);
\draw [fill=ffffff] (2.7,8.2) circle (2pt);
\draw [fill=ffffff] (2.7,5.8) circle (2pt);
\draw [fill=ffffff] (3.3,5.8) circle (2pt);
\draw [fill=ffffff] (4,7.2) circle (2pt);
\draw [fill=ffffff] (3.9,7.7) circle (2pt);
\draw [fill=ududff] (11,8) circle (3pt);
\draw [fill=ffffff] (11.3,8.2) circle (2pt);
\draw [fill=ffffff] (10.7,8.2) circle (2pt);
\draw [fill=ududff] (10.2,7.4) circle (3pt);
\draw [fill=ffffff] (10.1,7.7) circle (2pt);
\draw [fill=ffffff] (10,7.2) circle (2pt);
\draw [fill=ududff] (11,6) circle (3pt);
\draw [fill=ffffff] (10.7,5.8) circle (2pt);
\draw [fill=ffffff] (11.3,5.8) circle (2pt);
\draw [fill=ffffff] (12.4,8.8) circle (2pt);
\draw [fill=ffffff] (12.4,7.2) circle (2pt);
\draw [fill=black] (7,6.17641) circle (3pt);
\draw [fill=black] (7,7.82) circle (3pt);
\end{scriptsize}
\end{tikzpicture}
}
\caption{\label{fig:BExample}The covering of dual intersection graphs in \cref{exa:MainExample2}. The blue vertices correspond to ramification or branch points of the covering. %
The $1$'s are the genera of the corresponding irreducible components of the special fiber.   %
 }
\end{figure}

We illustrate the calculations for the edge between $0$ and $4$. The corresponding algebra is $A_{0}=\overline{\mathbb{Q}}[u,v,t]/(uv-t^4)$, which we embed into $K(x)$ using the maps 
\begin{align*}
u&\mapsto x,\\
v&\mapsto t^{4}/x,\\
t&\mapsto t.
\end{align*} 
We take $x_{1}^4=u$ and $x_{2}^4=v$ to obtain the regular algebra $A={\overline{\mathbb{Q}}}[x_{1},x_{2}]$, which serves as a local smooth coordinate chart for $f(z)$. 
The polynomial $f(z)$ now coincides with the one studied in \cref{exa:MainExample1}, so we already have the desired power series expansions. 

We calculate the Galois orbits to find the edges and vertices lying over $\mathfrak{m}_{0}=(u,v,t)$, $\mathfrak{p}_{1}=(u,t)$ and $\mathfrak{p}_{2}=(v,t)$. %
The residue field for $\mathfrak{p}_{2}$ is %
$k(\mathfrak{p}_{2})=\overline{\mathbb{Q}}(x_{1}^4)$. %
The coefficients are defined over the degree-two extension $k(\mathfrak{p}_{2})\subset \overline{\mathbb{Q}}(x_{1}^2)$. We then easily find that the orbits are given by $\{\gamma_{1},\gamma_{4}\}$ and $\{\gamma_{2},\gamma_{3}\}$. There are thus two vertices lying over the given vertex. 

For $\mathfrak{p}_{1}$, we have that the residue field is $k(\mathfrak{p}_{1})=\overline{\mathbb{Q}}(x_{2}^4)$. We rewrite our coefficients in terms of $t$ and find that the $t^{2}$-coefficients are defined over the extension $k(\mathfrak{p}_{1})\subset \overline{\mathbb{Q}}(x_{2}^{2})=k'$. %
We now consider the element $\alpha_{1}\in{k(\mathfrak{p}_{1})^{\sep}}$, whose minimal polynomial over $k'$ is $4w^2+(x_{2}^2)^2+x_{2}^2$. Applying the non-trivial automorphism of $k'/k(\mathfrak{p}_{1})$, this is mapped to the minimal polynomial of the $\beta_{i}$ over $k'$. This implies that there is only one orbit: 
$\{\delta_{1},...,\delta_{4}\}$. We conclude that there is one vertex lying over this vertex. Note that we did not need to know the full local Galois group of $f(z)$ over $\mathfrak{p}_{1}$ here, which turns out to be the dihedral group $D_{4}$. %

For the edge, we calculate the action of $D_{\mathfrak{m}}$ on the $\mathfrak{m}$-adic power series. Note that this action is simply the one induced from the Kummer extension $A_{0}\to A$.  %
The two vertices corresponding to $\{\gamma_{1},\gamma_{4}\}$ and $\{\gamma_{2},\gamma_{3}\}$ give the two edges lying over the given edge. %
By repeating this procedure for all points in $S_{\mathcal{Y}}$, %
we then obtain the covering from \cref{fig:BExample}. 

\end{example}

\begin{example}\label{exa:ExplicitTransferMap}

We review \cref{exa:SemistableModelSection} to show how the transfer maps in the $2$-limit work in terms of power series. Recall that a local affine chart for $\mathcal{X}$ is given by $\Spec(R[u,v,y]/(y^2-(u+1)(v+1),uv-\pi))$.   %
Consider the subset $S=\{\mathfrak{m}_{1},\mathfrak{m}_{-1},\mathfrak{p}_{u},\mathfrak{p}_{v}\}\subset\mathcal{X}$. The Hasse diagram of the corresponding poset is a cycle. For the spanning tree, we take the four points in $S$, but without the edge $e=\mathfrak{m}_{-1}\mathfrak{p}_{u}$ (the specific edge will not play a significant role). We apply the algorithm from \cref{sec:NPAlgorithm} to $f(z)=z^2-(1+v)$: over $\mathfrak{p}_{v}$ we have the approximations 
\begin{align*}
\delta_{1}&=1,\\
\delta_{2}&=-1.
\end{align*}  
Over $\mathfrak{p}_{u}$, we have the approximations 
\begin{align*}
\gamma_{1}&=y,\\
\gamma_{2}&=-y.
\end{align*}

As in the algorithm of \cref{sec:NPAlgorithm}, we can now calculate the $\mathfrak{m}_{1}$-adic expansions of the roots using the $\delta_{i}$ and $\gamma_{i}$, which boils down to plugging in $y=1$ into our approximations. This connects $\delta_{i}$ to $\gamma_{i}$. We can also calculate the $\mathfrak{m}_{-1}$-adic expansions using $\delta_{i}$ and $\gamma_{i}$ by plugging in $y=-1$, which connects $\delta_{1}$ to $\gamma_{2}$ and $\delta_{2}$ to $\gamma_{1}$. This explicitly gives the transfer map, which is non-trivial here. The corresponding covering can be found in \cref{fig:CircleCovering}. %

\begin{figure}[ht]
\scalebox{0.35}{
\begin{tikzpicture}[line cap=round,line join=round,>=triangle 45,x=1cm,y=1cm]
\clip(3,2.1) rectangle (12,13.14);
\draw [shift={(7,2)},line width=2.0pt]  plot[domain=0.7853981633974483:2.356194490192345,variable=\t]({1*2.8284271247461903*cos(\t r)+0*2.8284271247461903*sin(\t r)},{0*2.8284271247461903*cos(\t r)+1*2.8284271247461903*sin(\t r)});
\draw [shift={(7,6)},line width=2.0pt]  plot[domain=3.9269908169872414:5.497787143782138,variable=\t]({1*2.8284271247461903*cos(\t r)+0*2.8284271247461903*sin(\t r)},{0*2.8284271247461903*cos(\t r)+1*2.8284271247461903*sin(\t r)});
\draw [line width=2.2pt] (5,12)-- (9,12);
\draw [line width=2.2pt,loosely dashed] (5,8)-- (9,12);
\draw [line width=2.2pt] (5,8)-- (9,8);
\draw [line width=2.2pt] (9,8)-- (5,12);
\draw [->,line width=2.2pt] (7,7.5) -- (7,5.44);
\begin{scriptsize}
\draw [fill=ududff] (9,4) circle (5.5pt);
\draw [fill=ududff] (5,4) circle (5.5pt);
\draw [fill=ududff] (5,12) circle (5.5pt);
\draw [fill=ududff] (9,12) circle (5.5pt);
\draw [fill=ududff] (5,8) circle (5.5pt);
\draw [fill=ududff] (9,8) circle (5.5pt);
\end{scriptsize}
\end{tikzpicture}
}
\caption{\label{fig:CircleCovering}
The covering $\Sigma'\to\Sigma$ of dual intersection graphs in \cref{exa:ExplicitTransferMap}.  We can locally calculate power series expansions of the roots of $f(z)$ over $\Sigma$, but not %
globally. %
}
\end{figure}
\end{example}

\begin{example}\label{exa:PlaneQuartic}
Consider the plane quartic $X/\mathbb{Q}_{p}$ for $p\geq{5}$ given by the polynomial
\begin{align*}
f&=xy(x+y+1)(x-y-1)+p+py\\
&=-xy^3 - 2xy^2 + (x^3 - x + p)y + p.
\end{align*}
The extension of function fields $\mathbb{Q}_{p}(x)\subset \mathbb{Q}_{p}(x)[y]/(f)$ gives rise to a covering $\phi: X\to\mathbb{P}^{1}$ of degree $3$, which in local coordinates is given by the projection $(x,y)\mapsto x$. The branch locus $D$ of $\phi$ is defined over an unramified extension of $K=\mathbb{Q}_{p}(p^{1/2})$, with simple ramification over every $P\in{D}$. 
The covering of dual intersection graphs can be found in \cref{fig:Covering2}. We note here that the local structure of the covering is not enough to determine the global structure, as there are at least two more ways of connecting the edges and vertices, see \cref{fig:Covering}.  

To use our algorithms on this example, we will do the following. Consider the curve $V(f)$ over $K=\mathbb{Q}(t)$ given by 
\begin{align*}
f&=xy(x+y+1)(x-y-1)+t^2+t^2y\\
&=-xy^3 - 2xy^2 + (x^3 - x + t^2)y + t^2.
\end{align*}
We will then work out the dual intersection graph with respect to the $(t)$-adic valuation on $K$. By taking $t\sim p^{1/2}$ and $p\geq{5}$, one easily sees that this also gives the dual intersection graph over $\mathbb{Q}_{p}(p^{1/2})$. The calculations can be found in \texttt{Example4.12.jl}. Implementing a fully $p$-adic version of the algorithms in this paper is the subject of ongoing work.

\label{rem:WorkAround}

 \begin{figure}[h]
\scalebox{0.5}{
\begin{tikzpicture}[line cap=round,line join=round,>=triangle 45,x=1cm,y=1cm]
\clip(-0.6,-0.98) rectangle (17.24,10.94);
\draw [line width=1.5pt] (6,2)-- (4,4);
\draw [line width=1.5pt] (6,2)-- (3,1);
\draw [line width=1.5pt] (6,2)-- (9,2);
\draw [line width=1.5pt] (9,2)-- (12,2);
\draw [line width=1.5pt] (6,10)-- (9,8);
\draw [line width=1.5pt] (6,8)-- (9,8);
\draw [line width=1.5pt] (9,8)-- (12,8);
\draw [line width=1.5pt] (9,8)-- (12,6);
\draw [line width=1.5pt] (12,6)-- (9,6);
\draw [line width=1.5pt] (9,6)-- (6,6);
\draw [line width=1.5pt] (6,10)-- (3,9);
\draw [line width=1.5pt] (2.98,6.84)-- (6,8);
\draw [line width=1.5pt] (6,6)-- (2.98,6.84);
\draw [line width=1.5pt] (6,10)-- (4,10);
\draw [line width=1.5pt,loosely dashed] (6,8)-- (4,10);
\draw [line width=1.5pt,loosely dashed] (6,6)-- (4,8);
\draw (8.75,9.25) node[anchor=north west] {\Large{$1$}};
\draw [->,line width=1.5pt] (7.48,4.98) -- (7.48,3);
\begin{scriptsize}
\draw [fill=qqqqff] (6,2) circle (5.0pt); %
\draw [fill=qqccqq] (4,4) circle (5.0pt);
\draw [fill=ffqqtt] (3,1) circle (5.0pt);  %
\draw [fill=qqqqff] (9,2) circle (5.0pt);  %
\draw [fill=qqqqff] (12,2) circle (5.0pt); %
\draw [fill=qqqqff] (6,10) circle (5.0pt); %
\draw [fill=qqqqff] (9,8) circle (5.0pt); %
\draw [fill=qqqqff] (6,8) circle (5.0pt); %
\draw [fill=qqqqff] (12,8) circle (5.0pt); %
\draw [fill=qqqqff] (12,6) circle (5.0pt); %
\draw [fill=qqqqff] (9,6) circle (5.0pt);
\draw [fill=qqqqff] (6,6) circle (5.0pt); %
\draw [fill=ffqqtt] (3,9) circle (5.0pt);  %
\draw [fill=ffqqtt] (3,6.84) circle (5.0pt); %
\draw [fill=qqccqq] (4,10) circle (5.0pt); %
\draw [fill=qqccqq] (4,8) circle (5.0pt);
\end{scriptsize}
\end{tikzpicture}
}
\caption{\label{fig:Covering2}The covering of dual intersection graphs in \cref{exa:PlaneQuartic}. The $1$ corresponds to the genus of the corresponding irreducible component. }
\end{figure}

We illustrate the calculations over the closed annulus $1/2\leq{v(x)}\leq 1$, corresponding to the rightmost edge in \cref{fig:Covering2}. The corresponding coordinate ring is $A=R[u,v]/(uv-t)$, where 
$u=x/t$ and $v=t^{2}/x$. We write $\mathfrak{p}_{1}=(u,t)=(u)$, $\mathfrak{p}_{2}=(v,t)=(v)$ and $\mathfrak{m}=(u,v,t)=(u,v)$. We note that the polynomial reduces to 
\begin{equation*}
f=-u^2v(z^3 + 2z^2 - zu^4v^2 - zv + z - v), 
\end{equation*} 
so that it is essentially the one in \cref{eq:MainHypersurface}. 

Applying the algorithm with respect to $\mathfrak{p}_{1}=(u)$, we find 
\begin{lstlisting}[basicstyle=\ttfamily\scriptsize, escapeinside={(*@}{@*)}]
julia> C[1][7]
2-element Vector{Any}:
 Any[QQMPolyRingElem[w1, 1], 0]
 Any[QQMPolyRingElem[-w1*v^2 -w1 - 2*v - 1], 4]
\end{lstlisting}
and 
\begin{lstlisting}[basicstyle=\ttfamily\scriptsize, escapeinside={(*@}{@*)}]
julia> C[2][7]
2-element Vector{Any}:
 Any[QQMPolyRingElem[-1 1], 0]
 Any[QQMPolyRingElem[v 1], 4]
\end{lstlisting}
Here $w_{1}^2 + w_{1} - v=0$. This relation can be obtained from \texttt{C[1][3]}. We write $\rho_{i}$ for the roots of this equation. Note that it defines a projective line over the residue field.   
From this data we obtain the $\mathfrak{p}_{1}$-adic power series expansions 
 \begin{align*}
    \gamma_{1,\mathfrak{p}_{1}}&=\rho_{1}+(\rho_{1}v^2)/(\rho_{1}+2v+1)u^4,\\%
    \gamma_{2,\mathfrak{p}_{1}}&=\rho_{2}+(\rho_{2}v^2)/(\rho_{2}+2v+1)u^4,\\ %
    \gamma_{3,\mathfrak{p}_{1}}&=-1+vu^4.%
\end{align*}
We now apply our algorithm again to the $\rho_{i}$ to find
\begin{lstlisting}[basicstyle=\ttfamily\scriptsize, escapeinside={(*@}{@*)}]
julia> C[1][7]
4-element Vector{Any}:
 Any[QQMPolyRingElem[1 1], 1]
 Any[QQMPolyRingElem[-1 1], 2]
 Any[QQMPolyRingElem[2 1], 3]
 Any[QQMPolyRingElem[-5 1], 4]
\end{lstlisting}
and 
\begin{lstlisting}[basicstyle=\ttfamily\scriptsize, escapeinside={(*@}{@*)}]
julia> C[2][7]
5-element Vector{Any}:
 Any[QQMPolyRingElem[-1 1], 0]
 Any[QQMPolyRingElem[-1 1], 1]
 Any[QQMPolyRingElem[1 1], 2]
 Any[QQMPolyRingElem[-2 1], 3]
 Any[QQMPolyRingElem[5 1], 4]
\end{lstlisting}
The power series expansions are thus 
\begin{align*}
\rho_{1}&=v-v^2+2v^3-5v^4+...\\
\rho_{2}&=-1-v+v^2-2v^3+5v^4+...
\end{align*}
Note that we could have switched the order of the labeling for the $\rho_{i}$. This corresponds to two different embeddings $A^{\sh}_{\mathfrak{m}}\to A^{\sh}_{\mathfrak{p}_{1}}$, as discussed in \cref{sec:StrictHenselizations1}. By combining these with our previous results and truncating, we then find the $\mathfrak{m}$-adic approximations 
\begin{align*}
    \gamma_{1,\mathfrak{p}_{1},\mathfrak{m}}&=(v-v^2+2v^3-5v^4)+(v^3-4v^4)u^4,\\
    \gamma_{2,\mathfrak{p}_{1},\mathfrak{m}}&=(-1-v+v^2-2v^3)+(-v-v^3+4v^4)u^4,\\
    \gamma_{3,\mathfrak{p}_{1},\mathfrak{m}}&=-1 +vu^4.
\end{align*}

We now apply the algorithm with respect to $\mathfrak{p}_{2}=(v)$. This yields 
\begin{lstlisting}[basicstyle=\ttfamily\scriptsize, escapeinside={(*@}{@*)}]
julia> C[1][7]
5-element Vector{Any}:
 Any[QQMPolyRingElem[-1 1], 0]
 Any[QQMPolyRingElem[w1, 1], 1]
 Any[QQMPolyRingElem[w1 - u^4 2*w1 + 1], 2]
\end{lstlisting}
and 
\begin{lstlisting}[basicstyle=\ttfamily\scriptsize, escapeinside={(*@}{@*)}]
julia> C[2][7]
4-element Vector{Any}:
 Any[QQMPolyRingElem[1 1], 1]
 Any[QQMPolyRingElem[-1 1], 2]
 Any[QQMPolyRingElem[u^4 + 2 1], 3]
 Any[QQMPolyRingElem[-4*u^4 - 5 1], 4].
\end{lstlisting}
Here $w_{1}^2 + w_{1} - u^4=0$. This relation can be found in \texttt{C[1][3]}. Note that this defines an elliptic curve over the residue field. We write $\beta_{i}$ for the roots of this equation. We then find the approximations %
\begin{align*}
\gamma_{1,\mathfrak{p}_{2}}&=-1+(\beta_{1})v+((\beta_{1}-u^4)/(2\beta_{1}+1))v^2,\\
\gamma_{2,\mathfrak{p}_{2}}&=-1+(\beta_{2})v+((\beta_{2}-u^4)/(2\beta_{2}+1))v^2,\\
\gamma_{3,\mathfrak{p}_{2}}&=v-v^2+(u^4+2)v^3-(4u^4+5)v^4.
\end{align*}
To calculate the $\mathfrak{m}$-adic power series expansions, we apply our algorithm to the $\beta_{i}$ and find 
\begin{lstlisting}[basicstyle=\ttfamily\scriptsize, escapeinside={(*@}{@*)}]
julia> C[1][7]
2-element Vector{Any}:
 Any[QQMPolyRingElem[1 1], 4]
 Any[QQMPolyRingElem[-1 1], 8]
\end{lstlisting}
and 
\begin{lstlisting}[basicstyle=\ttfamily\scriptsize, escapeinside={(*@}{@*)}]
 Any[QQMPolyRingElem[-1 1], 0]
 Any[QQMPolyRingElem[-1 1], 4]
 Any[QQMPolyRingElem[1 1], 8]
\end{lstlisting}
In other words, we have 
\begin{align*}
\beta_{1}&=u^4-u^8+...\\
\beta_{2}&=-1-u^4+u^8+...
\end{align*}
By expanding the $\beta_{i}$ and truncating, we find the $\mathfrak{m}$-adic approximations 
\begin{align*}
   \gamma_{1,\mathfrak{p}_{2},\mathfrak{m}}&=-1+(u^4-u^8)v,\\
   \gamma_{2,\mathfrak{p}_{2},\mathfrak{m}}&=-1+(-1-u^4)v,\\
   \gamma_{3,\mathfrak{p}_{2},\mathfrak{m}}&=v-v^2+(u^4+2)v^3.%
\end{align*}

We now compare the $\mathfrak{m}$-adic approximations $\gamma_{i,\mathfrak{p}_{1},\mathfrak{m}}$ and $\gamma_{i,\mathfrak{p}_{2},\mathfrak{m}}$. We find the following correspondence:
\begin{align*}
\gamma_{1,\mathfrak{p}_{1},\mathfrak{m}} &\Leftrightarrow \gamma_{3,\mathfrak{p}_{2},\mathfrak{m}}\\
\gamma_{2,\mathfrak{p}_{1},\mathfrak{m}} &\Leftrightarrow \gamma_{2,\mathfrak{p}_{2},\mathfrak{m}}\\
\gamma_{3,\mathfrak{p}_{1},\mathfrak{m}}& \Leftrightarrow \gamma_{1,\mathfrak{p}_{2},\mathfrak{m}}
\end{align*}

Note that the $D_{\mathfrak{p}_{1}}$-orbits of the roots are given by %
$\{\gamma_{1,\mathfrak{p}_{1}},\gamma_{2,\mathfrak{p}_{1}}\}$ and $\{\gamma_{3,\mathfrak{p}_{1}}\}$, and the $D_{\mathfrak{p}_{2}}$-orbits are $\{\gamma_{1,\mathfrak{p}_{2}},\gamma_{2,\mathfrak{p}_{2}}\}$ and $\{\gamma_{3,\mathfrak{p}_{2}}\}$. %
We thus find two vertices over $\mathfrak{p}_{1}$ and two vertices over $\mathfrak{p}_{2}$. Moreover, the $D_{\mathfrak{m}}$-orbits are trivial, so there are three points lying over $\mathfrak{m}$.  %
The correspondence between the $\gamma_{i,\mathfrak{p}_{1},\mathfrak{m}}$ and $\gamma_{i,\mathfrak{p}_{2},\mathfrak{m}}$ allows us to connect the different edges, which gives us the twist in \cref{fig:Covering2}. Note that the correspondence could also easily have been $\gamma_{i,\mathfrak{p}_{1},\mathfrak{m}}\Leftrightarrow\gamma_{i,\mathfrak{p}_{2},\mathfrak{m}}$, which would give a different covering and a different stable reduction type for $X$, see \cref{fig:Covering} for instance. 

\end{example}

\subsection{Interpretation in terms of analytic spaces}\label{sec:InterpretationAnalyticSpaces}

We now give a sketch of how our algorithm works in terms of analytic spaces in the sense of either \cite{FresnelPut2003} or \cite{Berkovich1993}. We will leave some of the details %
to future work. The reader can find similar computations for hyperelliptic coverings $X\to \mathbb{P}^{1}$ in \cite{KK2022}. %

Let $\Sigma$ be the minimal skeleton of $(\mathbb{P}^{1},D)$ with semistable model $\mathcal{Y}$ as in \cref{sec:ModelsP1}. We write %
$Y=\mathbb{P}^{1,\an}$ for the corresponding rigid analytic or Berkovich space. Let $\mathrm{red}:Y\to \mathcal{Y}_{s}$ be the anti-continous reduction map and let ${U}\subset \mathcal{Y}_{s}$ be an open affine. We then obtain an affinoid $Z_{{U}}=\mathrm{red}^{-1}(U)$ with algebra $\mathcal{A}_{U}$.   
We similarly write $Z_{\Gamma}$ %
for the inverse image of a closed subset $\Gamma$ of $\mathcal{Y}_{s}$ under the reduction map. If $U\subset \mathcal{Y}$, then we also write $Z_{U}$ for the affinoid associated to the inverse image of $U$ under $\mathcal{Y}_{s}\to \mathcal{Y}$.

Let $x_{e}\in \mathcal{Y}$ be an ordinary double point with local affine chart $U_{e}$ isomorphic to the spectrum of
$A_{e}=R[u,v]/(uv-\pi^{n})$. 
The affinoid algebra $\mathcal{A}_{U_{e}}$  of the affinoid $Z_{U_{e}}$ 
is then isomorphic to $K\langle u,v\rangle/(uv-\pi^{n})$. 
 Let
$\Gamma_{1}$ and $\Gamma_{2}$ be the two irreducible components of $\mathcal{Y}_{s}$ containing $x_{e}$. We have a commutative diagram of inclusions 
\begin{equation*}
\begin{tikzcd}
Z_{x_{e}} \arrow[d] \arrow[r] & {Z_{\Gamma_{2}}} \arrow[d] \\
{Z_{\Gamma_{1}}} \arrow[r]           & {Z_{U_{e}}}          
\end{tikzcd}.
\end{equation*}
Here the notation is taken relative to $U_{e}$, so that all of the above spaces are {open} subsets with respect to the valuation topology of the affinoid space $Z_{U_{e}}$. By localizing $A_{e}$, we can shrink $U_{e}$ and avoid %
the other intersection points and the reduction of the branch locus. We will again write $U_{e}$ for this open. %

We represent the covering $X^{\an}\to \mathbb{P}^{1,\an}$ using a polynomial $f(z)$ as before. We will assume for simplicity here that $f(z)$ is integral and monic over $Z_{U_{e}}$. Write $x_{i}\in\mathcal{Y}$ for the generic points corresponding to the $\Gamma_{i}$. Let %
\begin{equation*}
f(z)=\prod_{j=1}^{r_{i}} F_{j,i}(z)
\end{equation*} 
be an irreducible factorization of $f(z)$
over the Henselization $\mathcal{O}^{h}_{\mathcal{Y},x_{i}}$. Note that every irreducible factor corresponds  to an extension of $x_{i}$ by \cref{lem:RootsDecomposition}. By taking the $\pi$-adic completion and tensor product with $K$, we can view the $F_{j,i}(z)$ as elements of $\Gamma(U_{e},Z_{\Gamma_{i}})[z]$. Here $\Gamma(U_{e},Z_{\Gamma_{i}})$ is the ring of analytic functions on $Z_{\Gamma_{i}}$ relative to $U_{e}$. Note that the morphisms $X\to \mathbb{P}^{1}$ and $\mathcal{X}\to \mathcal{Y}$ are locally given as the spectrum of the normalization of an algebra $A[z]/(f(z))$, where $A$ is a local coordinate chart on $\mathbb{P}^{1}$ or $\mathcal{Y}$. We can thus describe $\phi^{-1}(Z_{\Gamma_{i}})$ as the normalization of %
\begin{equation*}
V_{\Gamma_{i}}=\{(u,y)\in Z_{\Gamma_{i}}\times K:f(y)=0\}
\end{equation*}
or its counterpart in $Z_{\Gamma_{i}}\times \mathbb{A}^{1,\an}$ for Berkovich spaces. Outside finitely many singular $\overline{K}$-rational points, this will be the disjoint union of the components 
\begin{equation*}
V_{j,i}=\{(u,y)\in Z_{\Gamma_{i}}\times K: F_{j,i}(y)=0\}.
\end{equation*} 
\begin{remark}
For computations, we do not need to work with explicit equations of the normalization of $V_{\Gamma_{i}}$ or $V_{j,i}$. %
Indeed, if this space of dimension one has a singular point, then we can adopt our power series algorithms to work at this point as well.  %
More explicitly, we can consider the projection of %
the corresponding point in $\mathcal{Y}$, and then add this point to our poset $S_{\mathcal{X}}$ (together with its reduction in $\mathcal{Y}_{s}$). This will suffice for most applications.  %
\end{remark} %

We now similarly consider an irreducible factorization
\begin{equation*}
f(z)=\prod_{j=1}^{r_{e}}G_{j}(z)
\end{equation*}
of $f(z)$ over the Henselization $\mathcal{O}^{h}_{\mathcal{Y},x_{e}}$. We can then also view the $G_{j}(z)$ as elements of $\Gamma(U_{e},Z_{x_{e}})[z]$, and each irreducible factor gives an extension of $x_{e}$. As before, we can describe the components over $Z_{x_{e}}$ as the %
union of the following spaces outside finitely many $\overline{K}$-rational points: %
\begin{equation*}
V_{j,e}=\{(u,y)\in Z_{x_{e}}\times K: G_{j}(y)=0\}.
\end{equation*} 

\begin{remark}
Note that the normalizations $N(V_{j,i})$ and $N(V_{j,e})$ are wide open spaces in the sense of \cite{Coleman1982}. Indeed, assuming that $f(z)$ is integral, %
the normalization of the zero set of $f(z)$ is $X$, and our construction of the morphism of semistable models $\mathcal{X}\to \mathcal{Y}$ in \cref{pro:SimultaneousSemistable} shows that the spaces $V_{j,i}$ and $V_{j,e}$ give a semistable open covering of $X$. One can also the results in \cite{ABBR2015} and \cite{Helminck2023} to conclude that these give the star-shaped curves corresponding to the covering.  To work with $p$-adic integrals on these wide open spaces, one should then find a generalization of the methods in \cite{BBRK2010} for hyperelliptic curves. We also refer the reader to %
\cite{KK2022} and \cite{Kaya2022} for more on these techniques. %
\end{remark}

\begin{remark}\label{rem:FurtherFactorization}
Note that every factor $F_{j,i}(z)$ is in fact an element of $\mathcal{O}^{h}_{\mathcal{Y},x_{e}}$.   
Indeed, this follows from the inclusion of decomposition groups in \cref{lem:KummerOrbitRemark}. %
The $F_{j,i}(z)$ thus decompose into irreducible factors $G_{j}(z)$ over $\mathcal{O}^{h}_{\mathcal{Y},x_{e}}$. This gives the relation between the $V_{j,e}$ and $V_{j,i}$ in terms of irreducible factors.     
\end{remark}

We now return to the algorithm outlined in \cref{sec:OutlineAlgorithm}. In this algorithm, we calculate approximations $\alpha_{i,n}$ of the roots of $f(z)$ over (a finite extension of) the Henselizations of various local rings $\mathcal{O}_{\mathcal{Y},x}$. We can now combine these approximations $\alpha_{i,n}$ to obtain approximations of the irreducible factors   %
$F_{j,i}$ and $G_{j}$ as follows. Consider %
the product %
$\prod (z-\alpha_{i,n})$\footnote{Note here that $f(z)$ is assumed to be monic for the sake of simplicity. One can also modify this approach to work for non-monic polynomials. }. By modifying the $\alpha_{i,n}$ slightly so that the set of approximations becomes Galois invariant, this polynomial $\prod (z-\alpha_{i,n})$ becomes an element of $\mathcal{O}^{h}_{\mathcal{Y},x_{e}}[z]$ or $\mathcal{O}^{h}_{\mathcal{Y},x_{i}}[z]$, thus giving the desired approximations. Moreover, from the set-up of our algorithm, it is now clear that we obtain the correspondences between the different $F_{j,i}$ and $G_{j}$ described in \cref{rem:FurtherFactorization}, so that we can glue the wide open spaces $V_{j,i}$ and $V_{j,e}$.    
\begin{example}\label{exa:PlaneQuarticCont}
We use the power series data in \cref{exa:PlaneQuartic} to give the irreducible factors mentioned above. Over $\mathfrak{p}_{1}$, note that $\rho_{1}+\rho_{2}=-1$ and $\rho_{1}\rho_{2}=-v$. We then obtain the following approximations of the irreducible factors of $f(z)$: 
\begin{align*}
\tilde{F}_{1,1}&=(y-\gamma_{1,\mathfrak{p}_{1}})(y-\gamma_{2,\mathfrak{p}_{1}})=y^2+c_{1}y+c_{2},\\
\tilde{F}_{2,1}&=y-\gamma_{3,\mathfrak{p}_{1}}=y+1-vu^4,
\end{align*}
where
\begin{align*}
c_{1}&=1+v^5,\\
c_{2}&=-v-v^2u^4-(v^4u^8/(4v+1)).
\end{align*}
Similarly, over $\mathfrak{p}_{2}$ we have $\beta_{1}+\beta_{2}=-1$ and $\beta_{1}\beta_{2}=-u^{4}$. We then obtain %
the approximations 
\begin{align*}
\tilde{F}_{1,2}&=(y-\gamma_{1,\mathfrak{p}_{2}})(y-\gamma_{2,\mathfrak{p}_{1}})=y^2+d_{1}y+d_{2},\\
\tilde{F}_{2,2}&=z-\gamma_{3,\mathfrak{p}_{1}}=y-v+v^2-(u^4+2)v^3,%
\end{align*}
where \begin{align*}
d_{1}&=2+v,\\
d_{2}&=1+v-u^4v^2.
\end{align*}
The polynomials $\tilde{F}_{1,1}$ and $F_{1,2}$ are reducible over the maximal ideal $\mathfrak{m}=(u,v,p^{1/2})=(u,v)$, and the local gluing maps between the local analytic spaces are given by comparing the different irreducible factors. 
\end{example}

\begin{center}

\bibliographystyle{alpha}
\bibliography{main}{}

\end{center}

\newpage 
\section{Algorithms}\label{sec:Algorithms2} 

In this section, we collect some of the smaller algorithms introduced in the main body of the text. We will use the notation in \cref{sec:MapsSH} and %
\cref{sec:NotationAlgorithms}. 

\subsection{The \'{e}tale routine}

The following function applies the \'{e}tale routine, as described in \cref{sec:EtaleRoutine}. This function can be found in \texttt{RetrieveNewPolynomialsEtale()}. For more background regarding the individual functions (such as \texttt{ScaledTranslation()}), we refer the reader to the documentation surrounding the corresponding functions in \texttt{Code1.jl}.   
\begin{algorithm}[h]
\caption{The \'{e}tale routine}\label{alg:DiscreteNPIteration}\label{alg:EtaleRoutine}
\begin{algorithmic}[1]
\Require Array \texttt{C} with a specified state vector \texttt{C[i]}. Write $f(z)$ for the current polynomial \texttt{C[i][1]}. This is assumed to be defined over a finite subalgebra of %
the strict Henselization $A^{\sh}$ of $A=K[x_{1},...,x_{n}]$  %
at $\mathfrak{m}=(x_{1},...,x_{n})$. Moreover, we assume it splits completely over the fraction field $K(A^{\sh})$.  
\Ensure  Array \texttt{C} where \texttt{C[i]} has been replaced by new state vectors whose current polynomials are $b^{n}_{i}f(z+a_{i}/b_{i})$ for projective roots 
$[a_{i}:b_{i}]$ of the reduction of $\overline{f}$. The new algebras %
are extensions of \texttt{C[i][2]} that contain the new projective roots. The projective roots have been added to \texttt{C[i][7]}.  
\State Compute $f_{0}=\text{red}(f(z))$ using \texttt{ReductionMap}(). %
\State Compute the irreducible factors $g_{i}$ of $f_{0}$ for $i=1,...,r$ %
over the polynomial ring \texttt{C[i][2]} with relations \texttt{C[i][3]} using \texttt{FactorDomain()}. 
\For{$i=1,...,r$}
	\If{$\mathrm{deg}_{z}(g_{i})=1$} 
		\State $g_{i}(z)=b_{i}z+a_{i}$.
		 \State $f_{i}(z)=b_{i}^{n}f(z+a_{i}/b_{i})$. (\texttt{ScaledTranslation()}) 
		\State Create a new state vector \texttt{w} using this data. 
	\Else \Comment{This implies that $\mathrm{deg}(g_{i}(z))>1$}
		\State Construct a polynomial ring with new coefficient $w_{i}$ for the $g_{i}$. (\texttt{AlgebraFromPolynomial()}) 
		\State Construct a new set of relations: union of \texttt{C[i][3]} and $g_{i}(w_{i})$. %
		\State $f_{i}(z)=f(z+w_{i})$. (\texttt{ScaledTranslation()})
		\State Reduce $f_{i}(z)$ using the polynomials from \texttt{C[i][3]} and $g_{i}$. 
		\State Create a new state vector \texttt{w} using this data, with \texttt{w[8]=tropical} as the current state of the algorithm.     
	\EndIf
\EndFor
\State Replace \texttt{C[i]} with the collection of all these state vectors \texttt{w}. %
\State \Return \texttt{C}.   %
\end{algorithmic}
\end{algorithm}

\subsection{The tropical routine}
We now describe the tropical routine from \cref{sec:TropicalRoutine}. Rather than replacing the approximations \texttt{C[i][7]} as in the \'{e}tale routine, this replaces the approximation height \texttt{C[i][6]}. As before, it does provide a new polynomial however. The corresponding function can be found in \texttt{RetrieveNewPolynomialsTropical()}. A similar function for the first initialization is called \texttt{InitialTropicalization()}. This calculates all of the tropical zeros, and not just the positive ones. It is only used when starting the algorithm. We will refrain from describing it here, since it is more or less the same as the algorithm below.     %

\begin{algorithm}[h]
\caption{The tropical routine}\label{alg:NewtonPolytope}\label{alg:TropicalRoutine}
\begin{algorithmic}[1]
\Require An array \texttt{C} with a specific state vector \texttt{C[i]} whose current polynomial \texttt{C[i][1]} splits completely over the strict Henselization of $A=K[x_{1},...,x_{n}]$ 
at $\mathfrak{m}=(x_{1},...,x_{n})$. We write $f(z)$ for this polynomial.    %
\Ensure An array \texttt{C} where the old state vector \texttt{C[i]} has been replaced by a set of new state vectors, one for each positive tropical root $k_{j}>0$ of $f(z)$. The current approximation heights \texttt{C[i][6]} have been increased using the tropical root $k_{j}$. The current polynomials \texttt{C[i][1]}have been scaled so that roots of valuation $k_{j}$ are now of valuation zero in the new current polynomials. %

\State Calculate the tropical roots $k_{j}>0$ of $f(z)$ using \texttt{TropicalRoots()}. Here $j=1,...,r$.
\For{$j=1,...,r$}
	\State Define a new approximation height:  \texttt{C[i][6]}$+k_{j}$.
	\State $g_{j}(z)=f(x_{n}^{k_{j}}z)$. 
	\State $c_{j}=\mathrm{cont}_{x_{n}}(g_{j})$.
	\State $f_{j}(z)=g_{j}/c_{j}$. (\texttt{ValuationScaling})
	\State Create a new state vector \texttt{w} using this data, with \texttt{w[8]=etale}. 
\EndFor
\State Replace \texttt{C[i]} with the collection of all these state vectors \texttt{w}. %
\State \Return \texttt{C}. 
\end{algorithmic}
\end{algorithm}

\end{document}